\documentclass{article}
\usepackage[utf8]{inputenc}
\usepackage{amsmath}
\usepackage{amsfonts}
\usepackage{enumitem}
\usepackage{bbm}
\usepackage[letterpaper,hmargin=1in,vmargin=1.25in]{geometry}
\usepackage{comment}
\usepackage{biblatex}
\usepackage{authblk}
\usepackage{amsthm}
\usepackage{mathtools}
\usepackage [english]{babel}
\usepackage [autostyle, english = american]{csquotes}
\MakeOuterQuote{"}
\DeclarePairedDelimiter{\ceil}{\lceil}{\rceil}
\DeclarePairedDelimiter{\floor}{\lfloor}{\rfloor}
\DeclarePairedDelimiter{\parens}{\lparen}{\rparen}
\DeclarePairedDelimiter{\abracket}{\langle}{\rangle}
\DeclarePairedDelimiter{\norm}{\lVert}{\rVert}
\usepackage{circuitikz}

\makeatletter
\newtheorem*{rep@theorem}{\rep@title}
\newcommand{\newreptheorem}[2]{%
	\newenvironment{rep#1}[1]{%
		\def\rep@title{#2 \ref{##1}}%
		\begin{rep@theorem}}%
		{\end{rep@theorem}}}
\makeatother

\newtheorem{theorem}{Theorem}[section]
\newreptheorem{theorem}{Theorem}

\newtheorem{proposition}[theorem]{Proposition}
\newtheorem{corollary}[theorem]{Corollary}

\newtheorem{definition}[theorem]{Definition}

\newtheorem{lemma}[theorem]{Lemma}

\newcommand{\probp}[1]{{\mathbb{P}}\left({#1}\right)}

\title{Mixing Time of the Overlapping Cycles Shuffle}
\author[ ]{Olena Blumberg}
\author[1]{Ben Morris}
\author[2]{Hans Oberschelp}
\affil[1]{Department of Mathematics, University of California, Davis}
\affil[2]{Department of Mathematics, University of California, Davis}
\date{}

\begin{document}

\maketitle

\begin{abstract}
    In each step of the overlapping cycles shuffle on $n$ cards, a fair coin is flipped which determines whether the $m$th card or the $n$th card is moved to the top of the deck. Angel, Peres, and Wilson showed the following interesting fact: If $m = \floor*{\alpha n}$ where $\alpha$ is rational, then the relaxation time of a single card in the overlapping cycles shuffle is $\theta(n^2)$. However if $\alpha$ is the golden ratio, then the relaxation time of a single card is $\theta(n^\frac{3}{2})$. We show that the mixing time of the entire deck under the overlapping cycles shuffle matches these bounds up to a factor of $\log(n)^3$. That is, the mixing time of the entire deck is $O(n^2 \log(n)^3)$ if $\alpha$ is rational and $O(n^\frac{3}{2} \log(n)^3)$ if $\alpha$ is the golden ratio.
\end{abstract}

\section{Introduction}

The overlapping cycles shuffle was first described by Jonasson \cite{overlappingcyclesoriginal} and takes two parameters, $n \in \mathbb{N}$ and $m \in \{2,\dots,n-1\}$. We define the shuffle as follows: Begin with a deck of $n$ cards. In each round, flip an independent coin. If Heads, move the $m$th card to the top of the deck. If Tails, move the $n$th card to the top of the deck. \\
\\
Despite its simple construction, the overlapping cycles shuffle has interesting and surprising properties. Angel, Peres, and Wilson \cite{overlappingcyclesonecard} determined the spectral gap of the chain which tracks a single card in the overlapping cycles shuffle. Their analysis determined that if $m = \floor{\alpha n}$ for some $\alpha \in (0,1)$ then the asymptotic relaxation time as $n$ grows depends on how well approximated $\alpha$ is by rational numbers. In particular, if $\alpha$ is rational the relaxation time is $O(n^2)$. However the relaxation time can be as short as $O(n^\frac{3}{2})$ which occurs when $\alpha = \frac{\sqrt{5}-1}{2}$, the inverse golden ratio. \\
\\
Angeles, Peres, and Wilson ask if the mixing time of the entire deck is within a factor of $\log(n)$ of their result for individual cards. We prove something close: The mixing time is $O(n^2\log^3(n))$ for rational $\alpha$ and $O(n^\frac{3}{2}\log^3(n))$ for``very'' irrational $\alpha$ like the inverse golden ratio.\\
\\
The overlapping cycles shuffle has a simple description as a random walk on the symmetric group $S_n$. In each round $g$ is equally likely to go to $(1,2,\dots,m) g$ or $(1,2,\dots,n) g$. \\
\\
This explains where the name ``overlapping cycles shuffle'' comes from. Note that if $m$ and $n$ are both odd, then $(1,\dots,m)$ and $(1,\dots,n)$ will both be even permutations. Thus, the mixing time we seek to bound will be in respect to convergence to a distribution which is uniform across $A_n$, not $S_n$. If $m$ and $n$ are both even, then $(1,\dots,m)$ and $(1,\dots,n)$ will both be odd permutations, and the shuffle will be periodic. In this case we say the mixing time is the value $t$ such that if $r > t$ then the distribution after $r$ rounds of the shuffle is approximately uniform over $A_n$ if $r$ is even and $S_n \backslash A_n$ if $r$ is odd. If $(1,\dots,m)$ and $(1,\dots,n)$ have different parity then there are no issues and we consider mixing time in the typical sense. \\
\\
The $m = n-1$ case of the overlapping cycles shuffle was shown to have a mixing time of $O(n^3 \log(n))$ by Hildebrand \cite{mis2} in his dissertation. This case was studied before the overlapping cycles shuffle was defined in its general form and given its name. The $m = n-1$ case is especially slow because the cycles $(1,\dots,n-1)$ and $(1,\dots,n)$ act identically on every element of $\{1,\dots,n\}$ except for $n-1$ and $n$. We will only consider values of $m$ such that $\frac{m}{n}$ is bounded away from $0$ and $1$.

\section{Main Theorem}

Our main goal is to establish an upper bound on the mixing time of the overlapping cycles shuffle. The following theorem provides an upper bound on the mixing time of the overlapping cycles shuffle on $n$ cards with parameter $m$, which is tight up to a factor of $\log^3(n)$ provided that $\frac{m}{n}$ is bounded away from $0$ and $1$.

\begin{theorem}\label{mainoverlapping}
	Consider the overlapping cycles shuffle on $n$ cards with parameter $m$, where $\epsilon < \frac{m}{n} < 1-\epsilon$ for some $\epsilon \in (0,\frac{1}{2})$. Let $\ell_{\max}$ be defined by
	\begin{equation*}
		\ell_{\max} = \max\limits_{\omega \in \{1,\dots,2n-m\}} \Big\{ \min \big\{ |a| + |b|\sqrt{n} \ : \ a,b \in \mathbb{Z}, \ \omega \equiv a + bm \mod 2n-m+1 \big\} \Big\}.
	\end{equation*}
	Then the mixing time is at most
	\begin{equation*}
		\mathcal{A}\ell_{\max}^2 \log^3(n)
	\end{equation*}
	where $\mathcal{A}$ is a constant which depends on $\epsilon$.
\end{theorem}

\noindent
Note that $\ell_{\max}$ is always bounded above by $2n$ because $(a,b) = (\omega,0)$ is always an element of the inner set, and $|\omega| + |0|\sqrt{n} = \omega \leq 2n-m$. So the mixing time is at most a constant times $n^2 \log^3(n)$. \\
\\
Now fix some $\alpha \in (0,1)$ and for any deck size $n$ consider the shuffle where $m = \floor{\alpha n}$. We are interested in the asymptotic mixing time as $n$ approaches infinity. It turns out that if $\alpha$ is rational then we do no better than the universal upper bound and get a mixing time of $O(n^2 \log^3(n))$. However, if $\alpha$ is an irrational number whose multiples form a low-discrepancy sequence, then the mixing time is $O(n^\frac{3}{2} \log^3(n))$. We show this in particular for $\alpha = \frac{\sqrt{5}-1}{2}$, the inverse golden ratio, at the very end of Section \ref{entropysection}.

\section{Movement of a Single Card: Intuition and Notation}

In our analysis of the overlapping cycles shuffle we will make heavy use of sequences of coins. We defined the overlapping cycles shuffle in terms of a sequence of coins, and we will later imagine that in each step the shuffle draws from different pools of pre-flipped coins depending on its state. To start we define the following:
\begin{definition}
	If $c = (c_1,c_2,\dots,c_t)$ is a sequence of independent uniform $\{\text{Heads},\text{Tails}\}$-valued random variables, we say it is a sequence of coins of length $t$.
\end{definition}
\begin{definition}
	For a sequence of coins $S = (c_1,c_2,\dots,c_t)$ and $r \leq t$, define the Heads-Tails differential $\text{Diff}_r(S)$ to be the number of Heads in $(c_1,\dots,c_r)$ minus the number of Tails in $(c_1,\dots,c_r)$.
\end{definition}
\noindent
Note that $\text{Diff}_r(c)$ is a simple symmetric random walk. We will use this fact later to show with probability bounded away from 0 that $\text{Diff}_r(c)$ stays within constant standard deviations of $0$. \\
\\
To understand the overlapping cycles shuffle it is important to understand how each step affects cards in different parts of the deck. Note that if a card is in one of the top $m - 1$ positions in the deck, then it will move down one position in the round regardless of if Heads or Tails is flipped. On the other hand, a card in a position between $m + 1$ and $n$ is equally to stay put or move down one position, with the $n$th card ``wrapping around'' to the top of the deck if it moves ``down''. Lastly, the card in position $m$ is equally likely to either move to the top of the deck or move down one position. \\
[\baselineskip]
Since the behavior of cards in positions $1$ through $m-1$ varies from cards in positions $m+1$ through $n$, we should have language to quickly distinguish the two.
\begin{definition}
	If a card is in a position between $1$ and $m-1$ we say the card is in the \textbf{top part of the deck}. If a card is in a position between $m+1$ and $n$ we say the card is in the \textbf{bottom part of the deck}.
\end{definition}
Note that cards in the bottom part of the deck move, on average, half as quickly as cards in the top part of the deck. This is because cards in the top part of the deck move down one position each step, while cards in the bottom part of the deck move down only if Tails is flipped, which happens half the time. So if $x$ is a position in the top part of the deck, and $y$ is a position in the bottom part of the deck, we should think of positions $x$ and $x+1$ as being distance $1$ from each other, and positions $y$ and $y+1$ as being distance $2$ from each other. To quickly reference this notion, we define the following function:
\begin{definition}
	If $x \in \{1,\dots,n\}$ is a position in the deck, let 
	\begin{equation}
		p(x) = \begin{cases}
			x \ {\rm if} \ x \leq m \\
			x + (x-m) \ {\rm if} \ x > m
		\end{cases}
	\end{equation}
\end{definition}
We will be interested in finding the likelyhood of certain cards being in certain ``nearby positions'' after specific numbers of rounds. To do this we need to reevaluate our notion of distance away from the naive definition of physical distance in the deck. To see why, note that the card in position $m$ has a $\frac{1}{4}$ chance of being adjacent to the card in position $n$ after two steps. Thus, it makes sense to consider position $m$ and position $n$ as ``close''. \\
\\
We will name cards after their initial position in the deck. So card $1$ is the card initially on the top of the deck, card $2$ is the card initially second to top, etc.
\begin{definition}
	We use $i_t$ to denote the position of card i (i.e. the card that was originally in position i) after $t$ steps of the shuffle.
\end{definition}
To determine where card $i$ is after $t$ steps it is enough to know the sequence of coin flips $c_1, c_2, \dots c_t$ (where each $c_r \in \{\text{Heads},\text{Tails}\}$). However this is much more information than we need. For example, if $i << m$, then $i$ will deterministicly move downwards for many steps regardless of the early values of $(c_r)$. The following proposition allows us to compute the position of $i$ only using the coins that actually influence the movement of $i$.

\begin{proposition}\label{cardmovement}
	Let $i$ be a card in a deck of $n$ cards (where as previously mentioned, $i$ begins in position $i$). Suppose $(c_r)$ is a sequence of coins. Let $i_t$ be the position of card $i$ after doing $t$ steps of the Overlapping Cycle Shuffle. Let $H_B$ be the number of times in the first $t$ steps a Heads is drawn while $i$ is in position $m$. Let $T_B$ be the number of times in the first $t$ steps a Heads is drawn while $i$ is in position $m$. Let $H_S$ be the number of ``Heads'' drawn in the first $t$ steps while $i$ is in the bottom part of the deck. Let $T_S$ be the number of ``Tails'' drawn in the first $t$ steps while $i$ is in the bottom part of the deck. Then,
	\begin{equation*}
		p(i_t) \equiv p(i) + t + (T_S - H_S) + (T_B - m H_B) \text{ mod } (2n-m+1).
	\end{equation*}
\end{proposition}

\begin{proof}
	We proceed by induction. If $t = 0$ then the statement is trivially true. For the inductive step, it suffices to show that, if $i_{t-1}$ is the position of $i$ after $t-1$ steps, then $i$'s distance after step $t$ depends on the coin or lack of coin used in step $t$ exactly as the formula dictates. Specifically, we need to show that
	\begin{align*}
		p(i_t) \equiv &p(i_{t-1}) + 1 + \mathbbm{1}(\text{$i_{t-1} \geq m$ and $c_t = T$}) - \mathbbm{1}(\text{$i_{t-1} > m$ and $c_t = H$}) \\
		& \phantom{aaaaaaaaa} - m \cdot \mathbbm{1}(\text{$i_{t-1} = m$ and $c_t = H$}).
	\end{align*}
	
	\begin{itemize}
		\item If $i$ is positioned in the top part of the deck it must move down in the next step so $p(i_t) = p(i_{t-1})+1$. Similarly, if $i$ is in position $m$ and Tails is flipped then $i_{t-1} = m$ and $i_t = m+1$ and so $p(i_{t}) = m + 1 + 1 = p(i_{t-1}) + 1 + 1$.
		\item If $i$ is in the bottom part of the deck, but not in position $n$, and flips a Tails then $i_{t} = i_{t-1} + 1$ and so $p(i_{t}) = p(i_{t-1}) + 2 = p(i_{t-1}) + 1 + 1$. If $i$ is in position $n$ and flips a Tails, then it moves to position $1$, so $p(i_{t-1}) = 2n-m$ and $p(i_t) = 1$ and therefore $p(i_t) = p(i_{t-1}) + 1 + 1$ mod $(2n-m+1)$.
		\item If $i$ is in the bottom part of the deck and flips a Heads then $i_{t} = i_{t-1}$ and so $p(i_{t}) = p(i_{t-1}) + 1 - 1$.
		\item If $i$ is in position $m$ and flips Heads, then $i_{t-1} = m$ and $i_t = 1$. So $p(i_{t}) = p(i_{t-1}) + 1 - m$.
	\end{itemize}
\end{proof}
\noindent
We use $H_B$ and $T_B$ to denote movement in position $m$ because the choice of $i$ going to $m+1$ or $1$ is a ``big'' choice. We use $H_S$ and $T_S$ to denote movement in the bottom part of the deck, because the choice of $i$ moving down one position or staying in place has a ``small'' impact on its movement. From now on we will use the letter B for ``big coins'' which are drawn when certain cards are in position $m$ and the letter S for ``small coins'' which are drawn when certain cards are in the bottom part of the deck. \\
\\
As a consequence of Proposition \ref{cardmovement}, we have the following corollary:

\begin{corollary}\label{cardmovementcompare}
	Let $i$ and $j$ be two cards in a deck of $n$ cards. Suppose $(c_r)$ is a sequence of coins. Let $i_t$ be the position of card $i$ after doing $t$ steps of the Overlapping Cycle Shuffle drawing from $(c_r)$ as described previously. Let $H_B(i)$ and $H_S(i)$ be the number of Heads drawn from $(c_r)$ when $i$ is in position $m$ and in the bottom part of the deck, respectively, in the first $t$ steps. Let $T_B(i)$ and $T_S(i)$ be the number of Tails drawn from $(c_r)$ when $i$ is in position $m$ and in the bottom part of the deck respectively in the first $t$ steps. Let $j_t,H_B(j),H_S(j),T_B(j),T_S(j)$ be defined similarly for card $j$ after $t$ steps drawing from $(c_r)$. Then,
	\begin{align*}
		p(j_t) - p(i_t) &\equiv p(j_0) - p(i_0) + (T_S(j) - H_S(j)) - (T_S(i) - H_S(i)) \\
		&\hspace{3mm} +  (T_B(j) - T_B(i)) - m(H_B(j) - H_B(i)) \text{ mod } (2n-m+1).
	\end{align*}
\end{corollary}

This corollary helps inform us as to how we should consider the ``closeness'' of cards. Assume that $\frac{n}{10} < m < \frac{9n}{10}$. Note that after many steps, the magnitudes of $T_S(j) - H_S(j)$ and $T_S(i) - H_S(i)$ will each likely be much more than the magnitude of $H_B(i) - H_B(j)$. This is because whenever $i$ is in the bottom part of the deck the coins flipped add to either $T_S(i)$ or $H_S(i)$ and whenever $j$ is in the bottom part of the deck the coins flipped add to either $T_S(j)$ or $H_S(j)$. However coins only add to $H_B(i),T_B(i)$ or $H_B(j),T_B(j)$ when $i$ or $j$ is exactly in position $m$. We expect $i$ and $j$ to spend on the order of $n$ times more time in the bottom part of the deck than exactly in position $m$, so we should expect $|(T_S(i) - H_S(i))|$ and $|(T_S(j) - H_S(j))|$ to be on the order of $\sqrt{n}|H_B(i) - H_B(j)|$ and $\sqrt{n}|T_B(i)-T_B(j)|$. With this intuition in mind, we define the following metric:
\begin{definition}
	Let $\omega \in \mathbb{R}$. Then we define
	\begin{equation*}
		\norm{\omega} = \min \{ |a| + |b|\sqrt{n} \ : \ a \in \mathbb{R}, \ b \in \mathbb{Z}, \ \omega \equiv a + b m \text{ mod } (2n-m+1) \}.
	\end{equation*}
	In particular if $i$ is a position in the deck then
	\begin{equation*}
		\norm{p(i)} = \min \{ |a| + |b|\sqrt{n} \ : \ a,b \in \mathbb{Z}, \ p(i) \equiv a + b m \text{ mod } (2n-m+1) \}
	\end{equation*}
	and if $j$ and $k$ are positions in the deck then we have the distance
	\begin{equation*}
		\norm{p(k)-p(j)} = \min \{ |a| + |b|\sqrt{n} \ : \ p(k) \equiv p(j) + a + b m \text{ mod } (2n-m+1) \}.
	\end{equation*}
\end{definition}

It will be important to know how far apart cards can possibly be from each other under this metric. For this purpose we define the following constant.

\begin{definition}
	Let $\ell_{\max}$ be defined by
	\begin{equation}
		\ell_{\max} = \max\limits_{\omega \in \{1,\dots,2n-m\}} \norm{\omega}.
	\end{equation}
\end{definition}

We can show that $\ell_{\max} \geq \frac{1}{2}n^\frac{3}{4}$. To see this, note that if we choose $|a|<\ell$ and $|b|\sqrt{n} < \ell$ then we have $2\ell$ choices for $a$ and $\frac{2\ell}{\sqrt{n}}$ choices for $b$ and so there can be at most $\frac{4\ell^2}{\sqrt{n}}$ distinct cards with norms less than $\ell$. By the pigeonhole principle each of the elements of $\{1,\dots,2n-m\}$ are associated with exactly one duple $(a,b)$. In order to account for all $2n-m$ elements we need $\frac{4\ell^2}{\sqrt{n}} > 2n-m > n$ which translates to $\ell \geq \frac{1}{2}n^\frac{3}{4}$. \\
\\
It will also be useful to consider a more traditional ``one dimensional'' distance between positions in the deck, but allowing for ``wrapping around'' so that positions $n$ and $1$ are considered close. We define this as follows.
\begin{definition}\label{mdistance}
	Let $j$ and $k$ be positions in the deck. Then we define the distance
	\begin{equation*}
		|p(k)-p(j)|_M = \min\{|a| : p(k) \equiv p(j) + a \mod (2n-m+1)\}.
	\end{equation*}
\end{definition}
\noindent
It will be useful to analyze the distribution given by the inverse permutation of $t$ steps of the overlapping cycles shuffle. It turns out that this ``inverse overlapping cycles shuffle'' is just the overlapping cycles shuffle in disguise.
\begin{theorem}\label{inverseshuffle}
	Let $\pi_t$ be the random permutation that is $t$ steps of the overlapping cycles shuffle on $n$ cards with parameter $m$. Then,
	\begin{equation*}
		\pi_t^{-1} \stackrel{d}{=} \sigma \pi_t \sigma^{-1}
	\end{equation*}
	where
	\begin{equation}
		\sigma =
		\begin{pmatrix}
			1 & 2 & \dots & m & m+1 & m+2 & \dots & n \\
			m & m-1 & \dots & 1 & n & n-1 & \dots & m+1
		\end{pmatrix}.
	\end{equation}
\end{theorem}
\noindent
In other words, the inverse overlapping cycles shuffle is also an overlapping cycles shuffle after reordering the cards.

\begin{proof}
	Note that under $\pi_1^{-1}$ any $g$ is equally likely to go to $(m,m-1,\dots,1)g$ or $(n,n-1,\dots,1)g$. Also note that
	\begin{equation}
		\sigma(1,2,\dots,m)\sigma^{-1} = (m,m-1,m\dots,1)
	\end{equation}
	and
	\begin{align}
		\sigma(1,2,\dots,n)\sigma^{-1} &= \sigma(1,2,\dots,m,m+1,m+2,\dots,n)\sigma^{-1} \\
		&= (m,m-1,\dots,1,n,n-1,\dots,m+1) \\
		&= (n,n-1,\dots,1).
	\end{align}
	Since each step of the overlapping cycles shuffle is the same as the inverse up to the given re-ordering of the cards, we know that the same is true for $t$ steps with the same reordering.
\end{proof}

One useful property of the $| \cdot |_M$ distance defined in Definition \ref{mdistance} is that it is invariant under the reordering $\sigma$. So cards that are considered close with respect to $| \cdot |_M$ are also close when considering the inverse overlapping cycles shuffle.

\begin{proposition}\label{inversedist}
	Let $j$ and $k$ be positions in the deck. Let $\sigma$ be the permutation from Theorem \ref{inverseshuffle}. Then,
	\begin{equation*}
		p(k)-p(j) = p(\sigma(j))-p(\sigma(k)) \mod 2n-m+1.
	\end{equation*}
\end{proposition}

\begin{proof}
	We consider three cases.
	\begin{enumerate}
		\item $j,k \leq m$ \\
		In this case, $\sigma(j) = m+1-j$ and $\sigma(k) = m+1-k$. So,
		\begin{align}
			p(\sigma(k))-p(\sigma(j)) &= p(m+1-k)-p(m+1-j) \\
			&= (m+1-k)-(m+1-j) \\
			&= j-k \\
			&= p(j)-p(k).
		\end{align}
		\item $j,k > m$ \\
		In this case, $\sigma(j) = n+m+1-j$ and $\sigma(k) = n+m+1-k$. So,
		\begin{align}
			p(\sigma(k))-p(\sigma(j)) &= p(n+m+1-k)-p(n+m+1-j) \\
			&= (2n+m+2-2k)-(2n+m+2-2j) \\
			&= 2j-2k \\
			&= (2j-m) - (2k-m) \\
			&= p(j)-p(k).
		\end{align}
		\item $j \leq m < k$
		In this case, $\sigma(j) = n+m+1-j$ and $\sigma(k) = n+m+1-k$. So,
		\begin{align}
			p(\sigma(k))-p(\sigma(j)) &= p(n+m+1-k)-p(m+1-j) \\
			&= (2n+m+2-2k)-(m+1-j) \\
			&= j-(2k-2n-1) \\
			&\equiv (j)-(2k-m) \mod 2n-m+1\\
			&= p(j)-p(k).
		\end{align}
	\end{enumerate}
\end{proof}

\begin{corollary}\label{inversenorm}
	Let $j$ and $k$ be positions in the deck. Let $\sigma$ be the permutation from Theorem \ref{inverseshuffle}. Then,
	\begin{equation*}
		|p(k)-p(j)|_M = |p(\sigma(k))-p(\sigma(j))|_M
	\end{equation*}
	and
	\begin{equation*}
		\norm{p(k)-p(j)} = \norm{p(\sigma(k))-p(\sigma(j))}.
	\end{equation*}
\end{corollary}

\begin{proof}
	By Proposition \ref{inversedist} we know that
	\begin{equation}
		|p(k)-p(j)|_M = |p(\sigma(j))-p(\sigma(k))|_M.
	\end{equation}
	Since $|\cdot|$ is symmetric we get
	\begin{equation}
		|p(k)-p(j)|_M = |p(\sigma(k))-p(\sigma(j))|_M.
	\end{equation}
	Since $\norm{\cdot}$ is a function of $|\cdot|_M$ it follows that
	\begin{equation}
		\norm{p(k)-p(j)} = \norm{p(\sigma(k))-p(\sigma(j))}.
	\end{equation}
\end{proof}
This next proposition is the main result of this section. It allows us to predict where a card $i$ will be after $t$ steps of the shuffle using the Heads-Tails differentials of $i$ from the times it is in position $m$ and in the bottom part of the deck.

\begin{proposition}\label{movementbounds}
	Let $i$ be a card in the deck. Let $B$ be the record of flips when $i$ is in position $m$. Let the sequence $S$ be the record of flips when $i$ is in the bottom part of the deck. Let $x$ be the Heads-Tails differential of $S$ after $t$ steps of the shuffle. Let $y$ be the Heads-Tails differential of $B$ after $t$ steps of the shuffle. Then,
	\begin{align}
		\norm*{p(i_t)-p(i)-\left(t-\floor*{\frac{t}{2n}}(m-1)-x-\floor*{y\parens*{1-\frac{m}{2n}}}m \right)} \leq \left|\frac{y}{2}\right|+\left|\frac{x}{2n}\right|(\sqrt{n}+1) + 4\sqrt{n}.
	\end{align}
\end{proposition}

\begin{proof}
	By Theorem \ref{cardmovement} we know
	\begin{equation}
		p(i_t) \equiv p(i) + t + (T_S - H_S) + (T_B - m H_B) \text{ mod } (2n-m+1). \label{locbydiff}
	\end{equation}
	Note that
	\begin{equation}
		H_B = \frac{1}{2}\parens*{\text{number of times $i$ is in position $m$ in the first $t$ steps}} +\frac{1}{2}y.
	\end{equation}
	Imagine card $i$ has just reached position $m$ for the $T$th time. We are interested in how many steps  it takes for $i$ return to position $m$ again (for the $(T+1)$th time). If the next flip is Heads, then $i$ will move from position $m$ to position $1$. It will then take $m-1$ more steps for $i$ to return to position $m$. This is a total of $m$ steps, and we call this a short return. If instead the next flip is Tails, then $i$ will move to position $m+1$. It will then take $2(n-m)+\Delta_T$ steps for $i$ to cycle through the bottom of the deck back to position $1$. $\Delta_T$ represents the deviation from the expected number of steps, and $\Delta_T$ contributes negatively to the Heads-Tails differential of $S$. Namely we have $x = -\Delta_1 - \Delta_2 - \dots$. After card $i$ reaches position $1$ it will take $m-1$ more steps to reach position $m$. This is a total of $2n-m+\Delta_T$ steps and we call this a long return. Note that, ignoring $\Delta_T$, the average of the number of steps between the short return and long return is $n$ steps. Thus, if $i_0 = m$ and $y$ is nonnegative then the number of subsequent times $i$ hits position $m$ is
	\begin{equation}
		\floor*{y + \frac{t-my-x}{n}}.
	\end{equation}
	This is because other than the excess $y$ short returns which take $m$ steps, the remaining visits are divided evenly between long and short returns, so $n$ steps before counting the increased/decreased speed through the bottom part of the deck as recorded by $x$. Similarly if $y$ is negative then the number of subsequent times $i$ hits position $m$ in $t$ steps is
	\begin{equation}
		\floor*{(-y) + \frac{t-(2n-m)(-y)-x}{n}}.
	\end{equation}
	This means that
	\begin{align}
		H_B &= \frac{y}{2} + \frac{t}{2n} - \frac{my}{2n} - \frac{x}{2n} + \zeta + \frac{y}{2} \\
		&= \frac{t}{2n} + y\parens*{1 - \frac{m}{2n}} - \frac{x}{2n} + \zeta \label{ybpositive}
	\end{align}
	if $y$ is positive and
	\begin{align}
		H_B &= \frac{-y}{2} + \frac{t}{2n} - \frac{(2n-m)(-y)}{2n} - \frac{x}{2n} + \zeta + \frac{y}{2} \\
		&= \frac{t}{2n} + y\parens*{1 - \frac{m}{2n}} - \frac{x}{2n} + \zeta \label{ybnegative}
	\end{align}
	if $y$ is negative, where $\zeta \in [-\frac{1}{2},0]$ and rounds down so that $H_B$ is an integer. Note that both lines (\refeq{ybpositive}) and (\refeq{ybnegative}) are equal, so we have the same value regardless of if $y$ is positive or negative. These values represent the case where $i_0 = m$ and we don't count the fact that $i$ starts at $m$ as an $m$ ``hit''. At the other extreme, where $i$ starts at $m$ and we do count that as a ``hit'' we would have
	\begin{align}
		H_B = \frac{t}{2n} + y\parens*{1 - \frac{m}{2n}} - \frac{x}{2n} + \zeta + 1.
	\end{align}
	In the more general case where we have $i_0 \neq m$ and we count the number of times $i$ hits $m$ we get
	\begin{align}
		H_B = \frac{t}{2n} + y\parens*{1 - \frac{m}{2n}} - \frac{x}{2n} + \upsilon
	\end{align}
	where $\upsilon \in [-\frac{1}{2},1]$. We can use the fact that
	\begin{equation}
		T_B = \frac{1}{2}\parens*{\text{number of times $i$ is in position $m$ in the first $t$ steps}} - \frac{1}{2}y
	\end{equation}
	to similarly calculate
	\begin{align}
		T_B = \frac{t}{2n} - y\parens*{\frac{m}{2n}} - \frac{x}{2n} + \upsilon'
	\end{align}
	where $\upsilon' \in [-\frac{1}{2},1]$.  Plugging into (\refeq{locbydiff}) we get
	\begin{equation}
		p(i_t) \equiv p(i) + t - x + \left(\frac{t}{2n} - y\parens*{\frac{m}{2n}} - \frac{x}{2n} + \upsilon'\right) - m \left[\frac{t}{2n} + y\parens*{1 - \frac{m}{2n}} - \frac{x}{2n} + \upsilon\right]
	\end{equation}
	where the terms in square brackets form an integer. Note that for any integer $z$ we have that $z = \floor{z}+\delta$ where $\delta \in [0,1)$. So,
	\begin{equation}
		p(i_t) \equiv p(i) + t - x + \floor*{\frac{t}{2n}}(1-m) - \floor*{y\parens*{1-\frac{m}{2n}}}m + \left( \delta_1 - y\parens*{\frac{m}{2n}} - \frac{x}{2n} + \upsilon'\right) - m \left[\delta_2 + \delta_3 - \frac{x}{2n} + \upsilon\right]
	\end{equation}
	where $\delta_1,\delta_2,\delta_3 \in [0,1)$. This gives us
	\begin{align}
		&\norm*{p(i_t)-p(i)-t+\floor*{\frac{t}{2n}}(m-1)+x+\floor*{y\parens*{1-\frac{m}{2n}}}m} \\
		\leq &\norm*{\delta_1-y\left(\frac{m}{2n}\right)-\frac{x}{2n}+\upsilon' - m\left[\delta_2 + \delta_3 - \frac{x}{2n} + \upsilon \right]}\\
		\leq &\left|\delta_1-y\left(\frac{m}{2n}\right)-\frac{x}{2n}+\upsilon'\right| + \sqrt{n}\left|\delta_2 + \delta_3 - \frac{x}{2n} + \upsilon\right|.
	\end{align}
	Utilize the fact that $|\delta_1|,|\delta_2|,|\delta_3|,|\upsilon|,|\upsilon'|\leq 1$ and $m < n$ we get
	\begin{align}
		\norm*{p(i_t)-p(i)-t+\floor*{\frac{t}{2n}}(m-1)+x+\floor*{y\parens*{1-\frac{m}{2n}}}m} \leq \left|\frac{y}{2}\right|+\left|\frac{x}{2n}\right|(\sqrt{n}+1) + 4\sqrt{n}.
	\end{align}
\end{proof}

The norm on the left hand side of the inequality of Proposition $\ref{movementbounds}$ is quite complicated. To make our lives easier, it will be nice to consider values of $t$ such that $-t+\floor*{\frac{t}{2n}}(m-1) \equiv 0 \mod 2n-m+1$ thus eliminating those two terms from the norm. To make sure this is possible we have the following lemma.
\begin{lemma}
	Fix some $n,m \in \mathbb{N}$ such that $m < n$. Choose any $s \in \mathbb{N}$. Then there exists $t \in [s,s+4n]$ such that
	\begin{equation}
		-t + \floor*{\frac{t}{2n}}(m-1) \equiv 0 \mod 2n-m+1
	\end{equation}
\end{lemma}

\begin{proof}
	Let $s^* \in [s,s+2n)$ such that $s^*$ is a multiple of $2n$. Then for all $t^*$ in the interval $[s^*,s^*+2n)$ we have
	\begin{equation}
		\left(-(t^*+1) + \floor*{\frac{(t^*+1)}{2n}}(m-1)\right) = \left(-t^* + \floor*{\frac{t^*+1}{2n}}(m-1)\right) - 1
	\end{equation}
	Since there are $2n$ terms of the interval and it will take at most $2n-m+1$ unit steps to get to $0 \mod 2n-m+1$ we know there exists $t \in [s^*,s^*+2n)$ such that
	\begin{equation}
		-t + \floor*{\frac{t}{2n}}(m-1) \equiv 0 \mod 2n-m+1
	\end{equation}
\end{proof}

Later we will have equations that have a $t + \floor*{\frac{t}{2n}}m$ term we want to eliminate. By a similar proof, we have the following lemma as well.

\begin{lemma}
	Fix some $n,m \in \mathbb{N}$ such that $m < n$. Choose any $s \in \mathbb{N}$. Then there exists $t \in [s,s+4n]$ such that
	\begin{equation}
		t - \floor*{\frac{t}{2n}}m \equiv 0 \mod 2n-m+1
	\end{equation}
\end{lemma}

\section{Entropy and 3-Monte}
We will use techniques involving entropy to bound the mixing time of the overlapping cycles shuffle. In this section, we provide the necessary background in entropy. We utilize a new technique involving entropy, the 3-Monte, first described by Senda \cite{monte}, which is a generalization of a similar technique first described by Morris \cite{morris2008improved}.

\begin{definition}
	If $\pi$ is a random permutation in $S_n$, we define the relative entropy of $\pi$ with respect to the uniform distribution as
	\begin{equation*}
		{\rm ENT}(\pi) = \sum\limits_{\varphi \in S_n} \probp{\pi = \varphi} \log(n! \cdot \probp{\pi = \varphi}).
	\end{equation*}
\end{definition}
Note that in the case that $\pi$ is uniform, we have ENT$(\pi) = 0$. In the other extreme where $\pi$ is deterministic, we have ENT$(\pi) = \log(n!)$. \\
\\
This notion of relative entropy is useful, because we have by Pinsker's inequality \cite{pinsker} (Section 2.4, page 88) that
\begin{equation}
	\sqrt{\frac{1}{2}{\rm ENT}(\pi)} \geq \norm{\pi - \xi}_{\rm TV} \label{tventbound}
\end{equation}
where $\xi$ is the uniform random permutation in $S_n$. We now define the notion of 3-Monte shuffles and collisions, which will be useful for finding bounds on relative entropy.

\begin{definition}
	We say a random permutation $\mu$ in $S_n$ is a \textit{3-collision} if for some distinct $x,y,z \in \{1,\dots,n\}$ it is equally likely to be either the 3-cycle $(x,y,z)$ or the identity. So it has the distribution
	\begin{equation*}
		\probp{\mu = (x,y,z)} = \probp{\mu = {\rm id}} = \frac{1}{2}
	\end{equation*}
	In particular we say $\mu = c(x,y,z)$ is the 3-collision which has a one half chance of being the $(x,y,z)$ 3-cycle.
\end{definition}
We will be interested in random permutations which can be written as products containing 3-collisions.
\begin{definition}
	We say a random permutation $\pi$ in $S_n$ is 3-Monte if it has the form
	\begin{equation*}
		\pi = \nu c(x_k,y_k,z_k) \dots c(x_1,y_1,z_1)
	\end{equation*}
	where $\nu$ is a random permutation and $x_1,x_2,x_3,\dots,x_k,y_k,z_k$ and $k$ itself may be dependent on $\nu$, but conditional on $\nu$ the outcomes of $c(x_1,y_1,z_1),\dots,c(x_k,y_k,z_k)$ must be independent.
\end{definition}
To be clear, any random permutation is technically 3-Monte, as $k$ could be trivially set to $0$ conditioned on any $\nu$. Additionally, the same random permutation could be defined with different choices for the 3-collision. We will refer to random permutations as 3-Monte only after explicitly choosing 3-collisions and defining the permutation as a product involving those collisions. The following theorem regarding 3-Monte shuffles will allow us to bound entropy decay of the overlapping cycles shuffle.

\begin{theorem} 
	\label{3monte} \cite{monte} (Chapter 4, page 16)
	Let $\pi$ be a 3-Monte shuffle on $n$ cards. Fix an integer $t > 0$ and suppose that $T$ is a random variable taking values in $\{1, \cdots, t\}$, which is independent of the shuffles $\{\pi_i : i \geq 0\}$. Consider a card $x$. If $x$ is involved in a $3$-collision after time $T$ up to and including time $t$, then consider the first $3$-collision it is involved in after time $T$; say that $x, y, z$ collide in that order. If that 3-collision is also the first 3-collision after time $T$ that $y$ is involved in and it is the first 3-collision that $z$ is involved in, then we say that \textbf{that collision matches} $x$ (with $y$ and $z$) and define $y$ to be the \textbf{front match} of $x$, written as $m_1(x) = y$, and $z$ to be the \textbf{back match} of $x$, written $m_2(x) = z$. If $x$ is in no such collision, define $m_1(x) = m_2(x) = x$.

	Suppose that for every card $i$ there is a constant $A_i \in [0, 1]$ such that $\mathbb P(m_2(i)=j, m_1(i) < i) \geq \frac{A_i}{i}$ for each $j \in \{1, \dots, i-1\}$. (This also means that with probability at least $A_i$, $m_1(i)< i$ and $m_2(i)< i$. Note that it cannot be the case that exactly two of $i$, $m_1(i)$, and $m_2(i)$ are equal; the three are either all the same or all different.) Let $\mu$ be an arbitrary random permutation that is independent of $\{\pi_i: i \geq 0\}$. Then
	
	\begin{align}
		\mathbb{E}\big[\text{\rm ENT}(\pi_t\mu|\text{\rm sgn}(\pi_t\mu))\big] - \mathbb{E}\big[\text{\rm ENT}((\mu|\text{\rm sgn}(\mu)))\big] \leq \frac{-C}{\log (n)} \sum_{x=3}^{n}A_x E_x,
	\end{align}
	where $E_k = \mathbb{E}\big[\text{\rm ENT}(\mu^{-1}(k) \ | \ \mu^{-1}(k+1),\mu^{-1}(k+2),\dots,\mu^{-1}(n), \text{\rm sgn}(\mu))\big]$ and $C$ is a positive universal constant. 
\end{theorem}

The exact use of this theorem will be made apparent in Section \ref{entropysection}. For now just know that our immanent goal is to bound the values $A_x$ from below. \\
\\
We can write two steps of the overlapping cycles shuffle in 3-Monte form as follows: Let $\pi$ be the random permutation corresponding to two steps of the overlapping cycles shuffle. Then $\pi$ has the following distribution:
\begin{align}
	\mathbb{P}(\pi = (1,\dots,m) (1,\dots,m)) &= \frac{1}{4} \\
	\mathbb{P}(\pi = (1,\dots,n) (1,\dots,m)) &= \frac{1}{4} \\
	\mathbb{P}(\pi = (1,\dots,m) (1,\dots,n)) &= \frac{1}{4} \\
	\mathbb{P}(\pi = (1,\dots,n) (1,\dots,n)) &= \frac{1}{4}
\end{align}

Note that $(1,\dots,n) (1,\dots,m) = (1,\dots,m) (1,\dots,n) (m-1,m,n)$. Thus, we can rewrite the distribution of $\pi$ in the following way:
\begin{align}
	\mathbb{P}(\pi = (1,\dots,m)^2) &= \frac{1}{4} \\
	\mathbb{P}(\pi = (1,\dots,n)^2) &= \frac{1}{4} \\
	\mathbb{P}(\pi =  (1,\dots,m) (1,\dots,n) c(m-1,m,n)) &= \frac{1}{2}
\end{align}

If we are being precise, this is not technically a definition of collisions for the overlapping cycles shuffle. We have defined collisions for the shuffle which is two steps at a time of the overlapping cycles shuffle. It will be inconvenient to from now on imagine that we do two steps at a time, so instead we will continue to consider the standard one-step-at-a-time overlapping cycles shuffle and say that the cards in positions $m-1,m,n$ are in collision at time $t$ if
\begin{itemize}
	\item $t$ is even
	\item the coins flipped in steps $t,t+1$ land Heads, Tails or Tails, Heads.
\end{itemize}

To apply Theorem $\ref{3monte}$ to the overlapping cycles shuffle, we will need to examine the probabilities that cards $i,k,j$ end up in positions of $m-1,m,n$ respectively after an even number of steps. As a warm up we will first deal with a few special cases that elude the parameters of the general theorem in Section \ref{entropysection}.

\begin{lemma}\label{l1collide}
	Consider the overlapping cycles shuffle on $n$ cards where $\frac{m}{n} \in (\epsilon,1-\epsilon)$. Let $i,j,k$ be cards in $(n-\sqrt{n},n]$ such that $k=i+1$ and $j > i$. Let $T = 2n$ and $t = 2n+4\sqrt{n}$. Let $E$ be the event that the first time $i$ or $j$ or $k$ experiences a collision after time $T$, the collision is before time $t$ and with each other in the order $(i,k,j)$. Then if $n$ is sufficiently large we have
	\begin{equation}
		\probp{E} \geq \frac{D}{\sqrt{n}}
	\end{equation}
	for a constant $D$ which depends on $\epsilon$.
\end{lemma}

The idea for the proof is that we will show that with probability bounded away from 0 that $i$ and $k$ stay ``glued together''. Note that since $i$ and $k$ begin adjacent to each other in the part bottom of the deck, they will stay adjacent at least until one of them leaves the bottom part of the deck. Just before this, if $i_r = n-1$ and $k_r = n$, if the next two flips are Tails, then we will have $i_{r+2} = 1$ and $k_{r+2} = 2$. So $i$ stays one position above $k$. If $i$ always copies the coin that $k$ uses when in positions $m$ and $n$ then $i$ will always stay in the position above $k$. It then suffices to show that when $j$ reaches position $n$ that $k$ reaches position $m$, because if $i$ stays ``glued'' then we will also have $i$ in position $m-1$. These are the correct positions for $i,k,j$ to collide in that order. Based on Corollary \ref{cardmovementcompare} we know that there are on the order of $\sqrt{n}$ positions nearby to $n$ with respect to $\norm{\cdot}$ which $i$ is likely to be in when $j$ reaches position $n$. Since $\norm{p(m-1)-p(n)} = \sqrt{n}$ we see that $m-1$ is one of these positions.

\begin{proof}
	Let $Q$ be the event that the following hold:
    \begin{itemize}
        \item The first time $i$ is in position $m$, a Tails is flipped, sending $i$ to position $m+1$. The first two times $i$ is in position $n$, a Tails is flipped sending $i$ to position $1$.
        \item The first time $k$ is in position $m$, a Tails is flipped, sending $k$ to position $m+1$. The first two times $k$ is in position $n$, a Tails is flipped sending $k$ to position $1$.
        \item The first time $j$ is in position $m$, a Heads is flipped, sending $j$ to position $1$. The second time $j$ is in position $m$, a Tails is flipped sending $j$ to position $m+1$.
    \end{itemize}
    Note that $\probp{Q} \approx (\frac{1}{2})^8$ because it deals with $8$ coin flips. We don't have equality because of the rare chance that some of these coin flips occur in the same step, but this effect on the probability is negligible. \\
    \\
    On event $Q$, we know that $i$ and $k$ will be glued together, in that $i$ will stay "right above" $k$ (more precisely in $k$'s previous position). This is because when $i$ and $k$ are both in the top part of the deck, they deterministically each move down one, and if $i$ and $k$ are both in the bottom part of the deck, they will either both stay put with a Heads flip or both move down one with a Tails flip. Finally, if $k$ reaches position $n$ or position $m$, the event $Q$ dictates that $i$ will copy the coin $k$ uses and "follow" $k$ immediately to either position $1$ or position $m+1$. \\
    \\
    Note that on the event $Q$, both $i$ and $k$ are sent to position $m+1$ after their first visit to $m$, but $j$ is cycled back to position $1$. If $m \leq \frac{n}{2}$ then $i,k$ will spend at least $m - \sqrt{n}$ steps in the bottom part of the deck before $j$ loops back to position $m$ and enters the bottom part of the deck as well. That is $m$ steps for $j$ to go from position $1$ to position $m$, minus $\sqrt{n}$, accounting for the fact that $j$ starts up to $\sqrt{n}$ steps blow (and therefor "ahead") if $i,j$. Note that $i,k$ will spend at least $n-m$ steps in the bottom part of the deck before cycling back to position $1$. Since $\epsilon n < m < (1-\epsilon) n$, we know that $i,k$ will spend at least $\min\{m - \sqrt{n}, n - m\} \geq \epsilon n - \sqrt{n}$ steps in the bottom part of the deck apart from $j$. For large enough $n$ this is at least $\frac{1}{2} \epsilon n$ steps. \\
    \\
    Let $\tau$ be the number of steps it takes for card $k$ to reach position $m$ for the second time. On the event $Q$, since $i$ stays one position above $k$, we will have $i_\tau = m-1$. It now remains to show that with appropriate probability we have $j_\tau = n$. This will put $i,j,k$ in the right position to collide in the next two steps. Referring to Corollary \ref{cardmovementcompare}, we know that $j$ will be approximately $m$ positions above $i$ because $i$ flipped one Heads when in position $m$ and $j$ flipped none. This is good, because when we account for the $\mod 2n-m+1$ wrapping around, we see that position $n$ is $m$ positions "above" position $m-1$. To have $j_\tau = n$ exactly, we will need $(T_S(j)-H_S(j)) - (T_S(i)-H_S(i))$ to equal $p(i_0)-p(j_0)$ exactly. Since $i$ and $j$ start out less than $\sqrt{n}$ positions from each other, we know that $|p(i_0)-p(j_0)| < 2\sqrt{n}$. Recall that $i$ spends at least $\frac{1}{2}\epsilon n$ steps in the bottom of the deck without $j$, and $j$ will spend some more time (on the order of $\epsilon n$) steps in the bottom of the deck without $i$. Since the Heads-Tails differentials are simple symmetric random walks, this puts the standard deviation of the drift between the two Heads-Tails differentials on the order of $\sqrt{\epsilon n}$. So $|p(i_0)-p(j_0)|$ is within $\frac{1}{\sqrt{\epsilon}}$ standards deviations. This means
    \begin{equation}
        \probp{j_\tau = n \ | \ Q} \geq \frac{D_1}{\sqrt{n}}
    \end{equation}
    where $D_1$ is a constant that depends on epsilon. Since $\probp{Q}$ is a constant we have
    \begin{align}
        \probp{Q, i_\tau = m-1, k_\tau = m, j_\tau = n} &= \probp{Q, j_\tau = n} \\
        &\geq \frac{D_2}{\sqrt{n}}.
    \end{align}
    For $E$ to occur, we also need to have that $\tau \in (2n, 2n + 4\sqrt{n})$, and that $i,j,k$ will not collide with any other cards between time $2n$ and $\tau$. \\
    \\
    To analyze the value of $\tau$, note that $k$ will need to travel $\sqrt{n}$ positions from its starting position to the bottom of the deck. This will take approximately $2\sqrt{n}$ steps. Then $k$ must travel all the way through the deck, which takes approximately $m + 2(n-m) = 2n-m$ steps. Then $k$ must travel $m$ more steps to position $m$, which takes exactly $m$ steps. In total, this is about $2n + 2\sqrt{n}$ steps, and the additional $2\sqrt{n}$ wiggle room is more than enough to account for a few standard deviations of Heads-Tails differential moving $k$ extra slow/fast through the bottom part of the deck. \\
    \\
    Once we have $Q$ and $j_\tau = n$ and $\tau \in (2n, 2n+4\sqrt{n})$, we will have $E$ as long as $i,j,k$ collide on step $\tau$, and haven't collided between step $2n$ and $\tau$. Cards $i,j,k$ collide in step $\tau$ if $\tau$ is even and the coins for steps $\tau$ and $\tau+1$ land Heads, Tails or Tails, Heads. The probability that $\tau$ is even is about $\frac{1}{2}$, as $i,j,k$ start out in the bottom part of the deck, there is a $\frac{1}{2}$ chance that the first flip is Tails which doesn't move any of $i,j,k$ and "wastes" a move, changing the parity of $\tau$. The only reason the probability isn't exactly $\frac{1}{2}$ is this wasted move may effect the probability that $\tau \in (2n, 2n+4\sqrt{n})$, but this is negligible. If we know $\tau$ is even, there is an independent $\frac{1}{2}$ chance the next two flips land Heads, Tails or Tails, Heads, so the probability of a collision at time $\tau$ is about $\frac{1}{4}$. Finally, note that $i,k$ are in the top part of the deck between steps $\tau-4\sqrt{n}$ and $\tau$, and $j$ is in the bottom part of the deck. In particular, they are not in positions $m-1,m$ or $n$ during an even time, so they can't collide. The only exception is if $j$ reaches position $n$ before time $\tau$, but in this case all flips until time $\tau$ must be tails for $j$ to stay in position $n$, which rules out the Heads, Tails or Tails, Heads flips needed for $n$ to collide. \\
    \\
    The probability of $E$ is constant conditioned on $Q$ and $j_\tau = n$, so
    \begin{align}
        \probp{E} &\geq \probp{E \ | \ Q, j_\tau = n} \cdot \probp{Q, j_\tau = n} \\
        &= D_3 \cdot \frac{D_2}{\sqrt{n}} \\
        &= \frac{D}{\sqrt{n}}
    \end{align}
    where $D$ is a constant that depends on $\epsilon$.
\end{proof}

\begin{corollary}\label{l1collidecor}
	Consider the overlapping cycles shuffle on $n$ cards where $\frac{m}{n} \in (\epsilon, 1 - \epsilon)$. Let $i,j,k$ be cards in $(n-\sqrt{n},n]$ such that $k=i+2$ and $j=i+1$. Let $T = 2n$ and $t = 2n+4\sqrt{n}$. Let $E$ be the event that the first time $i$ or $j$ or $k$ experiences a collision after time $T$, the collision is before time $t$ and with each other in the order $(i,k,j)$. Then
	\begin{equation}
		\probp{E} \geq \frac{D}{\sqrt{n}}
	\end{equation}
    for a constant $D$ which depends on $\epsilon$.
\end{corollary}

\begin{proof}
	The proof is nearly identical to the one in Lemma \ref{l1collide}. In that proof we required $i,j,k$ to flip Tails, Heads, Tails respectively for each of their first visits to $m$. We require the same in this new case. Since $i,j,k$ start out adjacent to each other in that order, with probability $\frac{1}{8}$ they will all flip Tails when in position $n$ and so stay adjacent as they move to the top of the deck. Then, after $m-2$ more steps we will have $i,j,k$ in positions $m-2,m-1,m$ respectively. Then after the flips Tails, Heads, Tails we have $i,j,k$ in positions $m+1,2,m+2$ respectively. Now we have $i$ one position above $k$ as in Lemma \ref{l1collide}. The rest of the proof is the same and the equivalent result holds.
\end{proof}

These two results tell us that if $i,j \in (n-\sqrt{n},n]$ and $j>i$ then it is reasonably likely that $m_2(i)=j$ and $m_1(i)>i$ with regards to the notation in Theorem \ref{3monte}. We make this precise in the proposition below. Note that these inequalities are the opposite of those required in Theorem \ref{3monte}, but this is okay. Our choice of labeling card $1$ as the card in the top of the deck, and card $2$ as second to top, etc was arbitrary. We can utilize Theorem \ref{3monte} where the inequalities are with respect to any well-ordering of the deck, and we will in fact use a well-ordering later which starts its count from the bottom of the deck upwards. As further justification that this is a valid application of Theorem \ref{3monte}, note that entropy decay will be the same if we first use some deterministic change of basis permutation to reorder the cards and then apply the overlapping cycles shuffle.

\begin{proposition}\label{l1collidecomplete}
    Consider the overlapping cycles shuffle where $\frac{m}{n} \in (\epsilon, 1-\epsilon)$. Fix cards $i,j \in (n-\sqrt{n},n]$ such that $i \geq n-2$ and $i < j$. Let $T = 2n$ and $t=2n+4\sqrt{n}$. Let $m_1(i)$ be front match of $i$ the and let $m_2(i)$ be the back match, as defined by Theorem \ref{3monte}. Then,
	\begin{equation*}
		\probp{m_2(i)=j,m_1(i)>i} \geq \frac{D}{\sqrt{n}}
	\end{equation*}
	for a constant $D$ which depends on $\epsilon$.
\end{proposition}

\begin{proof}
	In the case that $j \neq i+1$ we have by Lemma \ref{l1collide} that
	\begin{equation}
		\probp{m_2(i)=j,m_1(i)>i} \geq \probp{m_2(i)=j,m_1(i)=i+1} \geq \frac{D_1}{\sqrt{n}}.
	\end{equation}
	In the case that $j = i+1$ we have by Corollary \ref{l1collidecor} that
	\begin{equation}
		\probp{m_2(i)=i+1,m_1(i)>i} \geq \probp{m_2(i)=i+1,m_1(i)=i+2} \geq \frac{D_2}{\sqrt{n}}.
	\end{equation}
	Taking $D = \min\{D_1,D_2\}$ completes the proof.
\end{proof}

\section{Movement of 3 Cards}

The goal of this section will be to show that after $t$ steps $i,j$ and $k$ have an approximately uniform distribution over all ``nearby'' positions, where our notion of ``near'' is related to the size of $t$. More specifically, we will show that after about $2\ell^2$ steps the cards $i,j,k$ are distributed approximately uniformly amongst positions $f_i,f_j,f_k$ such that $\norm{p(i)-p(f_i)},\norm{p(j)-p(f_j)},\norm{p(k)-p(f_k)} < \ell$. \\
\\
In order to prove things about the movement of cards $i,j$ and $k$ relative to each other, it will be useful to imagine that instead of using a single sequence of coins to determine if the card in position $m$ or $n$ is moved to the top in each step, we generate many sequence of coins, and choose which sequence to draw from in each step according to the state the deck is in. Here is one way of doing this:\\
\\
We generate coin sequences $B^i,B^j,B^k,S^i,S^j,S^k,S^{ij},S^{jk},S^{ik},S^{ijk}$ and simulate the movement of cards $i,j,k$ under the overlapping cycles shuffle according to the following rules:
\begin{itemize}
	\item If card $i$ is in position $m$, then use the next coin from $B^i$ to do the shuffle. Similarly, if $j$ or $k$ is in position $m$, use $B^j$ or $B^k$.
	\item Otherwise, use $S^A$ where $A \subseteq \{i,j,k\}$ is the set of which $i,j,k$ are in the bottom part of the deck. If all of $i,j,k$ are in the top part of the deck then these three cards will move down deterministically in the next step and no coin is necessary.
\end{itemize}
We will be able to bound the movement of $i,j$ and $k$ by putting restrictions on $S^i,S^j,S^k,S^{ij},S^{jk},S^{ik},$ and $S^{ijk}$. However, it may be the case that some of these sequences are drawn from more than others. For example, if the entire sequence of $B^i$ is made up of Heads, then $i$ will spend all its time in the top part of the deck and coins from $S^i,S^{ij},S^{ik},S^{ijk}$ will not be used at all! We need to make sure something like this does not happen. In particular, it will be important to show that $S^i,S^j,S^k$ are each drawn from a constant proportion of the time. This is equivalent to saying that $i,j,k$ each spend a constant proportion of time alone in the bottom part of the deck. Fortunately, this happens with high probability.

\begin{lemma}\label{ibottom}
	Consider the overlapping cycles shuffle on $n$ cards for sufficiently large $n$. Fix any cards $i,j,k$. Then with probability greater than $\frac{1}{8}\exp\left(-\frac{10n}{m}\right)$, at least at least $n-m$ out of the next $5n$ steps have $i$ in the bottom part of the deck and $j$ and $k$ in the top part of the deck.
\end{lemma}
\begin{proof}
	Let $A$ be the event that every time $j$ and $k$ are in position $m$ in the next $5n$ steps, a Heads is flipped sending $j$ and $k$ respectively back to position $1$. Since $j$ and $k$ each make at most $\frac{5n}{m}$ visits to position $m$ in $5n$ steps we have
	\begin{equation}
		\mathbb{P}(A) \geq \left(2^{-\frac{5n}{m}}\right)^2 \geq \exp\left(-\frac{10n}{m}\right).
	\end{equation}
	Let $B$ be the event that the next two times $i$ is in position $m$, Tails is flipped sending $i$ to the bottom part of the deck. Then $\probp{B} = \frac{1}{4}$. $A$ and $B$ are independent, so
	\begin{equation}
		\mathbb{P}(A,B) \geq \frac{1}{4}\exp\left(-\frac{10n}{m}\right).
	\end{equation} Let $G$ be the event that out of the next $5n$ coin flips, no more than $3n$ are Heads. This happens with high probability, but we will just use that for large enough $n$, we have $\mathbb{P}(G) \geq 1 - \frac{1}{8}\exp\left(-\frac{10n}{m}\right)$. Then by the union bound we have,
	\begin{equation}
		\mathbb{P}(A,B,G) \geq \frac{1}{8}\exp\left(-\frac{10n}{m}\right).
	\end{equation}
	On the events $A,B,G$ it must be that $i$ spends at least $n-m$ steps alone in the bottom part of the deck. This is because $G$ guarantees that, in the case that $j$ or $k$ begin in the bottom part of the deck, they will move swiftly to position $n$ and then wrap around to position $1$. Then $A$ guarantees that $j$ and $k$ will continue to cycle through the top part of the deck, and $B$ guarantees that $i$ moves alone into the bottom part of the deck. Since there are $n-m$ positions in the bottom part of the deck, and $C$ guarantees that $2n$ Tails are flipped to facilitate $i$ through at least 2 complete laps of the deck (at least one where it is alone in the bottom), we have proven the lemma.
\end{proof}

\begin{corollary}\label{ijkbottomt}
	Consider the overlapping cycles shuffle on $n$ cards for sufficiently large $n$. For any $a \in \mathbb{N}$, after $t = 5an$ steps, the probability that each of $i,j,k$ spend at least $\frac{a(n-m)}{16}\exp\left(-\frac{10n}{m}\right)$ of these steps alone in the bottom part of the deck is at least
	$$
	1 - 3\exp\left(-\frac{a}{32\exp\left(\frac{10n}{m}\right)}\right).
	$$
\end{corollary}
\begin{proof}
	We can divide $t$ into $a$ steps of size $5n$. By Lemma \ref{ibottom}, in each of these blocks of $5n$ steps card $i$ spends at least $n-m$ steps in the bottom part of the deck with probability greater than $\frac{1}{8}\exp\left(-\frac{10n}{m}\right) $ . Let $Y$ be the number of steps spent by $i$ in the bottom part of the deck in $t$ steps. Then $\frac{Y}{n-m}$ stochastically dominates a Binomial$\left(a,\frac{1}{8}\exp\left(-\frac{10n}{m}\right)\right)$ random variable. We can use Theorem \ref{hoeffdingbinomial}, Hoeffing's inequality on the binomial random variable to get
	\begin{align}
		\probp{\frac{Y}{n-m} \leq \frac{a}{16\exp\left(\frac{10n}{m}\right)}} \leq \exp\left(-\frac{a}{32\exp\left(\frac{10n}{m}\right)}\right).
	\end{align}
	Thus with high probability $i$ spends a bounded away from zero proportion of time alone in the bottom of the deck. In fact the probability that this does not happen is exponentially low. The same applies for $j$ and $k$ so using the union bound we get the result of the corollary.
\end{proof}

In the following proposition we aim to simplify Corollary \ref{ijkbottomt}. We do this by defining a constant $L$ which depends on the ratios $\frac{n}{m}$ and $\frac{n}{n-m}$ which absorbs most of the complicated coefficients in the parameters and bound of the corollary.

\begin{proposition}\label{ijkbottomC}
	Consider the overlapping cycles shuffle on $n$ cards where $\frac{m}{n} \in (\epsilon,1-\epsilon)$ and $n$ is sufficiently large. There is a constant $L$ which depends on $\epsilon$ such that the following holds: Choose some $t \in \mathbb{N}$ such that $t \geq Ln$. After $t$ steps, the probability that each of $i,j,k$ spend at least $\frac{t}{L}$ steps alone in the bottom part of the deck is at least
	\begin{equation}
		1 - 3\exp\left( -\frac{t}{2Ln} \right)
	\end{equation}
\end{proposition}
\begin{proof}
    Let 
	\begin{equation*}
		L = \frac{160n\exp(\frac{10n}{m})}{(n-m)}
	\end{equation*}
	We will use the previous Corollary \ref{ijkbottomt}. Set
	\begin{equation}
		a = \floor*{\frac{t}{5n}}.
	\end{equation}
	Since $L > 10$, and $t \geq Ln > 10n$ we have $a \geq \frac{t}{10n}$. Note that
	\begin{equation}
		\frac{a(n-m)}{16}\exp\left(-\frac{10n}{m}\right) \geq t \cdot \frac{(n-m)}{160n}\exp\parens*{-\frac{10m}{n}} = \frac{t}{L}
	\end{equation}
	and
	\begin{equation}
		\frac{a}{32\exp(\frac{10n}{m})} > t \cdot \frac{1}{320n} \exp\parens*{-\frac{10n}{m}} = \frac{t}{2Ln} \cdot \frac{n}{(n-m)} > \frac{t}{2Cn}
	\end{equation}
	So the probability that $i,j,k$ each spend at least $\frac{t}{L}$ steps alone in the bottom part of the deck is at least
	\begin{equation}
		1 - 3\exp\left( -\frac{t}{2Ln} \right).
	\end{equation}
\end{proof}

As previously mentioned, our goal in this section will be to control the movement of cards $i,j,k$. To do this, we will consider 3 ``stages''. In Stage 1, we will show that $i,j,k$ spread out from each other in terms of $\norm{\cdot}$. In this first stage we will not need to make precise statements about exactly the positions $i,j,k$ travel to. It will be enough to guarantee that they have a gap between each other on the order of $\ell$ after $\ell^2$ steps. \\
\\
In Stage 2 we will use a coupling argument to show that, provided that $i,j,k$ are spread out from each other, their distribution after $\ell^2$ steps is approximately uniform over cards within distance $\ell$ (in the $\norm{\cdot}$ sense) to $i,j,k$ respectively. \\
\\
Stage 3 is the inverse of the first stage. We show that $i,k,k$ spread out under the inverse overlapping cycles shuffle as well. To put the three stages together, we use Stage 1 to show that $i,j,k$ spread out, Stage 2 to show that $i,j,k$ move precisely to locations likely to hit our desired targets, and Stage 3 to show $i,j,k$ move back together to hit their targets. \\
\\
This next Proposition is useful for Stage 1 because it shows we can move $i,j,k$ nearby to targets which we will later choose to be spread out from each other.

\begin{proposition}\label{sdspread}
	Consider the overlapping cycles shuffle where $\frac{m}{n} \in (\epsilon,1-\epsilon)$. There exist constants $C,D$ which depend on $\epsilon$ such that the following holds: Fix $\ell$ such that $C\sqrt{n} \leq \ell \leq n$. Fix any positions $i,j,k,f_i,f_j,f_k$ such that
	\begin{equation*}
		\norm{p(i)-p(f_i)},\norm{p(j)-p(f_j)},\norm{p(k)-p(f_k)} \leq 2\ell.
	\end{equation*}
	Choose $T \in [\ell^2,\ell^2+4n]$ such that $T + \floor*{\frac{T}{2n}}(m-1) \equiv 0 \mod 2n-m+1$. Then,
	\begin{equation*}
		\probp{\norm{p(i_T)-p(f_i)},\norm{p(j_T)-p(f_j)},\norm{p(k_T)-p(f_k)} < \frac{\ell}{2000}} \geq D.
	\end{equation*}
\end{proposition}

\begin{proof}
	Since $\norm{p(i)-p(f_i)} < 2\ell$ we know that $p(f_i) = p(i) + a_i + b_im$ where $|a_i| < 2\ell$ and $|b_i| < \frac{2\ell}{\sqrt{n}}$. Suppose $B^i$ is the record of the coins card $i$ flips whenever it is in position $m$. Let $G_i$ be the event that the following holds:
	\begin{itemize}
		\item There exist $r \leq \frac{\ell^2}{2n}$ such that $\text{Diff}_r(B^i) = y_i$ where $\floor*{y_i\parens*{1-\frac{m-1}{2n}}} = b_i$.
		\item For all $s$ with $r \leq s \leq \frac{2T}{n}$ we have $\text{Diff}_r(B^i) \in (y_i - \frac{\ell}{8000\sqrt{n}}, y_i + \frac{\ell}{8000\sqrt{n}})$.
	\end{itemize}
	In event $G_i$ we require $\text{Diff}_r(B^i)$, which is a simple symmetric random walk, to hit $y_i$ within $\frac{\ell^2}{2n}$ steps and then stay close to $y_i$ until $\frac{T}{2n}$ steps. Note that $|y_i| \leq (1-\frac{m-1}{2n})^{-1}|b_i| \leq \frac{4\ell}{\sqrt{n}}$, which is a constant number of standard deviations of the maximum magnitude of $\text{Diff}_r(B^i)$ after $\frac{\ell^2}{2n}$ steps. Similarly, $\frac{\ell}{8000\sqrt{n}}$ is constant fraction of the same standard deviation. By Theorems \ref{invershoeffding} and \ref{hoeffdingmax} the event $G_i$ occurs with probability $D_1$ for some constant $D_1$. Similarly define events $G_j,G_k$ for cards $j$ and $k$. Since the coins used by $i,j,k$ when in position $m$ are independent, we have that
	\begin{equation*}
		\probp{G_i,G_j,G_k} \geq D_1^3.
	\end{equation*}   
	Now imagine that whenever $i,j,k$ are not in position $m$, we use the following sequences of $T$ small coins: $S^i,S^j,S^k,S^{ij},S^{ik},S^{jk},S^{ijk}$ as follows: If the coins in set $A \subseteq \{i,j,k\}$ are in the bottom part of the deck, use the next coin in the sequence $S^A$. Let $L$ be the constant depending on $\epsilon$ from Proposition \ref{ijkbottomC}. Let $E$ be the event that
	\begin{itemize}
		\item $|\text{Diff}_r(S^{ij})|,|\text{Diff}_r(S^{ik})|,|\text{Diff}_r(S^{jk})|,|\text{Diff}_r(S^{ijk})| <\frac{\ell}{16000}$ for all $r \leq T$.
		\item There exists $r_i,r_j,r_k \leq \frac{\ell^2}{L}$ such that $\text{Diff}_{r_i}(S^i) = a_i, \ \text{Diff}_{r_j}(S^j) = a_j, \ \text{Diff}_{r_k}(S^k) = a_k$.
		\item For all $s_i,s_j,s_k$ with $r_i \leq s_i \leq T, \ r_j \leq s_j \leq T, \ r_k \leq s_k \leq T$ we have 
		\begin{align*}
			\text{Diff}_{s_i}(S^i) &\in (a_i - \frac{\ell}{16000}, a_i + \frac{\ell}{16000}), \\
			\text{Diff}_{s_j}(S^j) &\in (a_j - \frac{\ell}{16000}, a_j + \frac{\ell}{16000}), \\
			\text{and } \text{Diff}_{s_k}(S^k) &\in (a_k - \frac{\ell}{16000}, a_k + \frac{\ell}{16000}).
		\end{align*}
	\end{itemize}
	Again by Theorems \ref{invershoeffding} and \ref{hoeffdingmax} we get,
	\begin{equation}
		\probp{E} \geq D_2\exp(-D_3 L) \cdot D_4^7
	\end{equation}
	where $D_2,D_3,D_4$ are universal constants. The $D_2\exp(-D_3 L)$ portion follows from Theorem \ref{invershoeffding}: If we imagine $\text{Diff}(S^i), \ \text{Diff}(S^j), \ \text{Diff}(S^k)$ as Binomial random variables with $\frac{\ell}{\sqrt{L}}$ trials, we require these trials to at some point have $a_i,a_j,a_k$ successes where $|a_i|,|a_j|,|a_k| < 2\ell$. The constant $D_4$ follows from Theorem \ref{hoeffdingmax}: We require all seven Heads-Tails differentials to stay within $\frac{\ell}{1600}$ of their targets for $T$ steps, and the standard deviation of their maximum magnitude after $T$ steps is on the order of $\ell$. Since the ``$B$ coin sequences'' are generated independently of the ``$S$ coin sequences'' we get
	\begin{equation*}
		\probp{G_i,G_j,G_k,E} \geq D_1^3 \cdot D_2 \exp(-D_3 L) D_4^7
	\end{equation*}
	Finally, let $Q$ be the event that $i,j,k$ each spend at least $\frac{\ell^2}{L}$ steps in the bottom of the deck. Recall that we have $\ell > C\sqrt{n}$ for some $C$ which we have the freedom to choose. By Corollary \ref{ijkbottomC} we have that
	\begin{align*}
		\probp{Q} &\geq 1 - 3\exp \parens*{-\frac{\ell^2}{2Ln}} \\
		&\geq 1 - 3 \exp\parens*{-\frac{C}{2L}}.
	\end{align*}
	Using the fact that $L > 160$ pick $C$ large enough that $3\exp(-\frac{C}{2L}) < D_1^3 \cdot D_2 \exp(-D_3 L) D_4^7$. Then by the union bound,
	\begin{equation}
		\probp{G_i,G_j,G_k,E,Q} \geq \frac{1}{2}D_1^3 \cdot D_2 \exp(-D_3 L) D_4^7 = D_5 \exp(-D_3 L)
	\end{equation}
	On the events $G_i,G_j,G_k,E,Q$ we have by Proposition \ref{movementbounds} that
	\begin{equation}
		\norm*{p(i_T)-p(f_i)} < 4 \cdot \frac{\ell}{16000} + \frac{\ell}{8000} + \frac{1}{2}\parens*{|y_i|+\frac{\ell}{8000\sqrt{n}}} + \frac{|a_i|+\frac{\ell}{8000}}{2n}(\sqrt{n}+1) + 4\sqrt{n}
	\end{equation}
	where the $4 \cdot \frac{\ell}{16000}$ comes from the fact that $S^i,S^{ij},S^{ik},S^{ijk}$ each differ at most $\frac{1}{16000}$ from their target Heads-Tails differential and the $\frac{\ell}{8000}$ comes from the fact that $B^i$ differs at most $\frac{\ell}{8000\sqrt{n}}$ from its target Heads-Tails differential. Using that fact that $\ell \leq n$ and $|a_i| \leq 2\ell$ and $|y_i| \leq \frac{4\ell}{\sqrt{n}}$ we get
	\begin{equation}
		\norm*{p(i_T)-p(f_i)} < \frac{3\ell}{8000} + \frac{2\ell}{\sqrt{n}} + \frac{\ell}{16000\sqrt{n}} + \frac{2\ell}{2\sqrt{n}} + \frac{2\ell}{2n}+ \frac{\ell}{16000\sqrt{n}} + \frac{\ell}{16000n} + 4\sqrt{n}.
	\end{equation}
	Recall that $\ell > C\sqrt{n}$ so $\sqrt{n} < \frac{\ell}{C}$. Choose $C$ large enough that
	\begin{equation}
		\norm{p(i_T)-p(f_i)} < \frac{\ell}{2000}.
	\end{equation}
	A similar argument shows the equivalent results for $j$ and $k$ follow from $G_i,G_j,G_k,E,Q$ and this completes the proof.
\end{proof}

Stage 3 is Stage 1 in reverse, so we will need the equivalent result for the inverse overlapping cycles shuffle.

\begin{corollary}\label{sdspreadinverse}
	For the same constants $C,D$ in Proposition \ref{sdspread} the following holds: Fix $\ell$ such that $C\sqrt{n} \leq \ell \leq n$. Fix any positions $i,j,k,f_i,f_j,f_k$ such that
	\begin{equation*}
		\norm{p(i)-p(f_i)},\norm{p(j)-p(f_j)},\norm{p(k)-p(f_k)} \leq 2\ell.
	\end{equation*}
	Choose $T \in [\ell^2,\ell^2+4n]$. If $i_{(-T)},j_{(-T)},k_{(-T)}$ are the locations of $i,j,k$ after $T$ steps of the inverse overlapping cycles shuffle we have
	\begin{equation*}
		\probp{\norm{p(i_{(-T)})-p(f_i)},\norm{p(j_{(-T)})-p(f_j)},\norm{p(k_{(-T)})-p(f_k)} < \frac{\ell}{2000}} \geq D
	\end{equation*}
\end{corollary}

\begin{proof}
	We know that the inverse overlapping cycles shuffle is the same as the forward overlapping cycles shuffle after reordering the cards. As Corollary \ref{inversenorm} tells us, $\norm{p(\cdot)-p(\cdot)}$ is invariant under $\sigma$. Since all parameters in Proposition \ref{sdspread} are in terms of $\norm{p(\cdot)-p(\cdot)}$, the equivalent statement also holds for the inverse overlapping cycles shuffle.
\end{proof}

The next several lemmas will work to prove our desired result for Stage 2: that if $i,j,k$ are spread out then they will have an approximately uniform distribution after $T$ steps. The first lemma tells us that for a reasonable choice $b_i,b_j,b_k$ we can lower bound the probability that $i,j,k$ flip exactly $b_i,b_j,b_k$ Heads when in position $m$.

\begin{lemma}\label{bigcoinsnoise}
	Choose cards $i,j,k$. Fix some $\ell$ such that  $2\sqrt{n} \leq \ell \leq \frac{1}{2}n$. Choose any $T \in [\ell^2,\ell^2+4n]$. Choose $b_i,b_j,b_k \in \left( \frac{T}{2n} - \frac{\ell}{\sqrt{n}}, \frac{T}{2n} + \frac{\ell}{\sqrt{n}} \right)$.
	\begin{itemize}
		\item Let $R_i$ be the number of Heads flipped by $i$ while in position $m$ in the first $T$ steps.
		\item Let $E_i$ be the event that $|p(i_t)-p(i)-t+mb_i|_M \leq 6\ell$.
	\end{itemize}
	Let $R_j,R_k,E_j,E_k$ be defined similarly for $j$ and $k$. Then,
	\begin{equation*}
		\probp{(R_i,R_j,R_k)=(b_i,b_j,b_k),E_i,E_j,E_k} \geq 7 \cdot 10^{-8} \cdot \frac{n^\frac{3}{2}}{\ell^3}
	\end{equation*}
\end{lemma}
\begin{proof}
	Let $S^i,S^j,S^k$ be the record of all coins used by $i,j,k$ respectively when in the bottom part of the of deck. Let $A_i$ be the event that $\text{Diff}_r(S_i) \leq 3\sqrt{T}$ for all $r \leq T$. Let $A_j$,$A_k$ be similarly defined for $j$ and $k$. By Hoeffding's inequality (Theorem \ref{hoeffdingmax} specifically) and the union bound we get
	\begin{equation}
		\probp{A_i,A_j,A_k} \geq 1-12\exp\left(-\frac{9}{2}\right) > \frac{17}{20}.
	\end{equation}
	Let $B^i,B^j,B^k$ be the record of coins used by $i,j,k$ respectively when in position $m$. On the events $A_i,A_j,A_k$ we can make precise statements about the number of times $i,j,k$ pass through position $m$ based on $B^i,B^j,B^k$. Following the proof of Proposition \ref{movementbounds} we recall that, on $A_i$, each time card $i$ is in position $m$ it will take exactly $m$ more steps to return to position $m$ for a short return and $2n-m$ steps to return to position $m$ for a long return, without including faster/slower movement in the bottom part of the deck dictated by the Heads-Tails differential of $S^i$. Thus if the first $c$ coins in $B^i$ have $h$ Heads and $c-h$ Tails, we know that it will take 
	$$m-i+hm+(c-h)(2n-m+1)+\Delta \text{ steps}$$
	if $i \leq m$ and 
	$$m-i+2(n-m)+hm+(c-h)(2n-m+1)+\Delta \text{ steps}$$
	if $i > m$ to reach position $m$ for the $(c+1)$th time, where $\Delta$ is the Heads-Tails differential of $S^i$ at that time. Let $\tau$ be the time when $i$ reaches position $m$ for the $(c+1)$th time. On event $A_i$ we have that $|\Delta| \leq 3\sqrt{T} \leq 3\sqrt{\ell^2+4n} \leq 3n$. This along with the fact that $0 \leq m - i$ in the first equation and $0 \geq m-i$ in the second gives us,
	\begin{equation}
		hm+(c-h)(2n-m+1)-3n \leq \tau \leq hm+(c-h)(2n-m+1)+5n. \label{chlowerupper}
	\end{equation}
	Fix some desired value of $h$ which we will call $h_0$. Setting the upper bound in (\refeq{chlowerupper}) equal to $T$ and solving for $c$, we get
	\begin{equation}
		c = \frac{T-5n}{2n-m+1} - h_0 \cdot \frac{2m-2n-1}{2n-m+1}.
	\end{equation}
	Thus, if we let
	\begin{equation}
		c^* = \floor*{\frac{T-5n}{2n-m+1} - h_0 \cdot \frac{2m-2n-1}{2n-m+1}}
	\end{equation}
	then if $h_0$ Heads appear in the first $c^*$ coins from $B^i$, we know that on event $A_i$ at time $\tau \leq T$ card $i$ will have drawn exactly $c^*$ coins from $B^i$ and therefore exactly $h_0$ Heads and $c^*-h_0$ Tails. Furthermore, we can substitute $c^*$ in to the lower bound for $\tau$ in (\refeq{chlowerupper}) and get
	\begin{equation}
		T - 8n - (2n - m + 1) \leq \tau.
	\end{equation}
	Let $G_i$ be the event that card $i$ flips exclusively Tails when in position $m$ between time $\tau$ and time $T$. Then $i$ will take at least $n$ steps to cycle back to $m$ each time and therefore can make at most $10$ of these cycles in the fewer than $8n+2n-m+1$ steps between $\tau$ and $T$. So
	$$
	\probp{G_i} \geq \left( \frac{1}{2} \right)^{10} = \frac{1}{1024}.
	$$
	For any valid $h$ let $X(h)$ be a binomial random variable with $\floor*{\frac{T-5n}{2n-m+1} - h \cdot \frac{2m-2n-1}{2n-m+1}}$ trials and $\frac{1}{2}$ chance of success. Then,
	\begin{equation}
		\probp{R_i = h} \geq \frac{1}{1024} \cdot \probp{X(h) = h}.
	\end{equation}
	Let $h = \frac{T}{2n} + \delta$ where $|\delta| \leq \sqrt{\frac{T}{n}}$. The number of trials $X(h)$ has is
	\begin{align}
		&\floor*{\frac{T-5n}{2n-m+1} - \parens*{\frac{T}{2n} + \delta} \cdot \frac{2m-2n-1}{2n-m+1}} \\
		= &\floor*{\frac{T-5n-\frac{m}{n}T+T+\frac{1}{2n}T}{2n-m+1} - \delta \cdot \frac{2m-2n-1}{2n-m+1}} \\
		= &\floor*{\frac{(2-\frac{m}{n}+\frac{1}{n})T}{2n-m+1} - \frac{\frac{1}{2n}T}{2n-m+1} - \frac{5n}{2n-m+1} - \delta \cdot \frac{2m-2n-1}{2n-m+1}} \\
		= &\floor*{\frac{T}{n} - \frac{\frac{1}{2n}T}{2n-m+1} - \frac{5n}{2n-m+1} + \delta \cdot \frac{2n-2m+1}{2n-m+1}}.
	\end{align}
	Recall that $T \leq n^2$ so,
	\begin{equation}
		\frac{1}{2n}T \leq \frac{1}{2}n < n
	\end{equation}
	for large enough $n$. This means that
	\begin{equation}
		- \frac{\frac{1}{2n}T}{2n-m+1} - \frac{5n}{2n-m+1} \geq -\frac{6n}{2n-m+1} \geq -6.
	\end{equation}
	Also note that
	\begin{equation}
		0 \leq \frac{2n-2m+1}{2n-m+1} \leq 1 \text{ for all } m \leq n.
	\end{equation}
	Thus we know that the number of trials $X(h)$ has is within
	\begin{equation}
		\parens*{\frac{T}{n} - 6, \ \frac{T}{n} + \delta}.
	\end{equation}
	The expected value of $X(h)$ is therefore between $\frac{T}{2n}-3$ and $\frac{T}{2n} + \frac{\delta}{2}$, and the standard deviation of $X(h)$ is between $\frac{1}{2}\sqrt{\frac{T}{n}-6}$ and $\frac{1}{2}\sqrt{\frac{T}{n}+\delta}$. We are interested in bounding $\probp{X(h) = h}$ from below. Since $|\delta| \leq \sqrt{\frac{T}{n}}$, the expected value of $X(h)$ (which is between $\frac{T}{2n}-3$ and $\frac{T}{2n}+\frac{\delta}{2}$) differs from the value of $h$ (which is $\frac{T}{2n} + \delta)$ by at most about two standard deviations. To be safe, we will say that this is less than three standard deviations, and this use a normal approximation to say
	\begin{equation}
		\probp{X(h)=h} \geq \frac{1}{\sqrt{2\pi}}\exp\left(-\frac{9}{2}\right) \sqrt{\frac{n}{T}} \hspace{1cm} \text{ for } h \in \parens*{\frac{T}{2} - \sqrt{\frac{T}{n}}, \ \frac{T}{2} + \sqrt{\frac{T}{n}}}.
	\end{equation}
	Since we required $b_i \in \parens*{\frac{T}{2n}-\frac{\ell}{\sqrt{n}},\frac{T}{2n}+\frac{\ell}{\sqrt{n}}}$ $\ell$ and $\ell \leq \sqrt{T}$ we know $b_i$ meets these parameters. \\
	\\
	The entire argument for card $i$ applies to $j$ and $k$ as well due to symmetry. Note that the bounds on $R_i,R_j,R_k$ are determined by considering the order of the coins used by $i,j,k$ when they are in position $m$ respectively. Thus, the events considered to achieve these bounds are independent after conditioning on $A_i,A_j,A_k$. This means that
	\begin{align}
		\mathbb{P}\big(A_i,A_j,A_k,(R_i,R_j,R_k) = (b_i,b_j,b_k)\big) &\geq \frac{17}{20} \parens*{\frac{1}{1024} \cdot \frac{1}{\sqrt{2\pi}}\exp\left(-\frac{9}{2}\right) \sqrt{\frac{n}{T}}}^3 \\
		&\geq 6 \cdot 10^{-17} \cdot \left(\frac{n}{T}\right)^\frac{3}{2}.
	\end{align}
	Recall that $A_i,A_j$ and $A_k$ are the events that $|\text{Diff}_r(S^i)|$, $|\text{Diff}_r(S^j)|$, and $|\text{Diff}_r(S^k)|$ respectively stay less than or equal to $3\sqrt{T}$. Since $T \leq \ell^2+4n$ and $\ell \geq 2\sqrt{n}$ we see that $3\sqrt{T} \leq \frac{9}{2}\ell$. On events $A_i,A_j,A_k$ and $(R_i,R_j,R_k) = (b_i,b_j,b_k)$ we have by Proposition \ref{cardmovement} that,
	\begin{align}
        |p(i_t)-p(i)-t-(T_S(i)-H_S(i)) - (T_B(i)-mb_i)|_M &= 0 \\
		|p(i_t)-p(i)-t+mb_i|_M &\leq |T_S - H_S| + T_B \\
		&\leq \frac{9}{2}\ell + (\ell + 5) < 6\ell
	\end{align}
	where we know $T_B(i) \leq \ell + 5$ because at most $\frac{T}{n}+1 \leq \frac{\ell^2+4n}{n}+1$ Tails may be flipped by $i$ when in position $m$, and $\ell \leq n$. A similar argument shows the equivalent statement for $j$ and $k$ also follow for $A_i,A_j,A_k$ and $(R_i,R_j,R_k) = (b_i,b_j,b_k)$. This tells us that
	\begin{equation}
		\probp{(R_i,R_j,R_k) = (b_i,b_j,b_k),E_i,E_j,E_k} \geq 6 \cdot 10^{-17} \cdot \parens*{\frac{n}{T}}^\frac{3}{2}.
	\end{equation}
	Since $T \geq \ell^2$ we get
	\begin{equation}
		\probp{(R_i,R_j,R_k) = (b_i,b_j,b_k),E_i,E_j,E_k} \geq 6 \cdot 10^{-17} \cdot \frac{n^\frac{3}{2}}{\ell^3}.
	\end{equation}
\end{proof}

The following lemmas will collectively make the following argument: For cards $i,j,k$ fix some desired number Heads flips while in position $m$. Now let $e_i,e_j,e_k$ be positions where you would expect $i,j,k$ to go in $T$ steps. Then if we choose positions $f_i,f_j,f_k$ uniformly at random nearby $e_i,e_j,e_k$ it is likely enough $i,j,k$ will go to $f_i,f_j,f_k$ in $T$ steps. \\
\\
We then show that the same statement applies to the inverse overlapping cycles shuffle, and this gives us the following variation: Uniformly choose $a_i,a_j,a_k$ nearby $i,j,k$. Then after $T$ steps it is likely enough $a_i,a_j,a_k$ will go to $e_i,e_j,e_k$. It will also be true that if we choose, for example, $a_i,a_j,a_k$ nearby $i+1,j+1,k+1$ then it is likely enough $a_i,a_j,a_k$ go to $e_i+1,e_j+1,e_k+1$. Using this logic, we show that if $a_i,a_j,a_k$ are chosen from an extra wide range of values around $i,j,k$ then $a_i,a_j,a_k$ have a distribution which is approximately uniform across values near to $e_i,e_j,e_k$.

\begin{lemma}\label{ijkrange}
	Fix positions $i,j,k$ and fix $\ell$ such that  $2\sqrt{n} < \ell < \frac{n}{3}$. Set $T \in [\ell^2,\ell^2+4n]$ such that $T-\floor{\frac{T}{2n}}m \equiv 0 \mod 2n-m+1$. Fix any $\beta_i,\beta_j,\beta_k \in \mathbb{Z} \cap \left( - \frac{\ell}{\sqrt{n}}, \frac{\ell}{\sqrt{n}} \right)$. Let $b_i = \beta_i + \floor*{\frac{T}{2n}}$ and let $b_j,b_k$ be defined similarly. Now let $\Delta_i,\Delta_j,\Delta_k$ be iid uniformly chosen from $(-6\ell,6\ell) \cap \mathbb{Z}$. Let $R_i,R_j,R_k$ be the number of Heads flipped by $i,j,k$ respectively when in position $m$ in the first $T$ steps. Let
	\begin{align*}
		\omega_i &\equiv p(i)-\beta_i m + \Delta_i \ (\text{mod } 2n-m+1), \\
		\omega_j &\equiv p(j)-\beta_j m + \Delta_j \ (\text{mod } 2n-m+1), \\
		\omega_k &\equiv p(k)-\beta_k m + \Delta_k \ (\text{mod } 2n-m+1).
	\end{align*}
	Then,
	\begin{equation*}
		\probp{(R_i,R_j,R_k) = (b_i,b_j,b_k), (p(i_T),p(j_T),p(k_T)) = (\omega_i,\omega_j,\omega_j)} \geq  3 \cdot 10^{-20} \cdot \frac{n^\frac{3}{2}}{\ell^6}.
	\end{equation*}
\end{lemma}

\begin{proof}
	By Lemma \ref{bigcoinsnoise} we know that with probability at least 
	$$
	6 \cdot 10^{-17} \cdot \frac{n^\frac{3}{2}}{\ell^3}
	$$
	we have
	\begin{align}
		p(i_T) &\equiv p(i)+T- b_im + \delta_i \\
		&\equiv p(i) + T - \floor*{\frac{T}{2n}}m - \beta_im + \delta_i \\
		&\equiv p(i) - \beta_im + \delta_i
	\end{align}
	for some $\delta_i \in (-6\ell,6\ell)$ and equivalent statements for $j,k$. Since there are $12\ell$ values in this range that $\Delta_i$ might take, the probability that $\Delta_i$ takes the ``correct'' one is $\frac{1}{12\ell}$. The same argument applies for $j$ and $k$. So,
	\begin{equation}
		\probp{(R_i,R_j,R_k) = (b_i,b_j,b_k), (p(i_T),p(j_T),p(k_T)) = (\omega_i,\omega_j,\omega_j)} > 6 \cdot 10^{-17} \cdot \frac{n^\frac{3}{2}}{\ell^3} \cdot \frac{1}{12^3} \cdot \frac{1}{\ell^3}.
	\end{equation}
\end{proof}

\begin{corollary}\label{needtoreword}
	Fix $i,j,k,\ell,T,\beta_i,\beta_j,\beta_k,b_i,b_j,b_k$ as in Lemma \ref{ijkrange}. As before let $R_i,R_j,R_k$ be the number of Heads flipped by $i,j,k$ respectively when in position $m$ in $T$ steps. Now let $Z_i,Z_j,Z_k$ be uniformly chosen from all positions such that
	\begin{align*}
		|p(Z_i)-p(i)+\beta_i m|_M &< 6\ell, \\
		|p(Z_j)-p(j)+\beta_j m|_M &< 6\ell, \\
		|p(Z_k)-p(k)+\beta_k m|_M &< 6\ell. \\
	\end{align*}
	Then,
	\begin{equation}
		\probp{(R_i,R_j,R_k) = (b_i,b_j,b_k), (i_T,j_T,k_T) = (Z_i,Z_j,Z_j)} \geq 3 \cdot 10^{-20} \cdot \frac{n^\frac{3}{2}}{\ell^6}.
	\end{equation}
\end{corollary}

\begin{proof}
	In the bound described by Lemma \ref{ijkrange} we neglect to use the fact that only ``valid'' choices of $\delta_i,\delta_j,\delta_k$, which allow for $\omega_i,\omega_j,\omega_k$ to be in the image of $p$ have a nonzero chance of being the locations of $i,j,k$ after $T$ steps. For example, depending on the parameters, we might randomly choose $\omega_i = m+1$. But $m+1$ is not associated with any position in the deck, as $p(m) = m$ and $p(m+1) = m+2$. By conditioning on the probability $1$ event that $i_T,j_T,k_T$ are associated with true positions (in the image of $p$) we have this Corollary.
\end{proof}

We can reword the statement of Corollary \ref{needtoreword} to immediately get the following Corollary:

\begin{corollary}\label{ijknoiseinverse}
	Fix $i,j,k,\ell,T,\beta_i,\beta_j,\beta_k,b_i,b_j,b_k$ as in Lemma \ref{ijkrange}. Let $Z_i,Z_j,Z_k$ be chosen uniformly from all positions such that
	\begin{align*}
		|p(Z_i)-p(i)+\beta_i m|_M &< 6\ell, \\
		|p(Z_j)-p(j)+\beta_j m|_M &< 6\ell, \\
		|p(Z_k)-p(k)+\beta_k m|_M &< 6\ell. \\
	\end{align*}
	Let $c_i,c_j,c_k$ be the cards which, after $T$ steps end up in positions $Z_i,Z_j,Z_k$. Let $R_i,R_j,R_k$ be the number of Heads flipped in these $T$ steps while $c_i,c_j,c_k$ are in position $m$. Then,
	\begin{equation*}
		\probp{(R_i,R_j,R_k) = (b_i,b_j,b_k), (c_i,c_j,c_k) = (i,j,k)} \geq 3 \cdot 10^{-20} \cdot \frac{n^\frac{3}{2}}{\ell^6}.
	\end{equation*}
\end{corollary}

This leads to the following lemma:

\begin{lemma}\label{ijkuniform}
	Fix $i,j,k,\ell,T,\beta_i,\beta_j,\beta_k,b_i,b_j,b_k$ as in Lemma \ref{ijkrange}. Fix any cards $\alpha_i,\alpha_j,\alpha_k$ such that
	\begin{align*}
		|(p(\alpha_i) - \beta_i m) - p(i) |_M &< 2\ell, \\
		|(p(\alpha_j) - \beta_j m) - p(j)|_M &< 2\ell, \\
		|(p(\alpha_k) - \beta_k m) - p(k)|_M &< 2\ell.
	\end{align*}
	Now let $Z_i',Z_j',Z_k'$ be uniformly chosen from all positions such that
	\begin{align*}
		|p(Z_i')-p(i)|_M &< 8\ell, \\
		|p(Z_j')-p(j)|_M &< 8\ell, \\
		|p(Z_k')-p(k)|_M &< 8\ell. \\
	\end{align*}
	Let $c_i',c_j',c_k'$ be the cards which, after $T$ steps end up in positions $Z_i',Z_j',Z_k'$. Let $R_i,R_j,R_k$ be the number of Heads which $c_i',c_j',c_k'$ flip in these $T$ steps while in position $m$. Then,
	\begin{equation*}
		\probp{(R_i,R_j,R_k) = (b_i,b_j,b_k), (c_i',c_j',c_k') = (\alpha_i,\alpha_j,\alpha_k)} \geq 10^{-21} \cdot \frac{n^\frac{3}{2}}{\ell^6}.
	\end{equation*}
\end{lemma}

\begin{proof}
	Let $\Lambda_i$ be the set of all $\zeta$ such that that $|p(\zeta)-p(i)|_M < 8\ell$. Let $\Gamma_i$ be the set of all $\omega$ such that $|(p(\alpha_i) -\beta_i m ) -p (\omega)|_M < 6\ell$. Note that $\Gamma_i$ is a subset of $\Lambda_i$ because $|(p(\alpha_i) - \beta_i m) - p(i)|_M< 2\ell$. Since $|p(z+1)-p(z)|_M \in \{1,2\}$ for all $z$ we see that $|\Gamma_i| \geq 3\ell$ and $|\Lambda_i| \leq 8\ell$. Since $Z_i'$ is chosen uniformly from $\Lambda_i$ we get
	\begin{equation}
		\probp{Z_i' \in \Gamma_i} \geq \frac{3}{8}.
	\end{equation}
	Define $\Gamma_j,\Gamma_k$ similarly for $\alpha_j,\alpha_k$. Since $Z_i',Z_j',Z_k'$ are chosen independently we have
	\begin{equation}
		\probp{Z_i' \in \Gamma_i, Z_j' \in \Gamma_j, Z_k' \in \Gamma_k} \geq \parens*{\frac{3}{8}}^3.
	\end{equation}
	Conditioning on $Z_i' \in \Gamma_i, Z_j' \in \Gamma_j, Z_k' \in \Gamma_k$ we can apply Corollary \ref{ijknoiseinverse} to get
	\begin{equation}
		\probp{(R_i,R_j,R_k) = (b_i,b_j,b_k), (c_i,c_j,c_k) = (\alpha_i,\alpha_j,\alpha_k) \ | \ Z_i' \in \Gamma_i, Z_j' \in \Gamma_j, Z_k' \in \Gamma_k} \geq 3 \cdot 10^{-20} \cdot \frac{n^\frac{3}{2}}{\ell^6}.
	\end{equation}
	So,
	\begin{equation}
		\probp{(R_i,R_j,R_k) = (b_i,b_j,b_k), (c_i,c_j,c_k) = (\alpha_i,\alpha_j,\alpha_k)} \geq \parens*{\frac{3}{8}}^3 \cdot 3 \cdot 10^{-20} \cdot \frac{n^\frac{3}{2}}{\ell^6}.
	\end{equation}
\end{proof}

Recall that our goal with these lemmas is to show that certain cards end up in an approximately uniform distribution. To this point we have shown that the triplets of cards we are interested in end up in ``reasonable'' positions with probability greater than or equal to
\begin{equation}
	(\text{constant}) \cdot \left(\frac{\sqrt{n}}{\ell} \cdot \frac{1}{\ell}\right)^3. \label{pregammabound}
\end{equation}
This is good because it matches the $2\frac{\ell}{\sqrt{n}}$ choices for $\beta_i$ and the $2\ell$ to $4\ell$ outcomes for $\alpha_i$. However there is one problem. Consider the case of $m = \frac{n}{2}$. In this case $3m \equiv 1 \mod 2n-m+1$. This means that, for example, if we for example choose $\beta_i = 2$ the positions nearby $p(i)-\beta_i m$ will be almost the same set of positions as if we picked $\beta_i = 5$ or $\beta_i = -1$ etc. In the $m = \frac{n}{2}$ case we only have a constant times $\ell^3$ possible values for $\alpha_i,\alpha_j,\alpha_k$ across all choices of $\beta_i,\beta_j,\beta_k$. This does not match our probability in line (\refeq{pregammabound}). \\
\\
To remedy this we introduce a constant $\gamma$ which depends on $n,m$ and $\ell$. It will give us the ability to increase the probability beyond that of (\refeq{pregammabound}) in cases like $m = \frac{n}{2}$.

\begin{definition}
	Let $\gamma = \gamma(\ell,n,m)$ be defined as follows:
	\begin{equation*}
		\gamma = \left| \left\{ \kappa \in \mathbb{N} \ : \ |\kappa m|_M < \ell, \ \kappa < \frac{\ell}{\sqrt{n}}\right\}\right|.
	\end{equation*}
\end{definition}

Note that $\kappa = 0$ is always an element of the set and so $\gamma \geq 1$. It will be important to be able to use $\gamma$ to count how many cards there are with norms less than or equal to $\ell$, so we prove the following Lemma.

\begin{lemma}\label{gammaovercount}
	Let $N_\ell$ be defined by
	\begin{equation}
		N_\ell = \left\{\text{positions } \omega \text{ such that } \omega \equiv a + bm \mod 2n-m+1 \text{ where } |a| \leq \ell, \ |b| \leq \frac{\ell}{\sqrt{n}}\right\}.
	\end{equation}
	Then,
	\begin{equation}
		|N_\ell| \geq \frac{\ell^2}{2 \gamma \sqrt{n}}
	\end{equation}
\end{lemma}

\begin{proof}
	To bound $|N_\ell|$ from below, we consider a smaller set $\tilde{N}_\ell$ which is defined by
    \begin{equation}
		\tilde{N}_\ell = \left\{\text{positions } \omega \text{ such that } \omega \equiv a + bm \mod 2n-m+1 \text{ where } 0 \leq a \leq \ell, \ 0 \leq b \leq \frac{\ell}{\sqrt{n}}\right\}.
	\end{equation}
	\\
    Since $\tilde{N}_\ell \subset N_\ell$ we have $|N_\ell| \geq |\tilde{N}_\ell|$. Let,
    \begin{equation}
        \mathcal{K}_\ell = \left\{\omega \in \{1,\dots,2n-m\} \text{ such that } \omega \equiv a + bm \mod 2n-m+1 \text{ where } 0 \leq a \leq \ell, \ 0 \leq b \leq \frac{\ell}{\sqrt{n}}\right\}
    \end{equation}
    Basically, let $\mathcal{K}_\ell$ be the superset of $N_\ell'$ which includes values of $\omega$ that do not refer to positions in the deck. There are $\ell$ choices for $a$ in $\mathcal{K}_\ell$, and $\frac{\ell}{\sqrt{n}}$ choices for $b$. However, we may have some choices $a,b$ and $a',b'$ such that $a+bm \equiv a'+b'm$. Without loss of generality, assume $b' > b$. Then,
    \begin{equation}
        (b'-b)m \equiv a'-a \mod 2n-m+1
    \end{equation}
	and $|a'-a|_M < \ell$. This means that $b'-b$ is associated with some $\kappa \in \mathbb{N}$ such that $|\kappa m|_M < \ell$ and $\kappa < \frac{\ell}{\sqrt{n}}$. We know there are at most $\gamma$ such $\kappa$. Thus, we can divide the $\ell \cdot \frac{\ell}{\sqrt{n}}$ choices for $(a,b)$ into equivalence classes of size at most $\gamma$. This tells us that,
    \begin{equation}
        |\mathcal{K}_\ell| \geq \ell \cdot \frac{\ell}{\sqrt{n}} \cdot \frac{1}{\gamma} = \frac{\ell^2}{\gamma \sqrt{n}}
    \end{equation}
    Recall that $\tilde{N}_\ell$ is the subset of $\mathcal{K}_\ell$ that only includes $\omega \in (1,\dots,2n-m+1)$ such that $p(a) = \omega$ for some $a \in (1,\dots,n)$. Since all $\omega \leq m$ have this property and every other $\omega > m$ has it we get
	\begin{equation}
		|\tilde{N}_\ell| \geq \frac{1}{2}|\mathcal{K}_\ell|.
	\end{equation}
	So we have,
    \begin{equation}
        |N_\ell| \geq |\tilde{N}_\ell| \geq \frac{1}{2}|\mathcal{K}_\ell| \geq \frac{\ell^2}{2\gamma \sqrt{n}}.
    \end{equation}
\end{proof}
\noindent
Now that we have defined $\gamma$ we will incorporate it into our probability bound.
\begin{lemma}\label{ijkgamma}
	Fix positions $i,j,k$ and fix $\ell$ such that  $2\sqrt{n} < \ell < \frac{n}{3}$. Set $T \in [\ell^2,\ell^2+4n]$ such that $T-\floor{\frac{T}{2n}}m \equiv 0 \mod 2n-m+1$. Fix any $\beta_i',\beta_j',\beta_k' \in \mathbb{Z} \cap \left( - \frac{\ell}{\sqrt{n}}, \frac{\ell}{\sqrt{n}} \right)$. Fix any cards $\alpha_i',\alpha_j',\alpha_k'$ such that
	\begin{align*}
		|(p(\alpha_i')-\beta_i'm) - p(i)|_M &< \ell, \\
		|(p(\alpha_j')-\beta_j'm) - p(j)|_M &< \ell, \\
		|(p(\alpha_k')- \beta_k'm) - p(k)|_M &< \ell.
	\end{align*}
	Now let $\mathcal{Z}_i,\mathcal{Z}_j,\mathcal{Z}_k$ be uniformly chosen from all positions such that
	\begin{align*}
		|p(Z_i')-p(i)|_M &< 8\ell, \\
		|p(Z_j')-p(j)|_M &< 8\ell, \\
		|p(Z_k')-p(k)|_M &< 8\ell.
	\end{align*}
	Let $\mathcal{C}_i,\mathcal{C}_j,\mathcal{C}_k$ be the cards which, after $T$ steps end up in positions $Z_i',Z_j',Z_k'$. Then,
	\begin{equation*}
		\probp{(\mathcal{C}_i,\mathcal{C}_j,\mathcal{C}_k) = (\alpha_i',\alpha_j',\alpha'_k)} \geq 10^{-21} \cdot \frac{\gamma^3 n^\frac{3}{2}}{\ell^6}.
	\end{equation*}
\end{lemma}

\begin{proof}
	Fix any $\kappa_i \in \mathbb{N}$ such that $|\kappa_i m|_M < \ell$ and $\kappa_i < \frac{\ell}{\sqrt{n}}$. If $\beta_i' \geq 0$ then let $\beta_i = \beta_i' - \kappa_i$ and if $\beta_i' < 0$ then let $\beta_i = \beta_i' + \kappa_i$. Note that this ensures $\beta_i \in \left( - \frac{\ell}{\sqrt{n}}, \frac{\ell}{\sqrt{n}} \right)$. Now because $|\kappa_i m|_M < \ell$ we have
	\begin{align}
		|(p(\alpha_i') - \beta_i m) - p(i)|_M &= |(p(\alpha_i') - \beta_i' m) \pm \kappa_i m - p(i)|_M \\
        &\leq |(p(\alpha_i') - \beta_i' m) - p(i)|_M + |\kappa_i m|_M \\
        &< 2\ell
	\end{align}
	Now fix some $\kappa_j,\kappa_k$ from the same subset of $\mathbb{N}$ as $\kappa_i$ was chosen from. Then we can use Lemma \ref{ijkuniform} to show that the probability of
	\begin{equation}
		(R_i,R_j,R_k) = \parens*{\beta_i-\floor*{\frac{T}{2n}},\beta_j-\floor*{\frac{T}{2n}},\beta_k-\floor*{\frac{T}{2n}}} \text{ and } (\mathcal{C}_i,\mathcal{C}_j,\mathcal{C}_k) = (\alpha_i',\alpha_j',\alpha_k')
	\end{equation}
	is at least 
	$$10^{-21} \cdot \frac{n^\frac{3}{2}}{\ell^6}.$$
    Recall that $\beta_i = \beta_i' \pm \kappa_i$ and $\beta_j = \beta_j' \pm \kappa_j$ and $\beta_k = \beta_k' \pm \kappa_k$ for any of the $\gamma^3$ choices for $(\kappa_i,\kappa_j,\kappa_k)$. Summing over all $\gamma^3$ choices gives us
	\begin{equation}
		\probp{(\mathcal{C}_i,\mathcal{C}_j,\mathcal{C}_k) = (\alpha_i',\alpha_j',\alpha_k')} \geq \gamma^3 \cdot 10^{-21} \cdot \frac{n^\frac{3}{2}}{\ell^6}.
	\end{equation}
\end{proof}

Be rewording Lemma \ref{ijkgamma} the following Corollary is immediate:
\begin{corollary}\label{ijkreverse}
	Fix positions $i,j,k$ and fix $\ell$ such that  $2\sqrt{n} < \ell < \frac{n}{3}$. Set $T \in [\ell^2,\ell^2+4n]$ such that $T-\floor{\frac{T}{2n}}m \equiv 0 \mod 2n-m+1$. Fix any $\beta_i',\beta_j',\beta_k' \in \mathbb{Z} \cap \left( - \frac{\ell}{\sqrt{n}}, \frac{\ell}{\sqrt{n}} \right)$. Fix any positions $f_i,f_j,f_k$ such that
	\begin{align*}
		|(p(f_i)-\beta_i'm) - p(i)|_M &< \ell, \\
		|(p(f_j)-\beta_j'm) - p(j)|_M &< \ell, \\
		|(p(f_k)-\beta_k'm) - p(k)|_M &< \ell.
	\end{align*}
	Now let $i',j',k'$ be uniformly chosen from all cards such that
	\begin{align*}
		|p(i')-p(i)|_M &< 8\ell, \\
		|p(j')-p(j)|_M &< 8\ell, \\
		|p(k')-p(k)|_M &< 8\ell.
	\end{align*}
	Let $i_{(-T)}',j_{(-T)}',k_{(-T)}'$ be the locations of $i',j',k'$ respectively after $T$ steps of the inverse overlapping cycles shuffle. Then,
	\begin{equation*}
		\probp{(i_{(-T)}',j_{(-T)}',k_{(-T)}') = (f_i,f_j,f_k)} \geq 10^{-21} \cdot \frac{\gamma^3 n^\frac{3}{2}}{\ell^6}.
	\end{equation*}
\end{corollary}

We now exploit the similarity between the inverse overlapping cycles shuffle and the normal forwards overlapping cycles shuffle to get the equivalent statement for $T$ positive steps.
\begin{proposition}\label{montegamma}
	Fix positions $i,j,k$ and fix $\ell$ such that  $2\sqrt{n} < \ell < \frac{n}{3}$. Set $T \in [\ell^2,\ell^2+4n]$ such that $T-\floor{\frac{T}{2n}}m \equiv 0 \mod 2n-m+1$. Fix any $\beta_i,\beta_j,\beta_k \in \mathbb{Z} \cap \left( - \frac{\ell}{\sqrt{n}}, \frac{\ell}{\sqrt{n}} \right)$. Fix any positions $f_i,f_j,f_k$ such that
	\begin{align*}
		|p(f_i) - p(i) + \beta_i m|_M &< \ell, \\
		|p(f_j) - p(j) + \beta_j m|_M &< \ell, \\
		|p(f_k) - p(k) + \beta_k m|_M &< \ell.
	\end{align*}
	Now let $i',j',k'$ be uniformly chosen from all cards such that
	\begin{align*}
		|p(i')-p(i)|_M &< 8\ell, \\
		|p(j')-p(j)|_M &< 8\ell, \\
		|p(k')-p(k)|_M &< 8\ell.
	\end{align*}
	Then,
	\begin{equation*}
		\probp{(i_T',j_T',k_T') = (f_i,f_j,f_k)} \geq 10^{-21} \cdot \frac{\gamma^3 n^\frac{3}{2}}{\ell^6}.
	\end{equation*}
\end{proposition}

\begin{proof}
	This follows from Corollary \ref{ijkreverse} and  Corollary \ref{inversenorm}. Corollary \ref{ijkreverse} describes the distribution of randomly chosen $i',j',k'$ after $T$ steps of the inverse overlapping cycles shuffle. In particular it says that if $\beta_i,\beta_j,\beta_k$ are appropriately fixed and $i',j',k'$ are uniformly sampled so that
	\begin{align*}
		|p(i')-p(i)|_M < 8\ell, \\
		|p(j')-p(j)|_M < 8\ell, \\
		|p(k')-p(k)|_M < 8\ell,
	\end{align*}
	then for any fixed positions $f_i,f_j,f_k$ such that
	\begin{align*}
		|p(f_i) - p(i) - \beta_i m|_M &< \ell, \\
		|p(f_j) - p(j) - \beta_j m|_M &< \ell, \\
		|p(f_k) - p(k) - \beta_k m|_M &< \ell.
	\end{align*}
	we have that $i',j',k'$ are sufficiently likely to travel to $f_i,f_j,f_k$ in $T$ steps of the inverse overlapping cycles shuffle. By Theorem \ref{inverseshuffle} we know that the inverse overlapping cycles shuffle is the same as the forward overlapping cycles shuffle after a reordering of the cards, and by Proposition \ref{inversedist} we know that $p(a)-p(b)$ becomes $p(b)-p(a)$ under this reordering. So, the equivalent result to Corollary \ref{ijkreverse} for the (forwards) overlapping cycles shuffle will hold if $i',j',k'$ are uniformly sampled so that
 \begin{align*}
		|p(i)-p(i')|_M < 8\ell, \\
		|p(j)-p(j')|_M < 8\ell, \\
		|p(k)-p(k')|_M < 8\ell,
	\end{align*}
    and $f_i,f_j,f_k$ are positions such that
	\begin{align*}
		|p(i) - p(f_i) - \beta_i m|_M &< \ell, \\
		|p(j) - p(f_j) - \beta_j m|_M &< \ell, \\
		|p(k) - p(f_k) - \beta_k m|_M &< \ell.
	\end{align*}
    We use the fact that $|-a|_M = |a|_M$ to finish the proof.
\end{proof}
We have shown that if we choose random cards nearby $i,j,k$ then the distribution of these random cards will be close to uniform over nearby positions. But we really care about the distribution of $i,j,k$ themselves, not some randomly chosen neighbors. We now use a coupling argument to show that the distribution of cards $i,j,k$ themselves will also be approximately uniform. The basic idea is that we choose random neighbors $i',j',k'$ and then show that under a certain coupling that $i,j,k$ will couple with $i',j',k'$ with probability bounded away from 0. The argument will require $i,j,k$ to be spread out from each other, which is why we require Stage 1 and Stage 3. \\
\\
To start with, we describe the coupling. \\
\\
Fix any cards $i,j,k$ such that $\norm{i-j},\norm{i-k},\norm{j-k} > 199\ell$. Let $T = \ell^2$. Let $i',j',k'$ be chosen uniformly from cards such that $|p(i)-p(i')|_M,|p(j)-p(j')|_M,|p(k)-p(k')|_M < 8\ell$. We now run two overlapping cycles shuffles $(\pi_t)$ and $(\pi_t')$ and we track $i,j,k$ in $(\pi_t)$ and $i',j',k'$ in $(\pi_t')$. We couple the two shuffle as follows:
\begin{itemize}
	\item Generate coin sequences $B^i,B^j,B^k$. Have $(\pi_t')$ draw from $B^i$ or $B^j$ or $B^k$ whenever $i'$ or $j'$ or $k'$ is in position $m$. Similarly have $(\pi_t)$ draw from $B^i$ or $B^j$ or $B^k$ whenever $i$ or $j$ or $k$ is in position $m$ with the following exceptions:
	\begin{itemize}
		\item If $i' \leq m < i$ then have $i$ skip $B^i_1$ and draw from $B^i_2$ on its first visit to $m$ and $B^i_3$ on its second visit etc.
		\item If $i \leq m < i'$ then have $i'$ skip $B^i_1$ and draw from $B^i_2$ on its first visit to $m$ and $B^i_3$ on its second visit etc.
		\item Have the equivalent exceptions for $j$ and $k$.
	\end{itemize}
	This will ensure that $i,j,k$ follow the same choice of big coins as their counterparts. For example, imagine that $i_0 = m-3$ and $i_0' = m-1$. Then after one step we have $i_1 = m-2$ and $i_1' = m$. Now if $B^i_1$ is a Heads, then $\pi'$ will flip Heads on its second step. So $i_2 = m-1$ and $i_2' = 1$. One more step and we get $i_3 = m$ and $i'_3 = 2$. Note that this is the first time $i$ reaches position $m$, so now $\pi$ must use $B^i_1$ which is Heads. So $i_4 = 1$ and $i_4' = 3$. In this way, $i$ and $i'$ will always follow the same trajectory of ``big'' coins and $i$ will ``follow behind'' $i'$. The reason we have the aforementioned exceptions is because if, for example, $j_0 = m-1$ and $j_0' = m+1$, then we want to wait to synchronize the draws from $B^j$ until after $j$ and $j'$ are in the same part of the deck. That way, as long as $B^j_1$ is a Tails (which happens with probability $\frac{1}{2}$) we will get that $j$ follows behind $j'$ as they will both start drawing from $B^j_2$ once they cycle back to position $m$.
	\item To determine the movement of $i',j',k'$ when none of these cards are in position $m$, generate sequences of coins $S^i,S^j,S^k$. Whenever $i'$ is in the bottom part of the deck (and neither $j'$ nor $k'$ is in position $m$) use a coin from $S^i$ for $(\pi_t')$. Whenever $i'$ is in the top part of the deck and $j'$ is in the bottom part of the deck (and $k'$ is not in position $m$) use a coin from $S^j$. Whenever $i'$ and $j'$ are in the top part of the deck and $k'$ is in the bottom part of the deck, use a coin from $S^k$. You can think of $S^j$ having ``priority'' over $S^k$ and $S^i$ having priority over both $S^j$ and $S^k$. Think of $B^i,B^j,B^k$ as all having priority over $S^i,S^j,S^k$.
	\item The movement of $i,j,k$ when none of these cards is in position $m$ is broken into four phases. Generate sequences of coins $X^i,X^j,X^k,X^{ij},X^{ik},X^{jk},X^{ijk},Y^j,Y^k,Y^{jk},Z^k$.
	\begin{itemize}
		\item At the start of the process, we say we are in ``Phase 1''. In this phase, if none of $i,j,k$ are in position $m$, use the next coin from $X^A$ to determine the movement of $i,j,k$ where $A$ is the set of exactly which of $i,j,k$ are in the bottom part of the deck. Let $\tau_1$ be random stopping time which is the minimum $t$ such that $i_t = i'_t + \Delta(i,t)$ where $\Delta(i,t)$ is a small random variable (very likely in $\{-1,0,1\}$) which we will define later. You can imagine $\tau_1$ is more or less when $i_t = i_t'$, and the reason that we add $\Delta(i,t)$ is a technicality which we will explain later. At time $\tau_1$ we move to Phase 2, and we change the rules for how $i,j,k$ move in order to couple $i$ to $i'$.
		\item At time $\tau_1$, suppose $r_i$ is the number of coins $i'$ has drawn from $S^i$. Let $\kappa^i$ to be such that $\kappa^i_s = S^i_{r_i+s}$. Now whenever $i$ is in the bottom part of the deck, draw the next coin from $\kappa^i$. In other words, have $i$ start following the same sequence of coins that $i'$ uses. Since $i$ and $i'$ will both draw from the same sequence of coins when they are in the bottom part of the deck, we know $i$ and $i'$ will stay coupled together, except for a small technicality which we will explain later. When $i$ is in the top part of the deck, use the next coin from $Y^j$ or $Y^k$ or $Y^{jk}$ depending on if $j$ is in the bottom of the deck or $k$ is or both. Let $\tau_2$ be the random stopping time which is the minimum $t \geq \tau_1$ such that $j_t = j'_t + \Delta(j,t)$ where $\Delta(j,t)$ is small random variable we will define later. At time $\tau_2$ we move to Phase 3, and change how $i,j,k$ move in order to couple $(i,j)$ to $(i',j')$.
		\item At time $\tau_2$, suppose $r_j$ is the number of coins $j'$ has drawn from $S^j$. Define the sequence $\kappa^j$ to be such that $\kappa^j_s = S^j_{r_j+s}$. Now whenever $i$ is in the top part of the deck and $j$ is in the bottom part of the deck, draw the next coin from $\kappa^j$. Continue to draw from $\kappa_i$ whenever $i$ is in the bottom part of the deck, regardless of if $j$ is in the bottom part of the deck or not. In other words, give $\kappa^i$ priority over $\kappa^j$ so $j$ follows the same sequence of coins that $j'$ uses. Now we know (other than a small technicality which will be explained later) that both $i,i'$ and $j,j'$ will stay coupled together because the joint movement of $(i,j)$ and $(i',j')$ follow the same rules. If  $k$ is alone in the bottom part of the deck, draw from $Z^k$. Let $\tau_3$ be the random stopping time which is the minimum $t \geq \tau_2$ such that $k_t = k'_t + \Delta(k,t)$ where once again $\Delta(k,t)$ will be defined later. At time $\tau_3$ we move the Phase 4, where $(i,j,k)$ is permanently coupled with $(i',j',k')$.
		\item At time $\tau_3$, suppose $r_k$ is the number of coins $k'$ has drawn from $S^k$. Define the sequence $\kappa^k$ to be such that $\kappa_s = S^k_{r_k+s}$. Now whenever $i$ and $j$ are in the top part of the deck and $k$ is in the bottom part of the deck, draw the next coin from $\kappa^k$. Now $(i,j,k)$ obey the same rules as $(i',j',k')$, so they will stay coupled forever.
	\end{itemize}
\end{itemize}
The technicality that could cause $i$ to ``decouple'' from $i'$ after time $\tau_1$ is the fact that $i$ is obligated to use coins from $B^j$ or $B^k$ whenever $j$ or $k$ are in position $m$, whereas $i'$ uses coins from $B^j$ or $B^k$ whenever $j'$ or $k'$ are in position $m$. Since $j$ will be in position $m$ at different times than $j'$ during Phase 2, there will likely be times during Phase 2 where $i_t \neq i_t'$. However, as long as we make sure that $j_t$ and $j_t'$ stay closer to each other than to $i_t$ and $i_t'$ this will not be an issue. To see why, consider the situation in Phase 2 where
$$
\begin{pmatrix}
	i_t \\ j_t \\ i_t' \\ j_t'
\end{pmatrix} = \begin{pmatrix}
	m+100 \\ m \\ m+100 \\ m-3
\end{pmatrix}
$$
Assume that $k_t$ and $k_t'$ are far away from position $m$. Also assume that the upcoming coin in $B^j$ is a Tails and $i_t,i_t'$ are at coin number $r$ in sequence $S^i$. Finally assume
$$
(S_r^i,S_{r+1}^i,S_{r+2}^i,S_{r+3}^i) = (H,T,H,T).
$$
\begin{enumerate}
	\item In the first step $\pi$ draws from $B^j$ to get a Tails, and $\pi'$ draws $S_r^i$ from $S^i$ to get a Heads. So
	$$
	\begin{pmatrix}
		i_{t+1} \\ j_{t+1} \\ i_{t+1}' \\ j_{t+1}'
	\end{pmatrix} = \begin{pmatrix}
		m+101 \\ m+1 \\ m+100 \\ m-2
	\end{pmatrix}
	$$
	Now $i$ and $i'$ have decoupled! This looks bad, but lets see what happens over the course of the next few moves.
	\item Now $\pi$ draws from $S^i$ but it is a step behind where $\pi'$ is in $S^i$. So $\pi$ draws $S_r^i$ from $S^i$ to get a Heads. On the other hand $\pi'$ draws $S_{r+1}^i$ from $S^i$ to get a Tails. So now
	$$
	\begin{pmatrix}
		i_{t+2} \\ j_{t+2} \\ i_{t+2}' \\ j_{t+2}'
	\end{pmatrix} = \begin{pmatrix}
		m+101 \\ m+1 \\ m+101 \\ m-1
	\end{pmatrix}
	$$
	It may look like everything is fixed now, but remember $\pi$ and $\pi'$ are at different points in the sequence $S^i$. This is going to cause trouble the next step. 
	\item Now $\pi$ draws $S_{r+1}^i$ from $S^i$ to get a Tails and $\pi'$ draws $S_{r+2}^i$ from $S^i$ to get a Heads. So 
	$$
	\begin{pmatrix}
		i_{t+3} \\ j_{t+3} \\ i_{t+3}' \\ j_{t+3}'
	\end{pmatrix} = \begin{pmatrix}
		m+102 \\ m+2 \\ m+101 \\ m
	\end{pmatrix}
	$$
	Now that $j'$ is in position $m$ everything will be fixed in the next step.
	\item We know $\pi$ will draw $S_{r+2}^i$ from $S^i$ to get a Heads. But now $\pi'$ will draw a Tails from $B^j$, the ``same'' Tails that $\pi$ drew three steps earlier. So now
	$$
	\begin{pmatrix}
		i_{t+4} \\ j_{t+4} \\ i_{t+4}' \\ j_{t+4}'
	\end{pmatrix} = \begin{pmatrix}
		m+102 \\ m+2 \\ m+102 \\ m+1
	\end{pmatrix}
	$$
	We see that $i$ and $i'$ have re-synchronized. Not only are they in the same position again, but $\pi$ and $\pi'$ are in the same place in $S^i$. This means that they will stick together until another discrepancy from $B^j$ or $B^k$ causes them to split. For example, lets look at one more step.
	\item Now $\pi$ and $\pi'$ BOTH draw $S_{r+3}^i$ from $S^i$ to get a Tails. So,
	$$
	\begin{pmatrix}
		i_{t+5} \\ j_{t+5} \\ i_{t+5}' \\ j_{t+5}'
	\end{pmatrix} = \begin{pmatrix}
		m+103 \\ m+3 \\ m+103 \\ m+2
	\end{pmatrix}
	$$
\end{enumerate}

In short, in any situation where $j$ hits position $m$ before $j'$ hits position $m$ (or vice versa), the decoupling of $i$ from $i'$ will only last until $j'$ hits position $m$, as long as this re-synchronization happens before $i$ or $i'$ hit either position $m$ or $n$ themselves (moving into a different ``part'' of the deck with different rules and potentially causing more trouble). The reason for this, in short, is because if we look at all the steps in between when $j$ and when $j'$ hit position $m$, as long as $i$ and $i'$ both stay in the bottom part of the deck then they will both use the same number of coins from $S^i$ plus one of the same coin from $B^j$. If instead $i$ and $i'$ stay in the top of the deck, it is even easier: $i$ and $i'$ deterministically move down one position each step. If we can guarantee that
\begin{equation}
	\norm{p(i_t)-p(j_t)},\norm{p(i_t')-p(j_t')} \geq |j_t'-j_t|_M + \sqrt{n}
\end{equation}
for all $t$ then we can ensure that re-synchronization happens before $i$ or $i'$ hit position $m$ or $n$. This is because if, for example, $j_t' < m < j_t$ then we know $i_t,i_t' \neq m$ because otherwise we would have 
\begin{align}
	&\norm{p(i_t)-p(j_t)} < |j_t'-j_t|_M \\
	\text{  or  } &\norm{p(i_t')-p(j_t')} < |j_t'-j_t|_M.
\end{align}
We also know that $i_t,i_t' \neq n$ because otherwise we would have
\begin{align}
	&\norm{p(i_t) - p(j_t)} \leq \norm{p(n) - p(m)} + |m - j_t| < \sqrt{n} + |j_t'-j_t|_M \\
	\text{  or  } &\norm{p(i_t') - p(j_t')} \leq \norm{p(n) - p(m)} + |m - j_t| < \sqrt{n} + |j_t'-j_t|_M.
\end{align}
The equivalent logic works comparing the distances between $i,i'$ to $i,k$ and $i',k'$ as well as comparing the distances between $j,j'$ to $j,k$ and $k,k'$. In short, as long as we keep each card closer to its counterpart than to the other tracked cards, we don't need to worry about desynchronizaiton. \\
\\
This same reasoning is how we define $\Delta(i,t),\Delta(j,t)$ and $\Delta(k,t)$. It would be a mistake to say $\tau_1$ is always when $i_t = i_t'$. This is because, for example, if $i_t = i_t'$ during some time $t$ when $j_t < m < j_t'$, if we start Phase 2 and switch $i$ to start using coins from $S^i$ then $i$ and $i'$ will become decoupled after $j$ reaches position $m$. So, we define $\Delta(i,t)$ to be such that if $j,j'$ or $k,k'$ straddle position $m$ at time $t$ then if we start Phase 2 when $i_t = i_t' + \Delta(i,t)$ we will have $i_r = i_r'$ AND $i_r,i_r'$ at the same point in the sequence $S^i$ at the soonest time $r > t$ when $j,j'$ or $k,k'$ stop straddling $m$. Define $\Delta(j,t)$ and $\Delta(k,t)$ similarly for $j$ and $k$. \\
\\
By our previous discussion, we know that $|\Delta(i,t)|,|\Delta(j,t)|,|\Delta(k,t)| \leq 1$ for all $t$ as long as $i,j,k$ and $i',j',k'$ stay closer to their counterparts than to each other. This will help us ensure that $i,j,k$ couple to $i',j',k'$. If for example $p(i_0) = p(i_0') - \ell$ then as previously explained $i$ will ``follow behind'' $i'$. If we can show that at some point, due to a surplus of Tails flipped by $i$ in the bottom part of the deck, that $i$ ``passes'' $i'$ then $i$ must hit $i'-1,i',i'+1$ on the way by and is therefore guaranteed to couple with $i'$. The same idea is used to show $j$ couples with $j'$ and $k$ couples with $k'$. \\
\\
Now that we have described the coupling, it is time to use it in the main proposition of this section.

\begin{proposition}\label{montecoupled}
	Consider the overlapping cycles shuffle on $n$ cards where $\frac{m}{n} \in (\epsilon, 1 - \epsilon)$. There exist constants $C,D$ which depend on $\epsilon$ such that the following holds: Fix $\ell > C\sqrt{n}$. Suppose $i,j,k$ are cards such that $\norm{p(i)-p(j)},\norm{p(i)-p(k)},\norm{p(j)-p(k)} > 199\ell$. Let $T \in [\ell^2,\ell^2+4n]$ such that $T - \floor*{\frac{T}{2n}}m \equiv 0 \mod 2n-m+1$. Define the set 
	\begin{equation*}
		N_i = \left\{ \omega \text{ such that } p(i) + a + bm = p(\omega) \mod 2n-m+1 \text{ with } |a| \leq \ell, \ |b| \leq \frac{\ell}{\sqrt{n}} \right\}.
	\end{equation*}
	Define $N_j$ and $N_k$ similarly for $j$ and $k$. Then,
	\begin{equation}
		\probp{i_T = f_i, j_T = f_j, k_T = f_k} \geq \frac{D\gamma^3 n^\frac{3}{2}}{\ell^6}
	\end{equation}
	for at least $\frac{3}{4}$ of all $(f_i,f_j,f_k)$ in $N_i \times N_j \times N_k$.
\end{proposition}
\begin{proof}
	Let $i',j',k'$ be chosen uniformly from positions such that $|p(i')-p(i)|_M,|p(j')-p(j)|_M,|p(k')-p(k)|_M < 8\ell$ as described by the previous coupling. Run two Overlapping Cycle shuffles $\pi$ and $\pi'$ tracking $i,j,k$ and $i',j',k'$ respectively and let $\pi$ be coupled to $\pi'$ as previously described. Let $L$ be the constant that depends on $\epsilon$ from Proposition \ref{ijkbottomC}. Now let $A$ be the event that the following is true:
	\begin{itemize}
		\item There exists $r_i,r_j,r_k,s_i,s_j,s_k \leq \frac{\ell^2}{3L}$ such that $\text{Diff}_{r_i}(X^i),\text{Diff}_{r_j}(Y^j),\text{Diff}_{r_k}(Z^k) = 34\ell$ and \\
		$\text{Diff}_{s_i}(X^i),\text{Diff}_{s_j}(Y^j),\text{Diff}_{s_k}(Z^k) = -34\ell$.
		\item $|\text{Diff}_{t}(X^i)|,|\text{Diff}_{t}(Y^j)|,|\text{Diff}_{t}(Z^k)| \leq 35\ell$ for all $t \leq T$.
	\end{itemize}
	In other words, $A$ is the event that when $i,j,k$ are alone in the bottom part of the deck in Phase 1, Phase 2, and Phase 3 respectively, they oscillate greatly so that they will be likely to pass by (and couple with) $i',j',k'$. Since we are controlling movement within a constant times $\sqrt{L}$ standard deviations, we see by Theorem \ref{invershoeffding} and Theorem \ref{hoeffdingmax} that
	\begin{equation}
		\probp{A} \geq D_1\exp(-D_2L)
	\end{equation}
	for some constants $D_1,D_2$. Note that the event $A$ is independent of the distribution of $i',j',k'$ since $A$ concerns coins which are not used to generate the paths of $i',j',k'$. We now show, conditioning on $A$, that $(i,j,k)$ is likely to couple to $(i',j',k')$. To see this, let $W^j$ be the record of coins used by $j$ before time $\tau_1$ while in the bottom part of the deck. Let $W^k$ be the record of coins used by $k$ before time $\tau_2$ while in the bottom part of the deck. Let $V^i,V^j,V^k$ be the records of coins used by $i',j',k'$ before time $T$ while in the bottom part of the deck. Let $G$ be the event that the following is true:
	\begin{itemize}
		\item $|\text{Diff}_t(W^j)|,|\text{Diff}_t(W^k)|,|\text{Diff}_t(V^i)|,|\text{Diff}_t(V^j)|,|\text{Diff}_t(V^k)| \leq 12\ell$ for all $t \leq T$.
	\end{itemize}
	Then by Theorem \ref{hoeffdingmax} we have that
	\begin{equation}
		\probp{G} \geq 1-5\cdot4\exp\parens*{-\frac{12^2}{2}} \geq 1 - 10^{-30}.
	\end{equation}
	Note that $G$ is independent of $A$ because $G$ is a record of coins separate from the coins that determine $A$. Let $Q$ be the event that $i$ spends at least $\frac{\ell^2}{3L}$ steps in the bottom part of the deck in the first $\frac{\ell^2}{3}$ steps of Phase 1, and $j$ spends at least $\frac{\ell^2}{3L}$ steps in the bottom part of the deck in the first $\frac{\ell^2}{3}$ steps of Phase 2, and $k$ spends at least $\frac{\ell^2}{3L}$ steps in the bottom part of the deck in the first $\frac{\ell^2}{3}$ steps of Phase 3. By Proposition \ref{ijkbottomC} we know that
	\begin{equation}
		\probp{Q} \geq 1 - 3\exp\left(-\frac{\ell^2}{6L}\right).
	\end{equation}
	Note that,
	\begin{align}
		\probp{Q \ | \ A} &= 1 - \probp{Q^C \ | \ A} \\
		&= 1 - \frac{\probp{Q^C,A}}{\probp{A}} \\
		&\geq 1 - \frac{\probp{Q^C}}{\probp{A}} \\
		&\geq 1 - \frac{3\exp\parens*{-\frac{\ell^2}{6L}}}{D_1\exp(-D_2 L)}.
	\end{align}
	Recall that $\ell > C\sqrt{n}$ for a constant $C$ which we have the freedom to choose. Choose $C$ large enough that,
 \begin{align}
     \probp{Q \ | \ A} &\geq 1 - \frac{3\exp\parens*{-\frac{C^2 n}{6L}}}{D_1\exp(-D_2 L)} \\
     &= 1 - \frac{3}{D_1} \parens*{D_2 L - \frac{1}{6} C^2 n} \\
     &\geq 1 - 10^{-30}.
 \end{align}
Note that on the events $A,G,Q$ we can show that $\tau_3 < T$. This is because $G$ forces $i',j',k'$ to stay within $12\ell$ steps of their expected position. On the other hand, $i$ starts out within distance $8\ell$ distance of $i'$ in phase 1, and for $j$ and $k$ the event $G$ means that $j$ and $k$ don't drift more than $12\ell$ steps further than this initial $8\ell$ due to $W^j,W^k$. Thus, the total distance $i,j,k$ are stretched from $i',j',k'$ respectively before we count $X^i,Y^j,Z^k$ is at most $8\ell + 12\ell + 12\ell = 32\ell$. Conditioning on $A$ and $Q$ we know that $X^i$ overcomes this gap in the first $\frac{\ell^2}{3L}$ coins, and that $X^i$ uses these coins in $\frac{\ell^3}{3}$ steps. So at some point in the first $\frac{\ell^2}{3}$ steps $i$ passes by $i'$ and couples to end Phase 1. By a similar argument Phase 2 and Phase 3 last no more than $\frac{\ell^2}{3}$ steps each. \\
	\\
	Finally, it we will need to ensure $i,j,k$ stay separated from each other so that $i,j,k$ don't decouple from $i',j',k'$. Let $E_i$ be the event that in between the $x$th time $i$ hits position $m$ and the $x$th time $i'$ hits position $m$, none of $j,j',k,k'$ hit position $m$ (where we don't count the first time $i'$ hits position $m$ if $i' \leq m < i$ and we don't count the first time $i$ hits position $m$ if $i \leq m < i'$). Let $E_j$ and $E_k$ be the equivalent events for $j,j'$ and $k,k'$. Now we want to show that $E_i,E_j,E_k$ are all likely, even when conditioning on $A$. To see why, let $U^i,U^j,U^k$ be the records of all coins used by $i,j,k$ while in the bottom part of the deck except those drawn from $X^i,Y^j,Z^k$. Let $V^i,V^j,V^k$ be the records of all coins used by $i',j',k'$ while in the bottom part of the deck. Then as long as
	\begin{equation}
		|\text{Diff}_t(U^i)|,|\text{Diff}_t(U^j)|,|\text{Diff}_t(U^k)|,|\text{Diff}_t(V^i)|,|\text{Diff}_t(V^j)|,|\text{Diff}_t(V^k)| \leq 12\ell \text{ for all } t \leq T
	\end{equation}
	and
	\begin{equation}
		|\text{Diff}_t(B^i)|,|\text{Diff}_t(B^k)|,|\text{Diff}_t(B^k)| \leq \frac{12\ell}{\sqrt{n}} \text{ for all } t \leq \frac{T}{n},
	\end{equation}
	then the gaps in norm between $i,j,k$ and the gaps in norm between $i',j',k'$ will never smaller than the gaps between $i,i'$ and $j,j'$ and $k,k'$. This is because the gaps between $i,j,k$ start out at more than $199\ell$ and the gaps between $i',j',k'$ start out more than $199\ell - 2 \cdot 8\ell$. Our restrictions on the $U^i,U^j,B^i,B^j,X^i,Y^j$ sequences mean that these gap between $i$ and $j$ can shrink at most to $199\ell - (12+12+12+12+35+35)\ell = 81\ell$. Our restrictions on $V^i,V^j,B^i,B^j$ mean that the gap between $i'$ and $j$ can shrink at most to $199\ell - 2 \cdot 8\ell - (12+12+12+12)\ell = 135\ell$. On the other hand, due to our restrictions on $U^i,V^i,X^i$ the gap between $i$ and $i'$ which start out at most $8\ell$ can grow to at most $8\ell + (12+12+35)\ell = 67\ell$. Similarly the gap between $j$ and $j'$ can grow to at most $67\ell$. Since $67\ell < 81\ell$ and $67\ell < 141\ell$ we don't have to worry about $i,j$ or $i',j'$ interfering with each other and causing a decoupling. \\
	\\
	All constants used are symmetric so same reasoning applies to the pairs $i,k$ with $i',k'$ and $j,k$ with $j',k'$. There are 9 coin sequences we need to bound ($U^i,U^j,U^k,V^i,V^j,V^k,B^i,B^j,B^k$ excluding $X^i,X^j,X^k$ because we already have those bounds from event $A$). We want to bound the maximum magnitude of each Heads-Tails differential in the first $\ell^2$ steps by $12\ell$. Theorem \ref{hoeffdingmax} gives us,
	\begin{equation}
		\probp{E_i,E_j,E_k \ | \ A} \geq 1 - 9\cdot4\exp\parens*{-\frac{12^2}{2}} \geq 1 - 10^{-30}.
	\end{equation}
	All together this means
	\begin{align}
		\probp{G,Q,E_i,E_j,E_k \ | \ A} \geq 1 - 3 \cdot 10^{-20}.
	\end{align}
	Recall that the distribution of $i',j',k'$ is independent of $A$. By Lemma \ref{gammaovercount} we get that
	\begin{equation}
		|N_i|,|N_j|,|N_k| \geq \frac{\ell^2}{2\gamma \sqrt{n}},
	\end{equation}
    so,
    \begin{equation}
		|N_i \times N_j \times N_k| \geq \frac{\ell^6}{8 \gamma^3n^\frac{3}{2}}.
	\end{equation}
	By Corollary \ref{montegamma} we have that
	\begin{equation}
		\probp{(i_T',j_T',k_T') = (f_i,f_j,f_k) \ | \ A} \geq 10^{-21} \cdot \frac{\gamma^3 n^\frac{3}{2}}{\ell^6}
	\end{equation}
	for all $(f_i,f_j,f_k) \in N_i \times N_j \times N_k$. Let $\mathcal{D} = 10^{22}$. Then,
	\begin{align}
		\frac{1}{\mathcal{D}|N_i \times N_j \times N_k|} &\leq \left(10^{22} \cdot \frac{8 \ell^6}{\gamma^3n^\frac{3}{2}}\right)^{-1} \\
		&< 10^{-21} \cdot \frac{\gamma^3 n^\frac{3}{2}}{\ell^6} \\
		&\leq \probp{(i_T',j_T',k_T') = (f_i,f_j,f_k) \ | \ A}
	\end{align}
	and
	\begin{align}
		1 - \frac{1}{8\mathcal{D}} \leq 1 - 3 \cdot 10^{-30} \leq \probp{G,Q,E_i,E_j,E_k \ | \ A}.
	\end{align}
	So by Theorem \ref{quasiuniform} in the appendix we have
	\begin{equation}
		\probp{(i_T',j_T',k_T') = (f_i,f_j,f_k) \ | \ A, G, Q, E_i, E_j, E_k} \geq \frac{1}{2} \cdot 10^{-21} \cdot \frac{\gamma^3 n^\frac{3}{2}}{\ell^6}
	\end{equation}
	for at least $\frac{3}{4}$ of all $(f_i,f_j,f_k) \in N_i \times N_j \times N_k$. Since $A,G,Q,E_i,E_j,E_k$ ensure that $(i,j,k)$ couple to $(i',j',k')$, we arrive at the statement of the theorem.
\end{proof}

Now we once again exploit the symmetry between the overlapping cycles shuffle and the inverse overlapping cycles shuffle to show the equivalent statement holds for the inverse overlapping cycles shuffle.

\begin{corollary}\label{montecoupledinverse}
	For the same constants in Theorem \ref{montecoupled} the following holds: Fix $\ell > C\sqrt{n}$. Suppose $i,j,k$ are cards such that $\norm{p(i)-p(j)},\norm{p(i)-p(k)},\norm{p(i)-p(k)} > 199\ell$. Let $T \in [\ell^2,\ell^2+4n]$ such that $T - \floor*{\frac{T}{2n}}m \equiv 0 \mod 2n-m+1$. Define the set 
	\begin{equation*}
		N_i = \left\{ \omega \text{ such that } p(i) + a + bm = p(\omega) \mod 2n-m+1 \text{ with } |a| \leq \ell, \ |b| \leq \frac{\ell}{\sqrt{n}} \right\}
	\end{equation*}
	Define $N_j$ and $N_k$ similarly for $j$ and $k$. Now let $i_{(-T)}$, $j_{(-T)}$, $k_{(-T)}$ be the locations of $i,j,k$ after doing $T$ steps of the inverse overlapping cycles shuffle. Then,
	\begin{equation}
		\probp{i_{(-T)} = f_i, j_{(-T)} = f_j, k_{(-T)} = f_k} \geq \frac{D\gamma^3 n^\frac{3}{2}}{\ell^6}
	\end{equation}
	for at least $\frac{3}{4}$ of all $(f_i,f_j,f_k)$ in $N_i \times N_j \times N_k$
\end{corollary}

\begin{proof}
	This corollary holds because the inverse overlapping cycles shuffle is itself an overlapping cycles shuffle up to the reordering of the cards in Proposition \ref{inverseshuffle}. According to Corollary \ref{inversenorm} the distances $|p(\cdot)-p(\cdot)|_M$ and $\norm{p(\cdot)-p(\cdot)}$ are fixed under this reordering. Since the parameters for $i,j,k$ and $N_i,N_j,N_k$ are defined using these distances, the proposition holds for the inverse overlapping cycles shuffle.
\end{proof}

We are now very close to what we want. We have shown that $i,j,k$ reach an approximately uniform distribution, but it's a distribution that excludes up to $\frac{1}{4}$ of our desired terms. What we really want is a distribution that includes \textit{every} term at a probability of a constant times the number of terms. To accomplish this, we make the following argument: \\
\\
Pick some cards $i,j,k$ and some reasonable targets $f_i,f_j,f_k$. Now run the overlapping cycles shuffle forward from $i,j,k$ and backwards from $f_i,f_j,f_k$. In the middle, $i,j,k$ and $f_i,f_j,f_k$ will be distributed over $\frac{3}{4}$ of all nearby terms. So a constant fraction of those terms will overlap, and we can show that $i,j,k$ goes to $f_i,f_j,f_k$ by summing over those overlapping middle terms. We formalize this in the following Theorem.

\begin{theorem}\label{montecomplete}
	Consider the overlapping cycles shuffle on $n$ cards where $\frac{m}{n} \in (\epsilon, 1 - \epsilon)$. There exist constants $C,D$ which depend on $\epsilon$ such that the following holds: Fix $\ell > C\sqrt{n}$. Fix any cards $i,j,k$ such that $\norm{i-k},\norm{i-k},\norm{j-k} > 199\ell$. Fix any positions $f_i,f_j,f_k$ such that $\norm{i-f_i},\norm{j-f_j},\norm{k-f_k} < \frac{\ell}{10}$ and $\norm{f_i-f_k},\norm{f_i-f_k},\norm{f_j-f_k} > 199\ell$. Let $T \in (\ell^2,\ell^2 + 4n)$ such that $T - \floor*{\frac{T}{2n}}m \equiv 0 \mod 2n-m+1$. Then,
	\begin{equation}
		\probp{(i_{2T},j_{2T},k_{2T}) = (f_i,f_j,f_k)} \geq \frac{D\gamma^3n^\frac{3}{2}}{\ell^6}
	\end{equation}
\end{theorem}

\begin{proof}
	Define the set 
	\begin{equation*}
		N_i = \left\{\text{positions } \omega \text{ such that } p(i) + a + bm = p(\omega) \mod 2n-m+1 \text{ with } |a| \leq \ell, \ |b| \leq \frac{\ell}{\sqrt{n}} \right\}.
	\end{equation*}
	Define the set
	\begin{equation*}
		N_i' = \left\{\text{positions } \omega \text{ such that } p(f_i) + a + bm = p(\omega) \mod 2n-m+1 \text{ with } |a| \leq \ell, \ |b| \leq \frac{\ell}{\sqrt{n}} \right\}.
	\end{equation*}
	Note that because $\norm{i-f_i} < \frac{\ell}{10}$ we know there exists $a_i,b_i$ such that $p(f_i) = p(i)+a_i+b_i m$ with $|a_i| < \frac{\ell}{10}, |b_i| < \frac{\ell}{10\sqrt{n}}$. For this reason $N_i$ and $N_i'$ overlap at least $(\frac{9}{10})^2 > \frac{4}{5}$ of their elements. Define $N_j,N_k,N_j',N_k'$ similarly for $j,k,f_j,f_k$. Then $N_i \times N_j \times N_k$ and $N_i' \times N_j' \times N_k'$ overlap at least $(\frac{4}{5})^3 \geq \frac{51}{100}$ of their elements. By Proposition \ref{montecoupled} and Corollary \ref{montecoupledinverse} we can bound the distribution of $(i_T,j_T,k_T)$ from below over $\frac{3}{4}$ of the elements of $N_j,N_k,N_k$ and bound the distribution of $({f_i}_{(-T)},{f_j}_{(-T)},{f_k}_{(-T)})$ from below over $\frac{3}{4}$ of the elements of $N_i',N_j',N_k'$. Let $\mathcal{S}$ be the subset of $(N_i \times N_j \times N_k) \cap (N_i' \times N_j' \times N_k')$ where the bounds from Proposition \ref{montecoupled} and Corollary \ref{montecoupledinverse} both hold. Then,
	\begin{equation}
		|\mathcal{S}| \geq \parens*{1 - \frac{1}{4} - \frac{1}{4} - \frac{49}{100}}|N_i \times N_j \times N_k| = \frac{1}{100}|N_i \times N_j \times N_k|.
	\end{equation}
	If we now run the shuffle $T$ steps forward from $i,j,k$ and $T$ steps backwards from $f_i,f_j,f_k$ we can compute the probability that they meet in the middle.
	\begin{align}
		&\probp{(i_{2T},j_{2T},k_{2T}) = (f_i,f_j,f_k)} \\
		\geq &\sum\limits_{(z_i,z_j,z_k) \in \mathcal{S}} \probp{(i_{T},j_{T},k_{T}) = (z_i,z_j,z_k)} \cdot \probp{({f_i}_{(-T)},{f_j}_{(-T)},{f_k}_{(-T)}) = (z_i,z_j,z_k)} \\
		\geq &\sum\limits_{(z_i,z_j,z_k) \in \mathcal{S}} \frac{D_1^2\gamma^6n^3}{\ell^{12}} \\
		= \ &\frac{D_1^2\gamma^6n^3}{100\ell^{12}}|N_i \times N_j \times N_k|.
	\end{align}
	As we showed in Lemma \ref{gammaovercount},
	\begin{equation}
		|N_i| \geq \frac{\ell^2}{2\gamma\sqrt{n}}.
	\end{equation}
	By symmetry the same bound applies to $N_j,N_k$ so
	\begin{equation}
		|N_i \times N_j \times N_k| \geq \frac{\ell^6}{8\gamma^3 n^\frac{3}{2}}.
	\end{equation}
	This gives us,
	\begin{align}
		\probp{(i_{2T},j_{2T},k_{2T}) = (f_i,f_j,f_k)} \geq \frac{D_1^2\gamma^3n^\frac{3}{2}}{800\ell^6}
	\end{align}
\end{proof}

\section{Entropy Decay}\label{entropysection}
In this section we will find our mixing time bound for the overlapping cycles shuffle. We will do this by applying Theorem \ref{montecomplete} to Theorem \ref{3monte} to bound each $A_x$ in the sum. As previously explained, we need $i,j,k$ to be spread out to apply Theorem \ref{montecomplete}, and so we use Theorem \ref{sdspread} and Corollary \ref{sdspreadinverse} as part of Stage 1 and Stage 3 to spread out $i,j,k$ and our targets $g_i,g_j,g_k$ to meet the requirements of $i,j,k$ and $g_i,g_j,g_k$ being distance $199\ell$ from each other. However there is one small technicality we have not addressed: How do we even know there exist positions $f_i,f_j,f_k$ that are spread out from each other and also a reachable distance from $i,j,k,g_i,g_j,g_k$? In other words, do there exist $f_i,f_j,f_k$ such that $\norm{p(f_i)},\norm{p(f_j)},\norm{p(f_k)} < \ell$ but $\norm{p(f_i)-p(f_j)},\norm{p(f_i)-p(f_k)},\norm{p(f_j)-p(f_k)} > c\ell$ for some not-to-small constant $c$? It seem clear that such positions should exist but a priori it is not obvious why. We deal with this in the following Theorem.

\begin{lemma}\label{positionspread}
	Fix some $\ell$ such that $100\sqrt{n} < \ell < \ell_{\rm max}$ where $\ell_{\rm max}$ is the maximum value of $\norm{\cdot}$ for the shuffle. There exist positions $f_i,f_j,f_k$ such that
	\begin{itemize}
		\item $\norm{p(f_i)},\norm{p(f_j)},\norm{p(f_k)} < \ell$
		\item $\norm{p(f_i) - p(f_j)}, \norm{(f_i)-p(f_k)}, \norm{p(f_j)-p(f_k)} > \frac{\ell}{5}$
	\end{itemize}
\end{lemma}

\begin{proof}
	Let $z$ be an element of $\{1,\dots,2n-m\}$ that maximizes $\norm{z}$. Then $z = a+bm$ where $\norm{z} = |a| + |b|\sqrt{n}$. For any constant $k \in (0,1)$, let $\abracket*{\frac{z}{2}}$ be defined by
	\begin{equation}
		\abracket*{\frac{z}{2}} = \ceil*{\frac{a}{2}} + \ceil*{\frac{b}{2}}m.
	\end{equation}
	Then,
	\begin{align}
		\norm*{\abracket*{\frac{z}{2}}} &\leq \frac{1}{2}\norm{z} + 1 + \sqrt{n},\\
		\norm*{\abracket*{\frac{z}{2}}} &\geq \frac{1}{2}\norm{z} - 1 - \sqrt{n}.
	\end{align}
	The second item is true because
	\begin{equation}
		\norm*{z} \leq \norm*{\abracket*{\frac{z}{2}} + \abracket*{\frac{z}{2}}} + 2 + 2\sqrt{n} \leq 2\norm*{\abracket*{\frac{z}{2}}} + 2 + 2\sqrt{n} .
	\end{equation}
	Also note that
	\begin{align}
		\norm*{z - \abracket*{\frac{z}{2}}} &= \norm*{\left(a - \ceil*{\frac{a}{2}}\right) + \left(b-\ceil*{\frac{b}{2}}\right)m} \\
		&= \norm*{\floor*{\frac{a}{2}} + \floor*{\frac{b}{2}}m} \\
		&\geq \norm*{\abracket*{\frac{z}{2}}} - 2 - 2\sqrt{n}.
	\end{align}
	Define $\abracket*{\frac{z}{4}}$ as $\abracket*{\frac{\abracket*{\frac{z}{2}}}{2}}$ and $\abracket*{\frac{z}{8}}$ as $\abracket*{\frac{\abracket*{\frac{z}{4}}}{2}}$, etc. Then it follows inductively that
	\begin{align}
		\norm*{\abracket*{\frac{z}{2^x}}} &\leq \frac{1}{2^x}\norm{z} + 2 + 2\sqrt{n} \\
		\norm*{\abracket*{\frac{z}{2^x}}} &\geq \frac{1}{2^x}\norm{z} - 2 - 2\sqrt{n}.
	\end{align}
	Now choose $x$ such that 
	\begin{equation}
		\frac{\ell}{5} < \frac{1}{2^{x+1}}\norm{z} - 6 - 4\sqrt{n} < \frac{1}{2^x}\norm{z} + 2 + 2\sqrt{n} < \ell \label{zlowerbound}.
	\end{equation}
	Using the inequality on the right of (\refeq{zlowerbound}) gives us
	\begin{equation}
		\norm*{\abracket*{\frac{z}{2^{x+1}}}} < \norm*{\abracket*{\frac{z}{2^{x}}}} < \ell.
	\end{equation}
	Using the inequality on the left of (\refeq{zlowerbound}) gives us
	\begin{align}
		\norm*{\abracket*{\frac{z}{2^{x}}} - \abracket*{\frac{z}{2^{x+1}}}} &\geq \norm*{\abracket*{\frac{z}{2^{x+1}}}} - 2 - 2\sqrt{n} \\
		&\geq \frac{1}{2^{x+1}}\norm{z}-4-4\sqrt{n} \\
		&\geq \frac{\ell}{5} + 2.
	\end{align}
	Let $f_i$ be position $1$ in the deck and let $f_j,f_k$ be positions in the deck such that
	\begin{align}
		\left| p(f_j) - \abracket*{\frac{z}{2^{x+1}}}\right| &\leq 1 \\
		\left| p(f_k) - \abracket*{\frac{z}{2^{x}}}\right| &\leq 1.
	\end{align}
	Then,
	\begin{align}
		\norm*{p(f_i)-p(f_j)} &\geq \norm*{\abracket*{\frac{z}{2^{x+1}}}} - 2 \geq \frac{1}{2^{x+1}} \norm{z} - 4 - 2\sqrt{n} \geq \frac{\ell}{5} \\
		\norm*{p(f_i)-p(f_k)} &\geq \norm*{\abracket*{\frac{z}{2^{x}}}} - 2 \geq \frac{1}{2^{x}} \norm{z} - 4 - 2\sqrt{n} \geq \frac{\ell}{5} \\
		\norm*{p(f_j)-p(f_k)} &\geq \norm*{\abracket*{\frac{z}{2^{x}}}-\abracket*{\frac{z}{2^{x+1}}}} - 2 \geq \frac{\ell}{5} + 2 - 2 = \frac{\ell}{5}.
	\end{align}
\end{proof}

We now have finished all the work necessary to justify Stage 1, Stage 2, and Stage 3. There is one final proposition we need to show before we apply Theorem \ref{3monte}. We need to show that after $i,j,k$ get close together, they are likely to collide with each other.

\begin{proposition}\label{sqrtncollide}
	Consider the overlapping cycles shuffle where $\frac{m}{n} \in (\epsilon, 1 - \epsilon)$. Consider any cards $i,j,k$ and suppose that $|p(i_T)-p(j_T)|_M,|p(i_T)-p(k_T)|_M,|p(j_T)-p(k_T)|_M < \sqrt{n}$. Let $E$ be the event that the next time $i$ or $j$ or $k$ collides after time $T$ it is with each other and in the order $(i,k,j)$, and the next time $i$ collides after time $T$ it is before time $T + 10n$. Then,
	\begin{equation}
		\probp{E} \geq \frac{D}{n}
	\end{equation}
	where $D$ is a constant that depends on $\epsilon$.
\end{proposition}

Intuitively, this makes sense. If $i,j,k$ begin distance $\sqrt{n}$ from each other with respect to $\norm{\cdot}$ then by the analysis we have done in previous sections we should believe that $i,j,k$ can travel to positions $m-1,n,m$ respectively in $C(\sqrt{n})^2 = Cn$ steps, as positions $m-1,n,m$ are also distance $\sqrt{n}$ from each other with respect to $\norm{\cdot}$. The only wrinkle is making sure $i,j,k$ do not collide with any other cards along the way.

\begin{proof}
	Let $\tau_1$ be the random stopping time given by the minimal $t \geq T$ such that $i,j,k \in (1,m)$. Let $Q_1$ be the event that 
	\begin{itemize}
		\item $i,j,k$ each flip Tails on their first visit to position $m$ before time $\tau_1$ (if it exists),
	\end{itemize}
	Note that $\probp{Q_1} \geq \parens*{\frac{1}{2}}^3$ accounting for at most $3$ Tails flips. Now let $Q_2$ be the event that,
    \begin{itemize}
        \item The next two times after $\tau_1$ that $i$ is in position $m$, the coin flips are Tails, Heads respectively.
        \item The next two times after $\tau_1$ that $k$ is in position $m$, the coin flips are Heads, Tails respectively.
        \item The next two times after $\tau_1$ that $j$ is in position $m$, the coin flips are Heads in both instances.
    \end{itemize}
    Note that $\probp{Q_2 \ | \ Q_1} = (\frac{1}{2})^6$. The event $Q_2$ will ensure that there is time when $i$ is in the bottom part of the deck by itself, followed by time when when $k$ is in the bottom part of the deck without $j$, followed by time when $j$ is in the bottom part of the deck by itself. Each of these blocks of time time will be on the order of $\min(m, n-m) = \epsilon n$. Let $\tau_2$ be the stopping time when $k$ reaches in position $m$ for the third time after $\tau_1$. At this point, using the notation from Corollary \ref{cardmovementcompare}, we will have $H_B(i) - H_B(k) = 0$ since $i$ and $k$ will have each had one Heads flip while in position $m$. We will also have $H_B(j) - H_B(k) = 1$, as $j$ will have flipped two Heads while in position $m$. This means that $i$ and $k$ will be close to the same position at time $\tau_2$, and $j$ will be about $m$ steps above $k$ (mod $2n-m+1)$. This is exactly what we want, as we want $i_{\tau_2} = m-1$ and $j_{\tau_2} = n$ for $i,j,k$ to collide in the next step. \\
    \\
    To bound the probability $i_{\tau_2} = m-1$ and $j_{\tau_2} = n$, recall that between time $\tau_1$ and $\tau_2$ card $i$ spent on the order of $\epsilon n$ steps in the bottom part of the deck by itself, and $j$ spent on the order of $\epsilon n$ steps in the bottom part of the deck by itself. In order to have $i_{\tau_2} = m-1$ and $j_{\tau_2} = n$ we will need:
    \begin{itemize}
        \item $(T_S(i) - H_S(i)) - (T_S(k) - H_S(k)) = p(k_T) - p(i_T) - 1$,
        \item $(T_S(j) - H_S(j)) - (T_S(k) - H_S(k)) = p(k_T) - p(j_T) - 1$.
    \end{itemize}
    Since $i,j,k$ were within $\sqrt{n}$ positions of each other at time $T$, we get that $p(k_T) - p(i_T)$ and $p(k_T) - p(j_T)$ are both less than $2\sqrt{n}$. So, to have $i_{\tau_2} = m-1$ and $j_{\tau_2} = n$ we need the Heads-Tails differentials of $i$ and $j$ to differ from those of $k$ by a precise amount less than $2\sqrt{n}$. Since $i$ and $j$ each spend at least on the order of $\epsilon n$ steps alone in the bottom of the deck, their Heads-Tails differential has a standard deviation on the order of $\sqrt{\epsilon n}$. We need $i$ and $j$ to have Heads-Tails differentials compared to $k$'s which hit a precise value one the order of $\frac{1}{\sqrt{\epsilon}}$ standard deviations away, so
    \begin{equation}
        \probp{i_{\tau_2} = m-1, j_{\tau_2} = n \ | \ Q_1, Q_2} = \frac{D_1}{\sqrt{n} \cdot \sqrt{n}} = \frac{D_1}{n}
    \end{equation}
    where $D_1$ is a constant that depends on $\epsilon$. It remains to justify that, conditioned $Q_1, Q_2, i_{\tau_2} = m-1, j_{\tau_2} = n$, we have with constant probability that $\tau_2 - T < 10n$ and $i,j,k$ won't collide with any other cards between times $T$ and $\tau_2$. Recall that $\tau_1 - T$ is the time it takes for $i,j,k$ to cycle through the deck until they are all in the top part of the deck. This will take at about $2n-m$ steps in the longest case (if they have to cycle all the way through the deck) so with high probability $\tau_1 - T < 3n$. Recall that $\tau_2 - \tau_1$ is the time it takes $k$ to reach position $m$ for the third time after time $\tau_1$. Laps through the deck take about $2n$ steps in the longest case, so with high probability $\tau_2 - \tau_1 < 3 \cdot 2n + n = 7n$. So, with high probability, $\tau_2 - T < 10n$. \\
    \\
    For $i,j,k$ to avoid collisions with other cards between time $T$ and $\tau_2$, recall that for a card to be in a collision it has to be in position $m-1$, $m$, or $n$ at an even time and have the next two flips be Heads, Tails or Tails, Heads. We can ensure that $i,j,k$ avoid collisions before time $\tau_2$ by having flips land Heads, Heads or Tails, Tails whenever $i,j,k$ are in positions $m-1, m$, or $n$ at an even time. Since $i,j,k$ are only in these special positions a constant number of times before $\tau_2$, we know they avoid collisions with a constant probability. So,
    \begin{align}
        \probp{E} &\geq \probp{E \ | \ i_{\tau_2} = m-1, j_{\tau_2} = n, Q_1, Q_2} \cdot \probp{i_{\tau_2} = m-1, j_{\tau_2} = n \ | \ Q_1, Q_2} \cdot \probp{Q_1,Q_2} \\
        &\geq D_2 \cdot \frac{D_1}{n} \cdot D_3 \\
        &= \frac{D}{n}
    \end{align}
    where $D$ is a constant that depends on $\epsilon$.
\end{proof}

Now we are ready to apply Theorem \ref{3monte}. As a reminder, our strategy is broken into 3 stages.
\begin{itemize}
	\item In Stage 1 we use Theorem \ref{sdspread} to show that $i,j,k$ spread out sufficiently.
	\item In Stage 2 we use Theorem \ref{montecomplete} to show that $i,j,k$ move to a precise position.
	\item In Stage 3 we use Theorem \ref{sdspreadinverse} to show that $i,j,k$ collapse back together.
\end{itemize}
Then finally we use Theorem \ref{sqrtncollide} to show that $i,j,k$ collide with each other.

\begin{theorem}\label{collidecomplete}
	Consider the overlapping cycles shuffle on $n$ cards where $\frac{m}{n} \in (\epsilon, 1-\epsilon)$. There exist constants $C,D$ which depend on $\epsilon$ such that the following is true: Fix any $\ell$ such that $C\sqrt{n} \leq \ell \leq \ell_{\rm max}$. Let $i,j,k$ be cards such that $\norm{p(i)},\norm{p(j)},\norm{p(k)} < \ell$. Let $T_1 \in [\ell^2,\ell^2+4n]$ such that $T_1 - \floor*{\frac{T_1}{2n}}(m-1) \equiv 0 \mod 2n-m+1$. Let $T_2 \in [10^{-6}\ell^2,10^{-6}\ell^2+4n]$ such that $T_2 - \floor*{\frac{T_2}{2n}}m \equiv 0$. Let $T = 2T_1 + 2T_2$ and let $t = T + 10n$. Let $E$ be the event that the first time $i$ collides after time $T$, it is with $j$ and $k$ on the front and back respectively, and it happens before time $t$. Then,
	\begin{equation*}
		\probp{E} \geq \frac{D\gamma^2n}{\ell^4}.
	\end{equation*}
\end{theorem}

\begin{proof}
	Let $N_\ell$ be the set of positions
	\begin{equation}
		N_\ell = \left\{\text{positions } \omega \text { such that } p(\omega) = a + bm \text{ with } |a| \leq \ell, \ |b| \leq \frac{\ell}{\sqrt{n}}\right\}
	\end{equation}
	Choose any positions $g_i,g_j,g_k$ such that
	\begin{itemize}
		\item $p(g_i) \in N_\ell$
		\item $|p(g_i)-p(g_j)|_M,|p(g_i)-p(g_k)|_M,|p(g_j)-p(g_k)|_M < \sqrt{n}$
	\end{itemize}
	Now according to Lemma \ref{positionspread} there exist positions $f_i,f_j,f_k$ such that
	\begin{align}
		\norm{p(f_i)},\norm{p(f_j)},\norm{p(f_k)} &< \ell, \\
		\norm{p(f_i)-p(f_j)},\norm{p(f_i)-p(f_j)},\norm{p(f_i)-p(f_k)} &> \frac{\ell}{5}.
	\end{align}
	Since $\norm{p(i)},\norm{p(j)},\norm{p(k)} < \ell$ we also have
    \begin{equation}
        \norm{p(i)-p(f_i)}, \norm{p(j)-p(f_j)}, \norm{p(k)-p(f_k)} < 2\ell.
    \end{equation}
    As long as $C$ is sufficiently large, we have according to Proposition \ref{sdspread} that
	\begin{equation}
		\probp{G_1} := \probp{\norm{p(i_{T_1})-p(f_i)},\norm{p(j_{T_1})-p(f_j)},\norm{p(k_{T_1})-p(f_k)} < \frac{\ell}{2000}} \geq D_1.
	\end{equation}
	for a constant $D_1$ that depends on $\epsilon$. Similarly by Corollary \ref{sdspreadinverse} we have that
	\begin{equation}
		\probp{G_2} := \probp{\norm{p({g_i}_{(-T_1)})-p(f_i)},\norm{p{(g_j}_{(-T_1)})-p(f_j)},\norm{p({g_k}_{(-T_1)})-p(f_k)} < \frac{\ell}{2000}} \geq D_1.
	\end{equation}
	On $G_1,G_2$ we get that
	\begin{itemize}
		\item $\norm{p(i_{T_1})-p({g_i}_{(-T_1)})},\norm{p(j_{T_1})-p({g_j}_{(-T_1)})},\norm{p(k_{T_1})-p({g_k}_{(-T_1)})} < \frac{\ell}{1000}$,
		\item $\norm{p(i_{T_1})-p(j_{T_1})},\norm{p(i_{T_1})-p(k_{T_1})},\norm{p(j_{T_1})-p(k_{T_1})} > \frac{\ell}{5} - \frac{2\ell}{2000} = \frac{199\ell}{1000}$,
		\item $\norm{p({g_i}_{(-T_1)})-p({g_j}_{(-T_2)})},\norm{p({g_i}_{(-T_1)})-p({g_k}_{(-T_1)})},\norm{p({g_j}_{(-T_1)})-p({g_k}_{(-T_2)})} > \frac{\ell}{5} - \frac{2\ell}{1000} = \frac{199\ell}{1000}$.
	\end{itemize}
	We apply Theorem \ref{montecomplete} using $10^{-3}\ell$ in place of $\ell$ to get
	\begin{equation}
		\probp{(i_{T_1+2T_2},j_{T_1+2T_2},k_{T_1+2T_2}) = ({g_i}_{(-T_1)},{g_j}_{(-T_1)},{g_k}_{(-T_1)}) \ | \ G_1,G_2} \geq \frac{D_2\gamma^3n^\frac{3}{2}}{\ell^6}.
	\end{equation}
	Since we are using Theorem \ref{montecomplete} with $10^{-3}\ell$ in place of $\ell$, the $\gamma$ in the numerator should be $\gamma(10^{-3}\ell,n,m)$. However, this and $\gamma(\ell,n,m)$ only differ by a constant factor, so assume that the $\gamma$ in the numerator is $\gamma(\ell,n,m)$ and the constant factor is absorbed into $D_2$. Together we get that
	\begin{equation}
		\probp{(i_{T_1+2T_2+T_1} ,i_{T_1+2T_2+T_1}, i_{T_1+2T_2+T_1}) = (g_i,g_j,g_k)} \geq \frac{D_1^2D_2\gamma^3n^\frac{3}{2}}{\ell^6}.
	\end{equation}
	Let $Q(g_i,g_j,g_k)$ be the event that the next time $g_i$ collides it is with $g_j$ as its front match and $g_k$ as its back match, and that this collision happens within $8n$ steps. As shown in Theorem \ref{sqrtncollide},
	\begin{equation}
		\probp{Q(g_i,g_k,g_k)} \geq \frac{D_3}{n}.
	\end{equation}
	Let $\mathcal{R}_\ell$ be the set of choices for the triplet $(g_i,g_j,g_k)$ according to our parameters at the start of the proof. To be specific, let
	\begin{equation}
		\mathcal{R}_\ell = \{(g_i,g_j,g_k) \text{ such that } g_i,g_j,g_k \in N_\ell \text{ and } |p(g_i)-p(g_j)|_M,|p(g_i)-p(g_k)|,|p(g_j)-p(g_k)|_M < \sqrt{n}\}.
	\end{equation}
	Then,
	\begin{align}
		\probp{E} &\geq \sum\limits_{(g_i,g_j,g_k) \in \mathcal{R}_\ell} \probp{E,(i_{T_1+2T_2+T_1} ,j_{T_1+2T_2+T_1}, k_{T_1+2T_2+T_1}) = (g_i,g_j,g_k)} \\
		&\geq \sum\limits_{(g_i,g_j,g_k) \in \mathcal{R}_\ell} \probp{(i_{T_1+2T_2+T_1} ,j_{T_1+2T_2+T_1}, k_{T_1+2T_2+T_1}) = (g_i,g_j,g_k)} \cdot \probp{Q(g_i,g_j,g_k)} \\
		&\geq \sum\limits_{(g_i,g_j,g_k) \in \mathcal{R}_\ell} \frac{D_4\gamma^3n^\frac{3}{2}}{n\ell^6}.
	\end{align}
	We now need to bound $|\mathcal{R}_\ell|$. Note that $g_i$ is chosen from $N_\ell$. For each choice of $g_i \in N_\ell$ we can choose $g_j$ such that $p(g_j) = p(g_i) + \delta_j$ for $\delta_j \in (\sqrt{n},\sqrt{n})$. At least half of these values of $\delta_j$ are associated with valid positions in the deck, so there are at least $\sqrt{n}$ choices for $g_j$. If $\delta_j$ is negative then we can choose $g_k$ such that $p(g_k) = p(g_j) + \delta_k$ with $\delta_k\in \{1,\dots,\sqrt{n}\}$ and if $\delta_j$ is positive we can do the same with $\delta_k\in \{-1,\dots,-\sqrt{n}\}$. This will meet the requirement that $|p(g_i)-p(g_j)|_M$ and $|p(g_i)-p(g_k)|_M$ and $|p(g_j)-p(g_k)|_M$ are all less than or equal to $\ell$. This gives us at least $\sqrt{n} \cdot \frac{\sqrt{n}}{2} = \frac{n}{2}$ choices for $(g_j,g_k)$. Thus,
	\begin{equation}
		|\mathcal{R}_\ell| \geq \frac{n}{2}|N_\ell|.
	\end{equation}
	As we showed in Lemma \ref{gammaovercount} we have
	\begin{equation}
		|N_\ell| \geq \frac{\ell^2}{2\gamma\sqrt{n}} \label{gammabound}.
	\end{equation}
	So,
	\begin{equation}
		|\mathcal{R}_\ell| \geq\frac{\sqrt{n}\ell^2}{4\gamma}
	\end{equation}
	and
	\begin{equation}
		\probp{E} \geq \frac{D\gamma^2n}{\ell^4}.
	\end{equation}
\end{proof}

Now that we have a bound on the $A_x$ in the sum of Theorem \ref{3monte} it is time to find a bound on the mixing time. To do this we first show that the Entropy of the overlapping cycles shuffle decays at an exponential rate. As part of the proof, we will use a slight variation of Theorem \ref{3monte}. The only variation will be that, instead of denoting $y>x$ for cards $y,x$ where $x$ is above $y$ in the deck, we will use a different well-ordering of the deck. We will define a permutation $\nu$ which translates from the top-to-bottom ordering to our new well-ordering. Using this well-ordering Theorem \ref{3monte} will still hold, because the ordering assumed in the Theorem is arbitrary. Since the Theorem is defined for any random permutation, there is no reason any particular ordering is preferred.

\begin{lemma}\label{penultimatelemma}
	Consider the overlapping cycles shuffle with $n$ cards and parameter $m$ where $\frac{m}{n} \in (\epsilon, 1 - \epsilon)$ and $n$ is sufficiently large. There exist constants $C,D$ which depend on $\epsilon$ such that the following is true: If $\pi_t$ is the overlapping cycles shuffle with $t$ steps then there is a value $t \in \{1,\dots,C\ell_{\rm max}^2\}$ such that
	\begin{align}
		\mathbb{E}\big[\text{\rm ENT}(\mu\pi_t|\text{\rm sgn}(\mu\pi_t))\big]\leq \parens*{1 - \frac{Dt}{\ell_{\max}^2\log^2(n)}}\mathbb{E}[\text{\rm ENT}(\mu \ | \ \text{\rm sgn}(\mu))]
	\end{align}
\end{lemma}

\begin{proof}
	Let $a = \ceil{\log_2(\mathcal{\ell_{\rm max}}) - \frac{1}{2}\log_2(n)}$. For $k \in \{1,\dots,a\}$, let $\ell_k = 2^{k}\sqrt{n}$. For each $k$, let $\gamma_k$ be defined by
	\begin{equation}
		\gamma_k = \left| \left\{ \kappa \in \mathbb{N} \ : \ |\kappa m|_M < \ell_k, \ \kappa < \frac{\ell_k}{\sqrt{n}}\right\}\right|.
	\end{equation}
    In other works, let $\gamma_k$ be the $\gamma = \gamma(\ell_k,n,m)$ from the bound in Theorem \ref{collidecomplete} associated with $\ell_k$. \\
    \\
    Now we partition the deck of $n$ cards as follows:
	\begin{itemize}
		\item Let $J_0 := \{n-2,n-3,\dots,n-\floor{\sqrt{n}}\}$.
		\item For $k \geq 1$ let $J_k = \{i \ : \ \norm{p(i)} \leq \ell_k\}$.
		\item Let $I_0 = J_0$ and for $k \geq 1$ let $I_k = J_k \backslash J_{k-1}$.
	\end{itemize}
	Let $\nu$ be a permutation which reorders the deck with the following properties:
	\begin{itemize}
		\item $\nu(n) = 1,\nu(n-1) = 2, \dots, \nu\left(n-\floor*{\sqrt{n}}\right) = \floor*{\sqrt{n}} + 1$
		\item $\nu$ respects the natural ordering of $I_k$. Specifically, if $x \in I_k$ and $y \in I_{k+1}$ then $\nu(x) < \nu(y)$.
	\end{itemize}
	Note that the second item does not contradict the first because $p(n-2),p(n-3),\dots,p(n-\floor{\sqrt{n}})$ all have norms less than or equal to $2\sqrt{n}$. Note that under this reordering, if $x \in I_k$ with $k \geq 1$ then $\nu(x) \geq |J_{k-1}|$.\\
	\\
	Let $E_j = \mathbb{E}[\text{\rm ENT}(\mu^{-1}(\nu(j)) \ | \ \text{\rm sgn}(\mu),\mu^{-1}(\nu(j)+1),\mu^{-1}(\nu(j)+2),\dots,\mu^{-1}(\nu(n))]$. Then as shown by Senda \cite{monte} in Appendix B we can decompose,
	\begin{align}
		\mathbb{E}[\text{\rm ENT}(\mu \ | \ \text{\rm sgn}(\mu))] = \sum\limits_{\nu(j)=3}^n E_j = \sum\limits_{k=1}^a \sum\limits_{j \in I_k} E_j.
	\end{align}
	Let $k^*$ be such that $\sum\limits_{j \in I_{k^*}} E_j$ is maximal. Then,
	\begin{align}
		\mathbb{E}[\text{\rm ENT}(\mu \ | \ \text{\rm sgn}(\mu))] &\leq a \sum\limits_{j \in I_{k^*}} E_j, \\
		\frac{1}{a}\mathbb{E}[\text{\rm ENT}(\mu \ | \ \text{\rm sgn}(\mu))] &\leq \sum\limits_{j \in I_{k^*}} E_j. \label{entdiva}
	\end{align}
	For any card $x$ in the deck let $A_x$ be the maximal value such that
	\begin{equation}
		\probp{m_2(x) = y, \nu(m_1(x)) < \nu(x)} \geq \frac{A_x}{\nu(x)}
	\end{equation}
	for all cards $y$ such that $\nu(y) < \nu(x)$. Now, by Theorem \ref{3monte}, if we examine the shuffle after any number $t$ steps, we have
	\begin{align}
		\mathbb{E}\big[\text{\rm ENT}(\mu\pi_t|\text{\rm sgn}(\mu\pi_t))\big] - \mathbb{E}\big[\text{\rm ENT}((\mu|\text{\rm sgn}(\mu)))\big] &\leq \frac{-C_1}{\log (n)} \sum_{\nu(x)=3}^{n}A_x E_x \\
		&\leq \frac{-C_1}{\log (n)} \sum\limits_{x \in I_{k^*}} A_x E_x \label{entdiffbound}
	\end{align}
	where the second inequality comes from the fact that we are summing over fewer negative terms. We now consider three cases for $k^*$.
	\begin{enumerate}
		\item $k^* = 0$ \\
		\\
		Fix any $x \in I_0$. Then $x \in [n-2,n-3,\dots,n-\floor{\sqrt{n}}]$. Fix $t = 2n + 5\sqrt{n}$ and $T = 2n - 5\sqrt{n}$. Then by Proposition \ref{l1collidecomplete} we have
		\begin{equation}
			A_x \geq \frac{D_1}{\sqrt{n}}.
		\end{equation}
		for a constant $D_1$ which depends on $\epsilon$. Plugging into the bound from (\refeq{entdiffbound}) we get
		\begin{equation}
			\mathbb{E}\big[\text{\rm ENT}(\mu\pi_t|\text{\rm sgn}(\mu\pi_t))\big] - \mathbb{E}\big[\text{\rm ENT}((\mu|\text{\rm sgn}(\mu)))\big] \leq \frac{-C_1D_1}{\sqrt{n} \log (n)} \sum\limits_{x \in I_{k^*}} E_x.
		\end{equation}
		Since $t$ is less than a constant times $n$ we have,
		\begin{equation}
			\mathbb{E}\big[\text{\rm ENT}(\mu\pi_t|\text{\rm sgn}(\mu\pi_t))\big] - \mathbb{E}\big[\text{\rm ENT}((\mu|\text{\rm sgn}(\mu)))\big] \leq \frac{-D_2t}{n^\frac{3}{2} \log (n)} \sum\limits_{x \in I_{k^*}} E_x.
		\end{equation}
		Recall that $\ell_{\max} \geq \frac{1}{2}n^\frac{3}{4}$, so $n^\frac{3}{2} \leq 4\ell_{\max}^2$. This gives us
		\begin{equation}
			\mathbb{E}\big[\text{\rm ENT}(\mu\pi_t|\text{\rm sgn}(\mu\pi_t))\big] - \mathbb{E}\big[\text{\rm ENT}((\mu|\text{\rm sgn}(\mu)))\big] \leq \frac{-D_3t}{\ell_{\max}^2 \log (n)} \sum\limits_{x \in I_{k^*}} E_x.
		\end{equation}
		
		\item $1 \leq k^* \leq \log_2(C_2)$ where $C_2$ is the constant called ``$C$'' in Theorem \ref{collidecomplete}. \\
		\\
		In this case all cards $x$ in $I_{k^*}$ have $\norm{x} < C_2\sqrt{n}$. Set $t = C_2^2(2+2\cdot 10^{-6})n+4 \cdot 4n+10n$. Then by Theorem \ref{collidecomplete}, for each $i,j,k$ such that $i \in I_{k^*}$ and $\nu(j),\nu(k) < \nu(i)$ we have
        \begin{equation}
            \probp{m_2(i)=j,m_1(i)=k} \geq \frac{D_4\gamma_{k^*}^2n}{\ell_{k^*}^4}.
        \end{equation}
        By summing over all $k \neq i$ such that $\nu(k) < \nu(i)$ we get,
        \begin{equation}
            \probp{m_2(i)=j,\nu(m_1(i))<\nu(i)} \geq \frac{D_4\gamma_{k^*}^2n}{\ell_{k^*}^4} \nu(x).
        \end{equation}
        So, for each $A_x$ in the sum $\sum\limits_{x \in I_{k^*}} A_x E_x$ we have
		\begin{align}
			A_x \geq \parens*{\frac{D_4\gamma_{k^*}^2n}{\ell_{k^*}^4} \nu(x)} \nu(x) = \frac{D_4\gamma_{k^*}^2n}{\ell_{k^*}^4} \nu(x)^2.
		\end{align}
		  Recall that $\ell_{k^*} \leq C_2\sqrt{n}$ and $\nu(x) \geq \sqrt{n}$ and $\gamma_{k^*} \geq 1$. Plugging in this information we get
		\begin{align}
			A_x \geq \frac{D_4}{C_2^2}.
		\end{align}
		Now plugging this into the bound from (\refeq{entdiffbound}) gives us
		\begin{equation}
			\mathbb{E}\big[\text{\rm ENT}(\mu\pi_t|\text{\rm sgn}(\mu\pi_t))\big] - \mathbb{E}\big[\text{\rm ENT}((\mu|\text{\rm sgn}(\mu)))\big] \leq \frac{-C_1D_4}{C_2^2 \log (n)}\sum\limits_{x \in I_{k^*}} E_x.
		\end{equation}
		Note that $t < n^\frac{3}{2}$ for sufficiently large $n$. So for large $n$ we have
		\begin{align}
			\mathbb{E}\big[\text{\rm ENT}(\mu\pi_t|\text{\rm sgn}(\mu\pi_t))\big] - \mathbb{E}\big[\text{\rm ENT}((\mu|\text{\rm sgn}(\mu)))\big] &\leq \frac{-D_5t}{n^\frac{3}{2} \log (n)} \sum\limits_{x \in I_{k^*}} E_x.
		\end{align}
		Using again that $n^\frac{3}{2} \leq 4\ell_{\max}$ we get
		\begin{align}
			\mathbb{E}\big[\text{\rm ENT}(\mu\pi_t|\text{\rm sgn}(\mu\pi_t))\big] - \mathbb{E}\big[\text{\rm ENT}((\mu|\text{\rm sgn}(\mu)))\big] &\leq \frac{-D_6t}{\ell_{\max}^2 \log (n)}  \sum\limits_{x \in I_{k^*}} E_x.
		\end{align}
		
		\item $k^* \geq \log_2(C_2)$ \\
		\\
		Set $t = \ell_{k^*}^2(2+2\cdot 10^{-6})+4 \cdot 4n+10n$. Then by Theorem \ref{collidecomplete}, for each $i,j,k$ such that $i \in I_{k^*}$ and $\nu(j),\nu(k) < \nu(i)$ we have
        \begin{equation}
            \probp{m_2(i)=j,m_1(i)=k} \geq \frac{D_4\gamma_{k^*}^2n}{\ell_{k^*}^4}.
        \end{equation}
        By summing over all $k \neq i$ such that $\nu(k) < \nu(i)$ we get,
        \begin{equation}
            \probp{m_2(i)=j,\nu(m_1(i))<\nu(i)} \geq \frac{D_4\gamma_{k^*}^2n}{\ell_{k^*}^4} \nu(x).
        \end{equation}
        So, for each $A_x$ in the sum $\sum\limits_{x \in I_{k^*}} A_x E_x$ we have
		\begin{align}
			A_x \geq \parens*{\frac{D_4\gamma_{k^*}^2n}{\ell_{k^*}^4} \nu(x)} \nu(x) = \frac{D_4\gamma_{k^*}^2n}{\ell_{k^*}^4} \nu(x)^2.
		\end{align}
		Recall that for all $x \in I_{k^*}$ we have $\nu(x) > |J_{k^*-1}|$. For $k \geq 1$ we know $J_k$ is defined identically to $N_{\ell_k}$ from Lemma \ref{gammaovercount}. So by that lemma,
		\begin{equation}
			|J_{k^*-1}| \geq \frac{\ell_{k^*-1}^2}{2\gamma_{k^*-1}\sqrt{n}}.
		\end{equation}
		This gives us
		\begin{align}
			A_x &\geq D_4\left(\frac{\ell_{k^*-1}}{\ell_{k^*}}\right)^4\left(\frac{\gamma_{k^*}}{\gamma_{k^*-1}}\right)^2.
		\end{align}
		Note that $\gamma_{k^*} > \gamma_{k^*-1}$ and recall that $\ell_{k^*-1} = \frac{1}{2}\ell_{k^*}$. So $A_x$ is bounded below by a constant,
		\begin{equation}
			A_x \geq D_7.
		\end{equation}
		Plugging this into $(\refeq{entdiffbound})$ we get
		\begin{equation}
			\mathbb{E}\big[\text{\rm ENT}(\mu\pi_t|\text{\rm sgn}(\mu\pi_t))\big] - \mathbb{E}\big[\text{\rm ENT}((\mu|\text{\rm sgn}(\mu)))\big] \leq \frac{-C_1D_7}{\log (n)} \sum\limits_{x \in I_{k^*}} E_x.
		\end{equation}
		Since $t$ is less than a constant times $\ell_{k^*}^2$ we have that
		\begin{align}
			\mathbb{E}\big[\text{\rm ENT}(\mu\pi_t|\text{\rm sgn}(\mu\pi_t))\big] - \mathbb{E}\big[\text{\rm ENT}((\mu|\text{\rm sgn}(\mu)))\big] &\leq \frac{-D_8t}{\ell_{k^*}^2\log (n)}\sum\limits_{x \in I_{k^*}} E_x \\
			&\leq \frac{-D_9t}{\ell_{\max}^2\log (n)} \sum\limits_{x \in I_{k^*}} E_x.
		\end{align}
	\end{enumerate}
	Set $D_{10}$ to be the minimum of $D_3,D_6,D_9$. Then we have the bound
	\begin{align}
		\mathbb{E}\big[\text{\rm ENT}(\mu\pi_t|\text{\rm sgn}(\mu\pi_t))\big] - \mathbb{E}\big[\text{\rm ENT}((\mu|\text{\rm sgn}(\mu)))\big] &\leq \frac{-D_{10}t}{\ell_{\max}^2\log (n)} \sum\limits_{x \in I_{k^*}} E_x
	\end{align}
	independent of the value of $k^*$. Recall from line (\refeq{entdiva}) that
	\begin{equation}
		\frac{1}{a}\mathbb{E}[\text{\rm ENT}(\mu \ | \ \text{\rm sgn}(\mu))] \leq \sum\limits_{x \in I_{k^*}} E_j.
	\end{equation}
	So we have
	\begin{equation}
		\mathbb{E}\big[\text{\rm ENT}(\mu\pi_t|\text{\rm sgn}(\mu\pi_t))\big] \leq \parens*{1 - \frac{D_{10}t}{a\ell_{\max}^2\log(n)}}\mathbb{E}[\text{\rm ENT}(\mu \ | \ \text{\rm sgn}(\mu))].
	\end{equation}
	Since $a \leq \log_2(\ell_{\max})+1 \leq \log_2(2n)+1 \leq 2\log(n)$ we have
	\begin{equation}
		\mathbb{E}\big[\text{\rm ENT}(\mu\pi_t|\text{\rm sgn}(\mu\pi_t))\big] \leq \parens*{1 - \frac{Dt}{\ell_{\max}^2\log^2(n)}}\mathbb{E}[\text{\rm ENT}(\mu \ | \ \text{\rm sgn}(\mu))].
	\end{equation}
\end{proof}
Now we are ready to find a bound on the mixing time for the overlapping cycles shuffle.
\begin{reptheorem}{mainoverlapping}
	Consider the overlapping cycles shuffle with $n$ cards and parameter $m$, where $\frac{m}{n} \in (\epsilon, 1 - \epsilon)$. The overlapping cycles shuffle has a mixing time which is at most
	\begin{equation*}
		\mathcal{A} \ell_{\max}^2 \log^3(n)
	\end{equation*}
	where $\mathcal{A}$ is a constant which depends on $\epsilon$.
\end{reptheorem}
\begin{proof}
	The previous Lemma \ref{penultimatelemma} implies that there exists $t_1 \in \{1,\dots,C\ell_{\max}^2\}$ such that
	\begin{align}
		\mathbb{E}\big[\text{\rm ENT}(\pi_{(t_1)}|\text{\rm sgn}(\pi_{(t_1)}))\big] \leq \left(1-\frac{D t_1}{\log^2(n) \ell_{\max}^2}\right)\mathbb{E}\big[\text{\rm ENT}((\text{id}|\text{\rm sgn}(\text{id}))\big].
	\end{align}
	Choose such a $t_1$, and then, by the same theorem, there exists $t_2 \in \{1,\dots,C\ell_{\max}^2 \}$ such that
	\begin{align}
		\mathbb{E}\big[\text{\rm ENT}(\pi_{(t_2)}\pi_{(t_1)}|\text{\rm sgn}(\pi_{(t_2)}\pi_{(t_1)}))\big] \leq \left(1-\frac{D t_2}{\log^2(n) \ell_{\max}^2}\right)\mathbb{E}\big[\text{\rm ENT}((\pi_{(t_1)}|\text{\rm sgn}(\pi_{(t_1)})\big].
	\end{align}
	Repeat this inductively, so choosing $t_k \in \{1,\dots,C\ell_{\max}^2 \}$ such that
	\begin{align}
		\mathbb{E}(\text{\rm ENT}(\pi_{(t_k)}\dots\pi_{(t_1)}|\text{\rm sgn}(\pi_{(t_k)}\dots\pi_{(t_1)}))) \leq \left(1-\frac{D t_k}{\log^2(n) \ell_{\max}^2}\right)\mathbb{E}(\text{\rm ENT}((\pi_{(t_{k-1})}\dots\pi_{(t_1)}|\text{\rm sgn}(\pi_{(t_{k-1})}\dots\pi_{(t_1)})),
	\end{align}
	and therefore
	\begin{align}
		\mathbb{E}(\text{\rm ENT}(\pi_{(t_k)}\dots\pi_{(t_1)}|\text{\rm sgn}(\pi_{(t_k)}\dots\pi_{(t_1)}))) \leq \prod\limits_{i=1}^k \left(1-\frac{D t_i}{\log^2(n) \ell_{\max}^2}\right)\mathbb{E}\big[\text{\rm ENT}((\text{id}|\text{\rm sgn}(\text{id}))\big].
	\end{align}
	Note that
	\begin{align}
		\prod\limits_{i=1}^k \left(1-\frac{D t_i}{\log^2(n) \ell_{\max}^2}\right) &\leq \exp \left( -\sum\limits_{i=1}^j \frac{D t_i}{\log^2(n) \ell_{\max}^2} \right) \mathbb{E} \big[\text{\rm ENT}(\text{id} \ | \ \text{\rm sgn} (\text{id}))\big] \\
		&= \exp \left( \frac{-D}{\log^2(n) \ell_{\max}^2} \sum\limits_{i=1}^j t_j  \right) \mathbb{E} \big[\text{\rm ENT}(\text{id} \ | \ \text{\rm sgn} (\text{id}))\big].
	\end{align}
	and 
	\begin{equation}
		\mathbb{E} \big[\text{\rm ENT}(\text{id} \ | \ \text{\rm sgn} (\text{id}))\big] = \text{\rm ENT}(\text{id} \ | \ \text{\rm sgn} (\text{id})) = \log\parens*{\frac{n!}{2}} \leq n\log(n).
	\end{equation}
	With this in mind, let
	\begin{equation}
		t = \frac{1}{D}\log^2(n)\ell_{\max}^2\big(\log(n)+\log(\log(n))- 2\log(\delta)\big) + C\ell_{\max}^2.
	\end{equation}
	for some choice of $\delta$. Since each $t_k$ is less than $C\ell_{\max}^2$, there exists some $\kappa$ such that
	\begin{equation}
		\frac{1}{D}\log^2(n)\ell_{\max}^2\big(\log(n)+\log(\log(n))-2\log(\delta)\big) < t_1 + \dots + t_\kappa < t.
	\end{equation}
	So,
	\begin{align*}
		\mathbb{E}\big[\text{\rm ENT}(\pi_t|\text{\rm sgn}(\pi_t))\big] &\leq \mathbb{E}\big[\text{\rm ENT}(\pi_{(t_\kappa)}\dots\pi_{(t_1)}|\text{\rm sgn}(\pi_{(t_\kappa)}\dots\pi_{(t_1)}))\big] \\
		&\leq \exp \left( \frac{-D}{\log^2(n) \ell_{\max}^2} \frac{1}{D}\log^2(n)\ell_{\max}^2\big(\log(n)+\log(\log(n)) - 2\log(\delta)\big) \right) n\log(n) \\
		&= \delta^2
	\end{align*}
	This is a bound on the conditional entropy given the sign of $\pi_t$. If $(1,\dots,m)$ and $(1,\dots,n)$ have the same sign, then this is the best we can hope for because we will always know if $\pi_t$ is even or odd by looking at if $t$ is even or odd. If $(1,\dots,m)$ and $(1,\dots,n)$ have different signs, then we can get a bound on the total entropy by doing a single additional step of the shuffle. Since the group element we multiply by in this additional step is equally likely to have an even or odd sign, we get
	\begin{equation}
		\mathbb{E}\big[\text{\rm ENT}(\pi_{(t+1)})\big] \leq \delta^2
	\end{equation}
	Plugging this into (\refeq{tventbound}) tells us that
	\begin{equation}
		\norm{\pi_{(t+1)}-\xi}_{\rm TV} \leq \delta
	\end{equation}
	and this gives us the mixing time. So the mixing time $t_{\rm mix}(\delta)$ has \begin{align}
		t_{\rm mix}(\delta) &\leq \frac{1}{D}\log^2(n)\ell_{\max}^2\big(\log(n)+\log(\log(n))- 2\log(\delta)\big) + C\ell_{\max}^2 \\
        &\leq \mathcal{D}\ell_{\max}^2(\log^3(n)-2\log^2(n)\log(\delta))
	\end{align}
    for a constant $\mathcal{D}$ that depends on $\epsilon$. Using the standard mixing time of $t_{\rm mix} = t_{\rm mix}(\frac{1}{4})$ we have
    \begin{equation}
        t_{\rm mix} \leq \mathcal{A}\ell_{\max}^2\log^3(n)
    \end{equation}
    for a constant $\mathcal{A}$ that depends on $\epsilon$.
\end{proof}
\noindent
Let $m = \floor*{\alpha n}$. As long as $\alpha$ is bounded away from $0$ and $1$, our result matches the mixing time shown by Angel, Peres, and Wilson for a single card after multiplying by the factor of a constant times $log^3(n)$. In the longest case, since we trivially have $\ell_{\max} \leq 2n$, we get that the mixing time is $O(n^2\log^3(n))$. This longest case is admitted if $\alpha$ is any rational, although the constant in front of $n^2\log^3(n)$ is smaller for rationals that have larger denominators in their reduced form. \\
\\
In the shortest case, since $\ell_{\max} \geq \frac{1}{2}n^\frac{3}{4}$, we have a mixing time of $O(n^\frac{3}{2}\log^3(n))$. This shortest case is admitted when $\alpha = \phi$ where $\phi = \frac{\sqrt{5}-1}{2}$ is the inverse golden ratio. This is because by Corollary \ref{goldenratiouniform} we have that $\ell_{\max}$ for $m = \floor{\phi n}$ is a constant times $n^\frac{3}{4}$. This follows from multiples of $\phi$ being about equally distributed across $(0,1) \mod 1$, which also means they are equally distributed across $(0,2n-\phi n + 1) \mod 2n-\phi n + 1$. If a more thorough justification is required, we have the following lemma.

\begin{lemma}
    Let $m = \floor{\phi n}$. Let $x \in \{1,\dots,2n-m\}$. Then,
    \begin{equation}
        \norm{x} \leq 6n^\frac{3}{4}.
    \end{equation}
\end{lemma}
This holds for all $x$ so we know $\ell_{\max}$ is on the order of $n^\frac{3}{4}$.
\begin{proof}
Let $x \in \{1,\dots,2n-\floor*{\phi n} +1\}$. As per Corollary \ref{goldenratiouniform} choose some $\beta \in \{1,\dots,\sqrt[4]{n}\}$ such that 
\begin{equation}
	\left|\frac{x}{2n}-(\beta\phi \mod 1)\right| \leq \frac{1}{2\phi^2} \cdot \frac{1}{\sqrt[4]{n}}.
\end{equation}
where $(z \mod \mathcal{M})$ refers to the number $\zeta \in (0,\mathcal{M}]$ such that $z \equiv \zeta \mod \mathcal{M}$. Then,
\begin{equation}
	\left|x-(2\beta \phi n \mod 2n)\right| \leq 2n \cdot \frac{1}{2\phi^2} \cdot \frac{1}{\sqrt[4]{n}} = \frac{1}{\phi^2} \cdot n^\frac{3}{4}.
\end{equation}
Note that $\beta \phi n \leq \beta (2n)$. So,
\begin{equation}
	(2\beta \phi n \mod 2n) = \kappa(2n) + 2\beta \phi n
\end{equation}
where $|\kappa| \leq \beta$. This means that
\begin{equation}
	(2\beta \phi n \mod 2n) = \kappa(2n - \phi n) + (2\beta+\kappa) \phi n.
\end{equation}
So,
\begin{align}
	\left|x-\kappa(2n - \phi n) - (2\beta+\kappa) \phi n\right| &\leq \frac{1}{\phi^2} \cdot n^\frac{3}{4} \\
	\left|x-\kappa(2n - \floor{\phi n} + 1) - (2\beta+\kappa) \phi n\right| &\leq \frac{1}{\phi^2} \cdot n^\frac{3}{4} + 2|\kappa| \\
	\left|x-\kappa(2n - \floor{\phi n} + 1) - (2\beta+\kappa) \floor{\phi n}\right| &\leq \frac{1}{\phi^2} \cdot n^\frac{3}{4} + 2|\beta| + 3|\kappa|.
\end{align}
Let $b = 2\beta + \kappa$. Then,
\begin{equation}
	\left|(x - b\floor{\phi n}) \mod 2n - \floor{\phi n} + 1\right| \leq \frac{1}{\phi^2} \cdot n^\frac{3}{4} + 2|\beta| + 3|\kappa|
\end{equation}
so 
\begin{equation}
	x \equiv b\floor{\phi n} + a \mod 2n - \floor{\phi n} + 1 \text{ where } |b| \leq 3\sqrt[4]{n} \text{ and } |a| \leq \frac{1}{\phi^2} \cdot n^\frac{3}{4} + 5\sqrt[4]{n}.
\end{equation}
This means,
\begin{align}
	\norm{x} &\leq 3\sqrt[4]{n} \cdot \sqrt{n} + \frac{1}{\phi^2} \cdot n^\frac{3}{4} + 5\sqrt[4]{n} \\
	&\leq \left(3 + \frac{1}{\phi^2}\right) n^\frac{3}{4} + 5\sqrt[4]{n} \\
	&\leq 6n^\frac{3}{4} \text{ for large enough } n.
\end{align}
\end{proof}

\appendix

\section{Appendix}

Here we have included some theorems cited throughout the paper whose uses are more generally applicable across probability.

\begin{theorem}\label{quasiuniform}
	Suppose $\mu$ is a probability measure on a finite probability space $\Omega$ such that for each $\omega \in \Omega$, we have $\mu(\omega) \geq \frac{1}{\mathcal{D} |\Omega |}$. Let $E$ be an event such that $\mu(E) \geq 1 - \frac{1}{8\mathcal{D}}$. Then there exists at least $\frac{3}{4}|\Omega|$ values $\alpha \in \Omega$ such that $\mu(\alpha | E) > \frac{1}{2\mathcal{D}|\Omega|}$.
\end{theorem}

\begin{proof}
	Let $S \subset \Omega$ be the set of values $\beta \in \Omega$ such that $\mu(\beta | E) \leq \frac{1}{2\mathcal{D}|\Omega|}$. Then,
	\begin{equation}
		\mu(S,E) = \mu(E) \sum\limits_{\beta \in S} \mu(
  \beta | E) \leq \frac{|S|}{2\mathcal{D}|\Omega|}.
	\end{equation}
	On the other hand,
	\begin{equation}
		\mu(S,E) = \mu(S) - \mu(S,E^C) \geq \mu(S) - \mu(E^C) \geq \frac{|S|}{\mathcal{D}|\Omega|} - \frac{1}{8\mathcal{D}}.
	\end{equation}
	So,
	\begin{align}
		\frac{|S|}{\mathcal{D}|\Omega|} - \frac{1}{8\mathcal{D}} &\leq \frac{|S|}{2\mathcal{D}|\Omega|}, \\
		\frac{|S|}{2\mathcal{D}|\Omega|} &\leq \frac{1}{8\mathcal{D}}, \\
		|S| &\leq \frac{1}{4}|\Omega|.
	\end{align}
	Since at most $\frac{1}{4}|\Omega|$ of $\beta \in \Omega$ have $\mu(\beta | E) \leq \frac{1}{2\mathcal{D}|\Omega|}$ we know that at least $\frac{3}{4}|\Omega|$ of $\alpha \in \Omega$ have $\mu(\alpha | E) > \frac{1}{2\mathcal{D}|\Omega|}$.
\end{proof}

We now provide a more general version of Theorem \ref{quasiuniform}. While we do not use the general version for any of our results, we provide it for the potential interest of the reader.

\begin{theorem}\label{quasiuniformgeneral}
	Let $\mu,\nu$ be probability measures on $\Omega$. Assume that there exist constants $a,b,\epsilon,\delta \in [0,1]$ such that
	\begin{itemize}
		\item $\mu(x) \geq \frac{a}{|\Omega|}$ at least $(1-\epsilon)|\Omega|$ of $x \in \Omega$
		\item $\nu(x) \leq \frac{b}{|\Omega|}$ at least $(1-\delta)|\Omega|$ of $x \in \Omega$
	\end{itemize}
	Then,
	\begin{equation}
		\norm{\mu-\nu}_{\rm TV} \geq (1-\epsilon)(1-\delta)(a-b)
	\end{equation}
\end{theorem}

\begin{proof}
	Let $S$ be the set of all $x \in \Omega$ such that $\mu(x) \geq \frac{a}{|\Omega|}$. Let $T$ be the set of all $x \in \Omega$ such that $\nu(x) \leq \frac{b}{|\Omega|}$. Then using one of the definitions of total variation distance we get,
	\begin{align}
		\norm*{\mu-\nu}_{\rm TV} &= \sum\limits_{x \in |\Omega|} (\mu(x)-\nu(x))^+ \\
		&\geq  \sum\limits_{x \in S \cap T} (\mu(x)-\nu(x))^+ \\
		&\geq \sum\limits_{x \in S \cap T} \frac{a-b}{|\Omega|} \\
		&= \frac{|S\cap T|}{|\Omega|}(a-b) \geq (1-\epsilon)(1-\delta)(a-b).
	\end{align}
\end{proof}

We can get Theorem $\ref{quasiuniform}$ from Theorem $\ref{quasiuniformgeneral}$ if we let $a = \frac{1}{\mathcal{D}}$ and $\epsilon = 0$. So $\mu$ is a probability measure where $\mu(x) \geq \frac{1}{\mathcal{D}|\Omega|}$ for all $x \in \Omega$. Then pick an event $E$ where $\mu(E) \geq 1 - \frac{1}{8\mathcal{D}}$ and let $\nu$ be the measure $\mu$ conditioned on $E$. Then $\norm*{\mu-\nu}_{\rm TV} \leq \frac{1}{8\mathcal{D}}$. Then if we set $b = \frac{1}{2\mathcal{D}}$ and solve for $\delta$ we will find that $\delta \geq \frac{3}{4}$ which means that no more than $\frac{1}{4}$ of all $x \in \Omega$ can have $\nu(x) = \mu(x | E) \leq \frac{1}{2\mathcal{D}}$. 

\begin{theorem}\label{invershoeffding}\cite{inversehoeffding} (Section 7.3, page 46) 
	Let $X$ be a binomial random variable with $n$ trials and probability $\frac{1}{2}$ chance of success. Let $k \geq 0$. Then,
	\begin{equation}
		\probp{X - \frac{n}{2} \geq k} \geq \frac{1}{15}\exp\left(\frac{-16k^2}{n}\right)
	\end{equation}
\end{theorem}

\begin{theorem}[Hoeffding's inequality]\label{hoeffdingbinomial}
	\cite{hoeffding} (Section 2, page 15) Let $X$ be a binomial random variable with $n$ trials and probability $p$ of success.
	\begin{equation*}
		\mathbb{P}\left( X-np \geq k \right) \leq \exp\left( - \frac{2k^2}{n}\right)
	\end{equation*}
\end{theorem}

\begin{corollary}
	Let $X_t$ be the simple random walk on the integers. Then
	\begin{equation*}
		\mathbb{P}(|X_t| \geq a\sqrt{n}) \leq 2\exp\left(-\frac{a^2}{2}\right)
	\end{equation*}
\end{corollary}

\begin{proof}
	After $t$ steps of the simple symmetric random walk, let $R$ be the amount of right steps. Then $X_t = 2R-t$, and $R$ is a Binomial random variable with $t$ trials and probability $\frac{1}{2}$ of success. So,
	\begin{align}
		\mathbb{P}(X_t \geq a\sqrt{n}) &= \mathbb{P}\left(R - \frac{t}{2} \geq \frac{a}{2}\sqrt{n}\right) \\
		&\leq \exp\left( -\frac{a^2}{2}\right) .
	\end{align}
	By symmetry, $-X_t$ has the same distribution. So by the union bound we get
	\begin{equation*}
		\mathbb{P}(|X_t| > a\sqrt{n}) \leq 2\exp\left(-\frac{a^2}{2}\right).
	\end{equation*}
\end{proof}

\begin{theorem}\label{hoeffdingmax}
	Let $X_t$ be the simple random walk on the integers. Let
	\begin{equation*}
		A_t = \max\{ |X_s| \ : \ s \leq t \}.
	\end{equation*}
	Then
	\begin{equation*}
		\mathbb{P}(A_t > a\sqrt{n}) \leq 4\exp\left( -\frac{a^2}{2}\right).
	\end{equation*}
\end{theorem}

\begin{proof}
	Let
	\begin{equation}
		M_t = \max\{ X_s \ : \ s \leq t \}
	\end{equation}
	Note that
	\begin{align}
		\mathbb{P}(M_t \geq k) &= \mathbb{P}(\text{ there exists } s \leq t : X_s = k) \\
		&= \mathbb{P}(\text{ there exists } s \leq t : X_s = k, \ X_t > k) \nonumber \\
		&\hspace{5mm} + \mathbb{P}(\text{ there exists } s \leq t : X_s = k, \ X_t < k) \nonumber \\
		&\hspace{5mm} + \mathbb{P}(\text{ there exists } s \leq t : X_s = k, \ X_t = k) .
	\end{align}
	By the Markov Property we see that
	\begin{equation}
		\mathbb{P}(\text{ there exists } s \leq t : X_s = k, \ X_t > k) = \mathbb{P}(\text{ there exists } s \leq t : X_s = k, \ X_t < k) .
	\end{equation}
	Also note that
	\begin{align}
		\mathbb{P}(\text{ there exists } s \leq t : X_s = k, \ X_t > k) = \mathbb{P}(X_t > k) \\
		\mathbb{P}(\text{ there exists } s \leq t : X_s = k, \ X_t = k) = \mathbb{P}(X_t = k)
	\end{align}
	so,
	\begin{align}
		\mathbb{P}(M_t \geq k) &= 2\mathbb{P}(X_t > k) + \mathbb{P}(X_t = k) \\
		&\leq 2\mathbb{P}(X_t \geq k) .
	\end{align}
	Setting $k = a\sqrt{n}$ and applying Hoeffding's inequality gives
	\begin{equation}
		\mathbb{P}(M_t \geq a\sqrt{n}) \leq 2\exp\left(-\frac{a^2}{2}\right) .
	\end{equation}
	Due to symmetry the minimum value of $X_s$ over the first $t$ steps has the same distribution as $-M_t$. So by the union bound
	\begin{equation}
		\mathbb{P}(A_t \geq a\sqrt{n}) \leq 4\exp\left(-\frac{a^2}{2}\right) .
	\end{equation}
\end{proof}

\begin{theorem}\label{fibbonacci}\cite{goldenratio}
	Let $\phi = \frac{\sqrt{5}-1}{2}$ be the inverse golden ratio. Fix any $N \in \mathbb{N}$. Now we define \\
    ${0 = a_0 < a_1 < \dots < a_N < 1}$ as the numbers where each $a_i = k\phi \mod 1$ for some natural number $k \leq N$. In other words, $a_0,\dots,a_N$ is a reordering of $0,\phi,2\phi,\dots,N\phi \mod 1$ from least to greatest. Then for any $i \in \{1,\dots,N\}$ we have
	\begin{align*}
		a_i - a_{i-1} &\in \{\phi^z,\phi^{z+1},\phi^{z+2}\} \\
		1 - a_N &\in \{\phi^z,\phi^{z+1},\phi^{z+2}\}
	\end{align*}
	where $z$ is defined as follows: $F_z$, the $z$th Fibbonacci number, is the largest Fibbonacci number less than or equal to $N$ (using the convention that $F_1 = F_2 = 1$). For a more numerical definition we can also write
	\begin{equation*}
		z = \max\left\{ x \in \mathbb{N} \text{ such that } \frac{\phi^{-x} - (-\phi)^{-x}}{\sqrt{5}} \leq N \right\}.
	\end{equation*}
\end{theorem}

\begin{corollary}\label{goldenratiouniform}
	Fix any $N \in \mathbb{N}$. Then for any $x \in [0,1]$ there exists $k \in \{0,1,\dots,N\}$ such that
	\begin{equation}
		|x-b_k| \leq \frac{1}{2\phi^2} \cdot \frac{1}{N+1}
	\end{equation}
	for some $b_k \in [0,1]$ with $b_k \equiv k\phi \mod 1$.
\end{corollary}

\begin{proof}
	Let $a_0,\dots,a_N$ be a reordering of $b_0,\dots,b_N$ in increasing order. I.e let $\{a_0,\dots,a_N\} = \{b_0,\dots,b_N\}$ where $a_0 < a_1 < \dots < a_N$. Let $g$ be the smallest gap between adjacent elements of $(a_0,\dots,a_N,1)$. Note that $g \leq \frac{1}{N+1}$ by the pigeon hold principle. By Theorem [\ref{fibbonacci}] we know that all gaps take the form $\phi^z,\phi^{z+1},\phi^{z+2}$ for a particular $z$. Using the fact that $\frac{1}{N+1} \geq g \geq \phi^{z+2}$ we see that
	\begin{equation}
		z \leq -\log_\phi(N+1) - 2
	\end{equation}
	Let $G$ be the largest gap between adjacent elements of $\{a_0,\dots,a_N,1\}$. Then $G \leq \phi^z$ so
	\begin{equation}
		G \leq \frac{1}{(N+1)\phi^2}.
	\end{equation}
	In the furthest case $x \in [0,1]$ is in the middle of a gap, in which case $x$ is at most distance $\frac{1}{2\phi^2} \cdot \frac{1}{N+1}$ from an element of $\{b_0,\dots,b_N\}$.
\end{proof}

\newpage

\printbibliography

@phdthesis{monte,
author={Senda,Alto E.},
year={2022},
title={A Mixing Time Bound for the Diaconis Shuffle},
journal={ProQuest Dissertations and Theses},
pages={117},
note={Copyright - Database copyright ProQuest LLC; ProQuest does not claim copyright in the individual underlying works; Last updated - 2023-03-25},
isbn={9798377625865},
language={English},
url={https://www.proquest.com/dissertations-theses/mixing-time-bound-diaconis-shuffle/docview/2787259704/se-2},
}

@misc{morris2008improved,
      title={Improved mixing time bounds for the Thorp shuffle and L-reversal chain}, 
      author={Ben Morris},
      year={2008},
      eprint={0802.0339},
      archivePrefix={arXiv},
      primaryClass={math.PR}
}

@article{overlappingcyclesoriginal,
 ISSN = {10505164},
 URL = {http://www.jstor.org/stable/25442782},
 author = {Johan Jonasson},
 journal = {The Annals of Applied Probability},
 number = {2},
 pages = {1034--1058},
 publisher = {Institute of Mathematical Statistics},
 title = {Biased Random-to-Top Shuffling},
 urldate = {2023-10-03},
 volume = {16},
 year = {2006}
}

@article{overlappingcyclesonecard,
 ISSN = {10505164},
 URL = {http://www.jstor.org/stable/25442664},
 author = {Omer Angel and Yuval Peres and David B. Wilson},
 journal = {The Annals of Applied Probability},
 number = {3},
 pages = {1215--1231},
 publisher = {Institute of Mathematical Statistics},
 title = {Card Shuffling and Diophantine Approximation},
 urldate = {2023-10-03},
 volume = {18},
 year = {2008}
}

@BOOK{pinsker,
  title     = "Information theory: Coding theorems for discrete memoryless
               systems",
  author    = "Csisz{\'a}r, Imre and K{\"o}rner, J{\'a}nos",
  publisher = "Cambridge University Press",
  month     =  jun,
  year      =  2011,
  address   = "Cambridge, England",
  language  = "en"
}

@article{hoeffding,
 ISSN = {01621459},
 URL = {http://www.jstor.org/stable/2282952},
 author = {Wassily Hoeffding},
 journal = {Journal of the American Statistical Association},
 number = {301},
 pages = {13--30},
 publisher = {[American Statistical Association, Taylor & Francis, Ltd.]},
 title = {Probability Inequalities for Sums of Bounded Random Variables},
 urldate = {2023-10-05},
 volume = {58},
 year = {1963}
}

@article{inversehoeffding,
  title={The probabilistic method},
  author={Matou{\v{s}}ek, Ji{\v{r}}{\'\i} and Vondr{\'a}k, Jan},
  journal={Lecture Notes, Department of Applied Mathematics, Charles University, Prague},
  year={2001}
}

@article{goldenratio,
author = {Świerczkowski, S.},
journal = {Fundamenta Mathematicae},
language = {eng},
number = {2},
pages = {187-189},
title = {On successive settings of an arc on the circumference of a circle},
url = {http://eudml.org/doc/213503},
volume = {46},
year = {1958},
}

@phdthesis{mis2,
	title        = {Rates of convergence of some random processes on finite groups},
	author       = {Hildebrand, Martin Victor },
	year         = 1990,
	school       = {Harvard University}
}

\end{document}